\documentclass[10pt]{amsart}

\usepackage{amsmath,amsfonts,latexsym,amssymb,amscd,
  graphicx,amsthm}
\usepackage{enumerate}
\usepackage{ulem}
\usepackage{cancel}
\usepackage{hyperref}
\usepackage[usenames]{color}

\allowdisplaybreaks[1]
\newcommand{\Pc}{\mathcal{P}}
\newcommand{\aseq}{{\stackrel{\mathrm{a.s.}}{=}}}

\newcommand{\asconv}{{\stackrel{\mathrm{a.s.}}{\longrightarrow}}}
\newcommand{\bpf}[1][Proof]{{\noindent {\sc #1: }}}
\newcommand{\epf}{{{\hfill $\Box$ \smallskip}}}
\newcommand{\ONE}{{\mathbf{1}}}

\newcommand{\N}{{\mathbb N}}
\newcommand{\Q}{{\mathbb Q}}
\newcommand{\Fc}{\mathcal{F}}
\newcommand{\Nc}{\mathcal{N}}
\newcommand{\Pp}{\mathsf{P}}
\newcommand{\Z}{{\mathbb Z}}
\newcommand{\iid}{i.i.d.\ }
\newcommand{\RZ}{\R\times\Z}
\newcommand{\ZR}{\Z\times\R}

\newcommand{\E}{\mathsf{E}}
\newcommand{\Leb}{\mathrm{Leb}}

\newcommand{\R}{{\mathbb R}}
\newcommand{\TT}{{\mathbb T}}

\newcommand{\HH}{{\mathbb H}}

\newcommand{\Co}{{\rm Co}}

\newcommand{\eps}{{\varepsilon}}

\newcommand{\visc}{\varkappa}

\newtheorem{theorem}{Theorem}[section]
\newtheorem{remark}{Remark}[section]
\newtheorem{lemma}{Lemma}[section]

\numberwithin{equation}{section}

\renewenvironment{proof}[1][Proof]{
{\noindent {\sc #1: }}
}{
{{\hfill $\Box$ \smallskip}}
}
\makeatletter
\let\orgdescriptionlabel\descriptionlabel
\renewcommand*{\descriptionlabel}[1]{%
  \let\orglabel\label
  \let\label\@gobble
  \phantomsection
  \edef\@currentlabel{#1}%
  \let\label\orglabel
  \orgdescriptionlabel{#1}%
}
\makeatother

\title[Directed polymers and Burgers equation]{Thermodynamic limit for directed polymers and stationary solutions of the Burgers equation}
\author{Yuri Bakhtin}
\address{Courant Institute of Mathematical Sciences\\ New York University \\ 251 Mercer St, New York, NY 10012 }
\author{Liying Li}
\email{bakhtin@cims.nyu.edu, liying@cims.nyu.edu}

\begin{document}
\begin{abstract}
 The first goal of this paper is to prove multiple asymptotic results for a time-discrete and space-continuous polymer model of a random walk in a random potential. These results include: existence of deterministic free energy density in the infinite volume limit for every fixed
 asymptotic slope; concentration inequalities for free energy implying a bound on its fluctuation exponent; straightness estimates implying a bound on the transversal fluctuation exponent. The culmination of this program is almost sure existence and uniqueness of polymer measures
 on one-sided infinite paths with given endpoint and slope, and interpretation of these infinite-volume Gibbs measures as thermodynamic limits.
 Moreover, we prove that marginals of polymer measures with the same slope and different endpoints are asymptotic to each other. 
 
 The second goal of the paper is to develop ergodic theory of the Burgers equation with positive viscosity and random kick forcing on the real line without any compactness assumptions. 
 Namely, we prove a One Force -- One Solution principle, using the infinite volume polymer measures to construct a family of stationary global solutions for this system, and proving that each of those solutions is a one-point pullback attractor on the initial
 conditions with the same spatial average. 
This provides a natural extension of the same program realized for the inviscid Burgers equation with the help of action minimizers that can be viewed as zero temperature
 limits of polymer measures.
\end{abstract}

\maketitle
\section{Introduction. The Burgers equation} The main topics in this paper are directed polymers and the Burgers equation. We begin with the Burgers equation which is one of the basic nonlinear evolutionary PDEs. It has various interpretations and applications. Although Burgers himself introduced it as a fluid dynamics model in an attempt to create a simplified picture of turbulence
(see \cite{Burgers:MR0001147}, \cite{burgers1973nonlinear}), 
the equation along with its modifications has been used to model diverse real world phenomena such as interface growth, traffic, or the distribution of matter in the Universe.  Various aspects of the Burgers turbulence are discussed 
in~\cite{Bec-Khanin:MR2318582}. The Burgers equation is also directly related to the KPZ equation that has been intensively studied lately. For an introduction to KPZ, see~\cite{Corwin:MR2930377} or~\cite{Quastel:MR3098078}. 

In one dimension, the Burgers equation is
\begin{equation}
\label{eq:Burgers}
 \partial_t u + u \partial_x u=\visc\partial_{xx} u + f.
\end{equation}
Under the fluid dynamics interpretation, the equation describes the evolution of a velocity profile $u$ of particles moving along the real line. 
The velocity of the particle located at time $t\in\R$ at point $x\in\R$ is denoted by $u(t,x)\in\R$. The left-hand side of~\eqref{eq:Burgers} represents the acceleration of the particle, 
so the right-hand side must contain all the forces acting on the particle. In this pressureless model, particles are subject to external forcing $f=f(t,x)$ and viscous friction forces 
represented by the term $\visc\partial_{xx}u(t,x)$, where $\visc\ge0$ is the
viscosity constant.

The following (viscous) Hamilton--Jacobi--Bellman (HJB) equation:
\begin{equation}
\label{eq:HJB}
 \partial_t U +  \frac{(\partial_x U)^2}{2}=\visc\partial_{xx} U + F, 
\end{equation}
is tightly connected to the Burgers equation.
Namely, if $U$ is a solution of~\eqref{eq:HJB}, then $u=\partial_x U$ solves~\eqref{eq:Burgers} with $f=\partial_x F$.

Under mild assumptions on $f$ and the initial conditions, the Cauchy problem for~\eqref{eq:Burgers} has smooth classical solutions for~$\visc>0$. In fact, the Hopf--Cole logarithmic transformation  reduces the problem
to the linear heat equation with multiplicative potential. This linear equation can be solved using the classical Feynman--Kac formula. The works~\cite{Hopf:MR0047234},~\cite{Cole:MR0042889} by Hopf and Cole were,
in fact, preceded by \cite{Florin:MR0029605}, see the note~\cite{Biryuk} for some historical comments.

Another way to represent solutions of viscous HJB equations is via stochastic control, see~\cite{Fleming-Soner:MR2179357} for systematic treatment of stochastic control.

If~$\visc=0$, then even smooth initial velocity profiles result in 
formation of discontinuities called shock waves. In this important case, one has to work with appropriately defined generalized solutions that can be obtained from the smooth solutions via a limiting $(\visc\to0)$ procedure. The solutions
can be characterized through a variational principle.

In this paper we continue the study of the case where the forcing $f=f_{\omega}(t,x)$ is a stationary random field, the argument $\omega$ being an element of a probability space $(\Omega,\Fc,\Pp)$.

Ergodic properties of the Burgers equation (and its generalizations) with random forcing have been studied in various settings in \cite{Sinai:MR1117645},  \cite{ekms:MR1779561}, 
\cite{Iturriaga:MR1952472},
\cite{Khanin-Hoang:MR1975784}, \cite{Suidan:MR2141893}, 
\cite{Gomes-Iturriaga-Khanin:MR2241814},  \cite{Dirr-Souganidis:MR2191776},  \cite{yb:MR2299503}, \cite{Bakhtin-quasicompact}, \cite{BCK:MR3110798}, \cite{Debussche-Vovelle:MR3418750}, \cite{kickb:bakhtin2016}.

The details of the systems considered in those papers vary. In one dimension, the equation may be considered on a circle,  a compact segment, or the entire real line. The kinds of 
time-stationary forcing that have been considered are: white noise, kick (applied at a discrete sequence of times), Poissonian (concentrated at configuration points of a Poisson point process). It is usual to assume the zero-range dependence in time (this ensures the Markov property of the solutions but is not absolutely necessary). 
Several kinds of spatial dependence structure and behavior at infinity have been considered. Also, the role of the external forcing can be played 
by random boundary conditions. Each setting calls for its own toolbox, but despite the variety of approaches and methods that have been employed, there are several general features that we proceed to describe informally.

In dissipative systems, statistically steady states often emerge due to some form of energy balance. 
For the Burgers equation, the energy is pumped into the system by the external forcing and 
dissipated either due to the Laplacian viscosity friction term (in the case of positive viscosity), or at shocks (in the zero viscosity case). So it is natural to expect existence of stationary regimes.

The question of ergodicity of the invariant measures for Markov processes including those associated with stochastic PDEs can often be approached by studying regularity properties of the transition probabilities
or controllability properties, see, e.g.,~\cite{H-M:MR2478676} for the stochastic Navier--Stokes case.

In contrast with turbulence described by the Navier--Stokes system, the dynamics generated by Burgers equation
is dominated by contraction.  So the random dynamical system approach turns out to be more fruitful and gives more detailed information about the pathwise behavior of the system. Namely, it is natural and beneficial to study 
the stochastic flow, i.e., the self-consistent (satisfying the so called {\it cocycle} property) family of random operators $\Phi^{st}_\omega$ constructing the solution
$\Phi^{st}_\omega u$ at time $t$ given the initial condition $u$ at time~$s$. 
For various settings, it turns out that one can describe ergodic components in the following terms:
for two velocity profiles $u^1$ and $u^2$ in the same ergodic component, $\Phi_\omega^{st}u^1$
and $\Phi_\omega^{st}u^2$ get close to each other as $t-s\to\infty$. In other words, the evolution over long time intervals depends mostly on the random forcing while the dependence on the initial condition
vanishes, which can be interpreted as loss of memory in the system. Moreover, with probability one, there is a limit 
\begin{equation}
\label{eq:general-pullback}
u_{t,\omega}=\Phi^{-\infty t}_\omega u^0=\lim_{s\to-\infty} \Phi^{s 	t}_\omega u^0,   
\end{equation}
and it does not depend on the initial condition $u^0$ within an ergodic component. The resulting family $(u_{t,\omega})$ of velocity profiles forms a global solution, i.e.,
\[
u_{t,\omega}= \Phi_\omega^{st}u_{s,\omega},\quad s<t,
\]
and is non-anticipating, i.e., $u_t$ depends only on the history of the forcing up to time~$t$. Moreover, for almost every $\omega$, $(u_{t,\omega})$ is a unique global solution with values in the given ergodic component.
This statement along with the pullback attraction property~\eqref{eq:general-pullback} is often called One Force --- One Solution Principle (1F1S).

The study of ergodic properties of solutions of~\eqref{eq:Burgers} with random forcing began in \cite{Sinai:MR1117645} where the evolution was considered on the circle (or one-dimensional torus) $\TT^1=\R^1/ \Z^1$ (i.e., all the functions involved
were assumed or required to be space-periodic), the forcing was assumed to be white in time and smooth in the space variable, and a mixing statement showing loss of memory in the system was proved. The key consideration in this paper
is the view at the iterative application of the Feynman--Kac formula as the product of positive operators. 

In~\cite{Kifer:MR1452549} the connection with the directed polymers in random environments was noticed and used for the first time. With the help
of the Hopf--Cole transform and Feynman--Kac formula it was shown that for the high-dimensional version of~\eqref{eq:Burgers} and sufficiently small forcing (this situation is known as {\it weak disorder} in 
the studies of directed polymers in random environments), certain series in the spirit of perturbation theory converge and can be used to define global attracting solutions of the Burgers equation.     

In~\cite{ekms:MR1779561}, the zero viscosity case on the circle was considered. Solutions of the Burgers equation with zero viscosity admit a variational Hamilton -- Jacobi -- Bellman -- Hopf -- Lax -- Oleinik representation. 
The minimizing paths in the variational principle can be identified with particle trajectories, and the analysis of solutions over long time intervals involves the study of asymptotic properties of those minimizers.
Since the mean velocity is preserved by the Burgers system, all velocity profiles in one ergodic component have the same mean.  One of the main results of~\cite{ekms:MR1779561} is that all functions with the same mean
form one ergodic component, i.e., there is a unique invariant measure for the corresponding Markov dynamics on this set. Moreover, for each mean velocity, 1F1S holds on the associated ergodic component. The global solution is defined
by a family of one-sided infinite action minimizers stretching into the infinite past. Also, {\it hyperbolicity} holds, i.e., all these minimizers are exponentially asymptotic to each other.

In \cite{Iturriaga:MR1952472}, this program was repeated for the multi-dimensional version of the inviscid Burgers equation on the torus $\TT^d=\R^d/\Z^d$, $d\in\N$, and in~\cite{Gomes-Iturriaga-Khanin:MR2241814}, 
it was extended to the positive viscosity case. Unlike \cite{Sinai:MR1117645}, the approach of~\cite{Gomes-Iturriaga-Khanin:MR2241814} was based on stochastic control. In fact, for a fixed mean velocity,
a unique global solution is constructed using optimal control of diffusions on semi-infinite time intervals stretching to the infinite past. The variational character of the stochastic control approach allowed to 
show that as $\visc\to 0$, the optimally controlled
diffusions converge to the one-sided action minimizers. This also allowed to deduce the convergence of invariant distributions as $\visc\to 0$.

In~\cite{yb:MR2299503}, 1F1S was established for the Burgers equation with random boundary conditions. Given an appropriate notion of generalized solutions and the associated variational characterization, the argument is very simple.
It turns out that it takes a finite random time to erase all the memory about the initial condition, so the system exhibits an extreme form of contraction. 

In all the results discussed above (see also \cite{Dirr-Souganidis:MR2191776} and \cite{Debussche-Vovelle:MR3418750} that do not use variational or stochastic representations and use PDE tools instead),
the space was assumed to be compact, being a torus or a segment, except \cite{Kifer:MR1452549}. 
Extending those results to noncompact situations turned out to be a nontrivial task. 
Quasi-compact settings where the system is considered on the entire real line but the forcing is mostly concentrated in a compact part, were studied in~\cite{Khanin-Hoang:MR1975784}, \cite{Suidan:MR2141893},  
\cite{Bakhtin-quasicompact}. 

However, truly noncompact situations with space-time homogeneous random forcing in one dimension for positive or zero viscosity presented  serious difficulties. In the noncompact case, there is much less rigidity
in the behavior of optimal paths or diffusions used in the representation of solutions, and they are much harder to control. Also, the approach of~\cite{Kifer:MR1452549} is useful only in the weak disorder case
and fails in dimension~$1$.

In the zero viscosity case, the ergodic theory of the Burgers equation on the real line without compactness or periodicity assumptions was constructed in~\cite{BCK:MR3110798} for forcing given by space-time Poisson point process,
and in~\cite{kickb:bakhtin2016} for kick forcing. Similarly to the compact case,  the ergodic components are essentially formed  by velocity profiles with common mean, but 
establishing 1F1S on each ergodic component required using methods originating from studies of long geodesics in the last-passage percolation theory. In the Poissonian forcing case, due to the discrete character of the forcing, 
all the one-sided minimizers giving rise to the global solution
coalesce, strengthening the hyperbolicity property for the spatially smooth periodic forcing case. However, the behavior of minimizers in the kick forcing case is more complicated. Although they are expected to be asymptotic to each other, 
only a much weaker {\it liminf} substitute of hyperbolicity was proved in~\cite{kickb:bakhtin2016}.

The main goal of the present paper is to extend this program to the Burgers equation with positive viscosity, space-time homogeneous random kick forcing, and without any compactness assumptions. 
Our main results at a formal level look similar to the ergodic results for the inviscid case with kick forcing: ergodic components are formed by velocity
profiles with common mean value, and on each ergodic component 1F1S holds. This extension to the viscous case is natural to expect because positive viscosity means stronger dissipation and contraction.
However, a very important feature of our work is that in order to analyze the Burgers equation we rely on the Feynman--Kac formula and the associated directed polymer model.  

For that directed polymer model in $1+1$ dimensional random environment, we prove a whole series of new results that are of independent interest.  
These results include: the existence of finite nonrandom quenched free energy density for every fixed asymptotic slope and its quadratic dependence on the slope; a concentration inequality for free energy and a bound on the fluctuation exponent; straightness estimates and a bound on the transversal fluctuation exponent.
We use those results to construct thermodynamic limits satisfying the DLR conditions and prove their uniqueness for every fixed
 asymptotic slope. We also prove results on limiting ratios of partition functions playing the role of Busemann functions and show that a version of hyperbolicity property holds true. Namely, we show that the marginals of any two infinite volume polymer measures 
with the same slope are asymptotic to each other in the total variation distance.

We state and discuss our main results on directed polymers in section~\ref{sec:polymers} that also has an introductory character. Before that, we explain the Burgers equation setting in Section~\ref{sec:setting} and state the main 
1F1S results in Section~\ref{sec:main_results}. 

\bigskip

{\bf Acknowledgments.} The authors are grateful to Konstantin Khanin for stimulating discussions. They also thank Firas Rassoul-Agha for useful bibliographical suggestions. Yuri Bakhtin gratefully acknowledges partial support by NSF through grant DMS-1460595.

\section{The setting. Kick forcing}
\label{sec:setting}
The main model that we study in this paper is the Burgers equation with kick forcing of the following form: 
\begin{equation*}
f(t,x)=\sum_{n\in\Z} f_{n}(x)\delta_{n}(t). 
\end{equation*}
This means that the additive forcing is applied only at integer times. At each time $n\in\Z$, the entire velocity profile receives an instantaneous macroscopic increment equal to $f_n$:
\begin{equation}
\label{eq:effect-of-one-kick}
u(n,x)=u(n-0,x)+f_{n}(x),\quad x\in \R, 
\end{equation}
and between the integer times
the velocity field evolves according to the unforced viscous Burgers equation
\begin{equation}
\label{eq:unforced-Burgers}
\partial_t u + u \partial_x u=\visc\partial_{xx} u. 
\end{equation}
We assume that the potential $F=F_{n,\omega}(x)$ of the forcing 
\[
f_n(x)=f_{n,\omega}(x)=\partial_x F_{n,\omega}(x), \quad n\in\Z,\ x\in\R,\ \omega\in \Omega, 
\]
is a stationary random field (i.e., its distribution is invariant under space-time shifts)
defined on some probability space $(\Omega,\Fc,\Pp)$. We will describe all conditions that we impose on $F$ in the end of this section. At this point we need only the following consequence of those
conditions: for every $\omega\in\Omega$,
and every $n\in\Z$, the function $F_{n,\omega}(\cdot)$ is measurable and satisfies  
\begin{equation}
\label{eq:forcing-averages-to-0}
\lim_{|x|\to\infty} \frac{F_{n,\omega}(x)}{|x|}=0.
\end{equation}

\bigskip

Let us now define the Burgers dynamics with kick forcing.  First, we remind that if $U$ is a $C^3$ function solving the unforced ($F\equiv 0$) equation~\eqref{eq:HJB}, then, differentiating~\eqref{eq:HJB}, we obtain
that $u=U_x$ solves~\eqref{eq:unforced-Burgers}. Furthermore, we can introduce a new variable $\varphi$ via
the Hopf--Cole transformation:
\begin{equation*}
U(t,x)=- 2 \visc \ln \varphi(t,x), 
\end{equation*}
or, equivalently,
\begin{equation*}
 \varphi(t,x)=e^{- \frac{U(t,x)}{2\visc}},
\end{equation*}
and directly check that if $\varphi$ is a $C^3$ positive solution of the heat equation
\begin{equation}
\label{eq:HE}
\partial_t \varphi =\visc \partial_{xx} \varphi, 
\end{equation}
then the function $u$ obtained from $\varphi$ by
\begin{equation}
\label{eq:Hopf-Cole-1}
 u(t,x)=\partial_x U(t,x)=- 2\visc \partial_x \ln \varphi(t,x) = - 2\visc\frac{\partial_x \varphi(t,x)}{\varphi(t,x)} 
\end{equation}
is a solution of~\eqref{eq:unforced-Burgers}. The solution for the Cauchy problem for~\eqref{eq:HE} under very broad assumptions (measurability and moderate growth of the initial condition) is given by
\begin{equation}
 \label{eq:heat-equation-solution}
 \varphi(t,x)=\int_\R \varphi(s,y) g_{2\visc(t-s) }(x-y)dy,\quad x\in\R,\ t>s,
\end{equation}
where
\begin{equation*}
g_{D}(x)=\frac{1}{\sqrt{2\pi D}}e^{-\frac{x^2}{2D}}, \quad x\in\R,
\end{equation*}
is the centered Gaussian density corresponding to variance $D>0$. So, the evolution between the kicks is governed by \eqref{eq:Hopf-Cole-1} and \eqref{eq:heat-equation-solution}.
To see what happens at the kick, we rewrite~\eqref{eq:effect-of-one-kick} at the level of potentials as
\begin{equation*}
U(n,x)=U(n-0,x)+F_{n}(x),\quad x\in \R, 
\end{equation*}
and in the multiplicative form as 
\begin{equation*}
\varphi(n,x)=\varphi(n-0,x)e^{-\frac{F_{n}(x)}{2\visc}}.
\end{equation*}
Combining the instantaneous kick at a time $n$ and unforced evolution on the time interval $[n,n+1)$ leads to the following random linear operator: 
\begin{align*}
 \Xi^{n,n+1}_\omega \varphi (y)&=\int_{\R} g_{2\visc}(y-x)e^{-\frac{F_{n,\omega}(x)}{2\visc}}\varphi(x)dx,\quad y\in\R.
\end{align*}
Iterating this, we obtain that the (linear) solution operator over an interval  $[m,n)$ (producing the solution
just before the next instantaneous kick is applied at time~$n$) is given by
\begin{equation}
\label{eq:F-K}
 \Xi^{m,n}_\omega \varphi (y)=\int_{\R} Z^{m,n}_{x,y,\omega}\varphi(x)dx,\quad y\in\R,
\end{equation}
where
\begin{equation}
\label{eq:Z}
Z^{m,n}_{x,y,\omega}=\int_{\R}\dots\int_\R \prod_{k=m}^{n-1} \left[g_{2\visc}(x_{k+1}-x_k)e^{-\frac{F_k(x_k)}{2\visc}}\right] \delta_x(dx_m) dx_{m+1}\ldots dx_{n-1}\delta_y(dx_n).
\end{equation}
We can interpret~\eqref{eq:F-K} as a time-discrete version of the Feynman--Kac formula for
 the following kicked parabolic  model
\begin{equation*}
\partial_t \varphi(t,x)=\visc \partial_{xx} \varphi(t,x)+ \varphi(t,x)\cdot\sum_{n\in\Z} (e^{-\frac{F_n(x)}{2\visc}}-1)\delta_n(t).  
\end{equation*}
Invoking the change of variables~\eqref{eq:Hopf-Cole-1}, we can define evolution on potentials by
\[
\Phi_\omega^{m,n} U = -2\visc\ln \Xi_\omega^{m,n} e^{- \frac{U}{2\visc}}. 
\]

Let us now be more precise about the domains of the operators we have introduced. The $C^3$~requirement that we started with
is superfluous since our Feynman--Kac formula makes sense and produces generalized solutions under much weaker
requirements.
The space of velocity potentials that we will consider will be~$\HH$, the space of all locally
Lipschitz functions $W:\R\to\R$
 satisfying
\begin{align*}
 \liminf_{x\to\pm\infty}\frac{W(x)}{|x|}&>-\infty.
\end{align*}
We will also need a family of spaces
\[
\HH(v_-,v_+)=\left\{W\in\HH:\ \lim_{x\to \pm\infty} \frac{W(x)}{x}=v_\pm \right\},\quad v_-,v_+\in\R.
\] 

\begin{lemma} \label{lem:invariant_spaces} For any $\omega\in\Omega$,
for any $l,n,m\in\Z$ with $l<n<m$ and $W\in\HH$,
\begin{enumerate}
 \item\label{it:inv-of-HH} $\Phi^{n,m}_\omega W$ is well-defined and belongs to $\HH$;
 \item\label{it:inv-of-HH-vv} if $W\in\HH(v_-,v_+)$ for some  $v_-,v_+$, then $\Phi^{n,m}_\omega W\in\HH(v_-,v_+)$;
 \item\label{it:cocycle} $\Phi^{l,m}_\omega W=\Phi^{n,m}_\omega\Phi^{l,n}_\omega W$.  
\end{enumerate}
\end{lemma}
We give a proof of this lemma in Section~\ref{sec:aux}. Due to the last statement of the lemma (called
the cocycle property) we can view the Burgers evolution as
a random dynamical system (see, e.g.,~\cite[Section~1.1]{Arnold:MR1723992}).

Potentials are naturally defined up to an additive constant. It is thus convenient to
work with $\hat \HH$, the space of equivalence classes of potentials from $\HH$. 
We can also introduce spaces $\hat\HH(v_-,v_+)$ 
as classes of potentials in $\HH(v_-,v_+)$ coinciding up to an additive constant.
The cocycle $\Phi$ can be projected on $\hat\HH$ in a natural way. We denote the resulting cocycle on $\hat\HH$ by~$\hat\Phi$.

We can also introduce the Burgers dynamics on the space $\HH'$ of velocities $w$ (actually, classes of equivalence of functions since we do not distinguish two functions coinciding almost everywhere) such that
for some function $W\in\HH$ and Lebesgue almost every~$x$,  $w(x)=W'(x)=\partial_x W(x)$.
For all $v_-,v_+\in\R$,  $\HH'(v_-,v_+)$ is the space of velocity profile with well-defined one-sided averages $v_-$ and $v_+$, it consists of functions $w$ such that
the potential~$W$ defined by $W(x)=\int_0^x w(y)dy$ belongs to $\HH(v_-,v_+)$.

We will
write $ w_1=\Psi^{n_0,n_1}_\omega w_0$ if $w_0=W'_0$, $w_1=W'_1$, and $W_1=\Phi^{n_0,n_1}_\omega
W_0$  for some $W_0,W_1\in\HH$. Of course, the maps belonging to the Burgers cocycle $(\Psi^{n_0,n_1})$ map $\HH'$ into itself. The spaces $\HH'(v_-,v_+)$  are also invariant.

\smallskip

Our choice of solution spaces follows that of~\cite{kickb:bakhtin2016} where discontinuous
solutions in $\HH'$ naturally arose in the zero viscosity setting. In our setting, in fact, for any $w\in \HH'$, $\omega\in\Omega$, 
and $m,n\in\Z$ satisfying $m<n$, $\Psi^{m,n}w$ can be viewed as  a function in $C^\infty(\R)$ as straightforward
differentiation in~\eqref{eq:F-K} shows.

In fact, the Burgers dynamics has the following additional regularizing property:
\begin{lemma}
\label{lem:monotonicity_of_x-ux}
For any $w\in \HH'$, $\omega\in\Omega$, and $m,n\in\Z$ satisfying $m<n$, the function
$x\mapsto x-\Psi^{m,n}w(x)$ is nondecreasing. 
\end{lemma}
We prove this lemma in Section \ref{sec:aux} using the results on monotonicity that we introduce in
Sections~\ref{sec:monotonicity} and~\ref{sec:uniqueness-of-global-solutions}.

\subsection{The requirements on the random forcing}
\label{sec:assumption-on-the-potential}
 For simplicity, we will work on the canonical probability space $(\Omega_0,\Fc_0,\Pp_0)$ 
of realizations of the potential, although other more general settings are also possible. 
We assume that $\Omega_0$ is the space of continuous functions $F:\RZ\to\R$ equipped with $\Fc_0$, the completion of the Borel $\sigma$-algebra with respect to local uniform topology, and $\Pp_0$ is
a probability measure preserved by the group of shifts $(\theta^{n,x})_{(n,x)\in\ZR}$  defined by
\[
 (\theta^{n,x}F)_m(y) = F_{n+m}(x+y),\quad (n,x),(m,y)\in\ZR, 
\]
i.e., $(F_n(x))_{(n,x)\in\ZR}$ is a space-time stationary process. In this framework, $F=F_\omega=\omega$, and we will use all these notations intermittently.

In addition to this, we introduce the following requirements: 

\begin{description}
\item[{(A1)}\label{item:stationary-in-space}] The flow $(\theta^{0,x})_{x\in\R}$ is ergodic. In particular,
for every $n\in\Z$,  $F_n(\cdot)$ is ergodic with respect to the spatial shifts.
\item[{(A2)}\label{item:indep-in-time}] The sequence of processes $\big( F_n(\cdot)\big)_{n \in \Z}$
  is \iid
\item[(A3)\label{item:smooth}] With probability 1, for all $n\in\Z$, $F_n(\cdot)\in C^1(\R)$. 
\item[(A4)\label{item:exponential-moment-with-viscosity}] For all  $(n,x) \in \Z\times\R$,
  \begin{equation*}
    \lambda:=\E e^{- \frac{1}{2 \visc} F_n(x)} < \infty.
  \end{equation*}
\item[(A5)\label{item:exponential-moment-for-maximum}] 
There is $\eta>0$ such that for all $(n,j) \in \Z  \times \Z$,
  \begin{equation*}
  \E e^{\eta F^{*}_{n}(j)} < \infty,
\end{equation*}
where $F^{*}_{n}(j) = \sup \{ |F_{n}(x)| : x \in [j,j+1] \}$.
\end{description}

We will use these standing assumptions throughout the paper. However, many of our results will hold true if one removes~\ref{item:smooth} 
because~\ref{item:exponential-moment-for-maximum} guarantees that $F$ is locally bounded which is sufficient for most of our results. Of course, differentiability of $F$ guarantees that $f=\partial_x F$ in the Burgers equation is defined as a function, but even this is not necessary for some of our claims on the Burgers equation.

 A~sufficient condition on distributional properties of~$F$ at any fixed time, say, time 0, for existence of an appropriate probability space 
satisfying~\ref{item:stationary-in-space} and \ref{item:indep-in-time} is mixing of~$F_0$ with respect to spatial shifts.
This (along the other requirements from the list above) holds, for example, for Gaussian processes with decaying correlations and processes with
finite dependence range. Also, processes obtained from Poissonian noise (or any other space-time ergodic process) via spatial smoothening
are compatible with  probability spaces satisfying~\ref{item:stationary-in-space}--\ref{item:indep-in-time}. So, the conditions that we impose define a very broad class of processes. We note that the entire inviscid Burgers equation program developed in~\cite{kickb:bakhtin2016} using shot-noise potential, also holds for this broad class of potentials.

Stationarity and~\ref{item:exponential-moment-for-maximum} imply that~\eqref{eq:forcing-averages-to-0} holds
with probability~$1$ on $\Omega_0$. It will be convenient in this paper to work on a modified probability space
\begin{equation}
\label{eq:def-of-Omega}
\Omega=\left\{F\in\Omega_0: \lim_{|x|\to\infty}\frac{F_{n}(x)}{|x|}=0,\quad n\in\Z\right\}\in\Fc_0.
\end{equation}
of probability $1$ instead of $\Omega_0$. As we will see, on this set, the Burgers evolution possesses 
some nice properties. Moreover, $\Omega$ is invariant 
under space-time shifts~$\theta^{n,x}$ and under Galilean  space-time shear transformations~$L^{v},$ $v\in\R$,
defined by 
\begin{equation}
\label{eq:def-of-shear}
(L^v F)_n (x)=F_n(x+vn),\quad (n,x)\in\ZR. 
\end{equation}

We denote the restrictions of $\Fc_0$ and $\Pp_0$ onto $\Omega$ by $\Fc$ and $\Pp$.
From
now on we work with the probability space $(\Omega,\Fc,\Pp)$. Under this modification, all the distributional properties are preserved.

Having introduced the shifts $\theta^{n,x}$, we can also rewrite the cocycle property as
 \[
 \Phi^{n+m}_{\omega}W=\Phi^{m}_{\theta^{n}\omega} \Phi^{n}_\omega W,\quad n,m\in\N,\quad\omega\in\Omega,
\]
where $\theta^{n}=\theta^{n,0}$ and  $\Phi^n_\omega=\Phi^{0,n}_\omega$.
The cocycle property of $\Psi$ and $\Xi$ can also be expressed similarly.

\section{Main results}\label{sec:main_results}
Our main results for the positive viscosity Burgers equation are parallel to those of~\cite{kickb:bakhtin2016}
for the inviscid case. 

We say that $u(n,x)=u_\omega(n,x)$, $(n,x)\in\Z\times\R$ is a global solution for the
cocycle~$\Psi$ if there is a set $\Omega'\in\Fc$ with $\Pp(\Omega')=1$ such that for all
$\omega\in\Omega'$, all $m$ and $n$ with $m<n$, we have $\Psi^{m,n}_\omega u_\omega(m,\cdot)= u_\omega(n,\cdot)$.

A function $u_\omega(x)$, $\omega\in\Omega$, $x\in\R$ is called skew-invariant
if there is a set $\Omega'\in\Fc$ with $\Pp(\Omega')=1$ such that for any $n\in\Z$, $\theta^n\Omega'=\Omega'$, and for any
$n\in\N$ and $\omega\in\Omega'$,
$\Psi^n_\omega u_\omega =u_{\theta^n\omega}$. 

If $u_\omega(x)$ is a skew-invariant function, then $u_\omega(n,x)=u_{\theta^n\omega}(x)$ is a stationary global solution. One can naturally view the potentials of $u_\omega(x)$ and
$u_\omega(n,x)$ as a  skew-invariant
function and global solution for the cocycle $\hat\Phi$.

To state our first result, a description of stationary global solutions, we need more notation.  For a subset $A$ of $\Z\times\R$, we denote by $\Fc_A$ the $\sigma$-sub-algebra
of $\Fc$  generated by random variables $(F_{n}(x))_{(n,x)\in A}$.

\begin{theorem}\label{thm:global_solutions}
For every $v\in\R$
there is a unique skew-invariant function
$u_v:\Omega\to\HH'$ such that for almost every $\omega\in\Omega$, $u_{v,\omega}\in
\HH'(v,v)$. The process $u_{v,\omega}(n,\cdot)=u_{v,\theta^n\omega}(\cdot)$ is a unique stationary global solution in $\HH'(v,v)$.

The potential $U_{v,\omega}$ defined by $U_{v,\omega}(x)=\int^x u_{v,\omega}(y)dy$ is a unique
skew-invariant function for $\hat\Phi$ in~$\hat\HH(v,v)$. It defines a unique stationary
global solution $U_{v,\omega}(n,\cdot)=U_{v,\theta^n\omega}(\cdot)$ for $\hat\Phi$ in $\hat\HH(v,v)$. The skew-invariant functions
$U_{v,\omega}$ and $u_{v,\omega}$ are measurable  w.r.t.\ $\Fc|_{(-\N)\times\R}$, i.e., they depend only on the
history of the forcing.   The spatial random process $(u_{v,\omega}(x))_{x\in\R}$ is stationary
and ergodic with respect to space shifts.
\end{theorem}
\begin{remark}\rm All uniqueness statements in this theorem are understood {\it up to zero-measure modifications}. We say that a process $u$ is 
a unique (up to a zero-measure modification) process with certain properties if for every process $\tilde u$ defined on the same probability space and possessing these properties,
$u$ and $\tilde u$ coincide with probability~$1$.  
\end{remark}

This theorem can be interpreted as a 1F1S Principle: for
any velocity value $v$, the
solution at time $0$ with mean velocity $v$ is uniquely determined by
the history of the forcing: $u_{v,\omega}\stackrel{\rm a.s.}{=}\chi_v(F|_{(-\N)\times\R })$ for some
deterministic functional $\chi_v$ of the forcing  in the past, i.e., in $(-\N)\times\R$. We actually describe
$\chi_v$ in the proof, it is constructed via infinite volume polymer measures on one-sided paths. Since the forcing is stationary in time, we obtain that $u_{v,\theta^n\omega}$ is a stationary
process in $n$, and the distribution of $u_{v,\omega}$ is an invariant distribution for the corresponding Markov
semi-group, concentrated on $\HH'(v,v)$.
\medskip

The next result shows that each of the global solutions constructed in Theorem~\ref{thm:global_solutions} plays the
role of a one-point pullback attractor.
 To describe the domains of attraction  we need to introduce several assumptions on initial potentials
$W\in\HH$. Namely, we will assume that there is $v\in\R$ such that $W$ and $v$ satisfy one of the following sets of
conditions:
\begin{align}
v&=0,\notag\\
\liminf_{x\to+\infty} \frac{W(x)}{x}&\ge 0,  \label{eq:no_flux_from_infinity}\\
\limsup_{x\to-\infty} \frac{W(x)}{x}&\le 0,\notag
\end{align}
or
\begin{align}
v&> 0,\notag \\
\lim_{x\to-\infty} \frac{W(x)}{x}&= v,\label{eq:flux_from_the_left_wins}\\
\liminf_{x\to+\infty} \frac{W(x)}{x}&> -v,\notag
\end{align}
or
\begin{align}
v&< 0,\notag\\
\lim_{x\to+\infty} \frac{W(x)}{x}&= v,\label{eq:flux_from_the_right_wins}\\
\limsup_{x\to-\infty} \frac{W(x)}{x}&< -v.\notag
\end{align}

Condition~\eqref{eq:no_flux_from_infinity} means that there is no macroscopic flux of particles from infinity toward
the origin for the initial velocity profile $W'$. In particular, any $W\in\HH(0,0)$ or any $W\in\HH(v_-,v_+)$ with
$v_-\le 0$ and $v_+\ge 0$ satisfies~\eqref{eq:no_flux_from_infinity}. If, additionally, $v_+>0$ and $v_-<0$, then it is natural to call this situation a rarefaction fan.
We will see that in this case the long-term behavior is described
by the global solution $u_0$ with mean velocity $v=0$.

Condition~\eqref{eq:flux_from_the_left_wins} means that the initial velocity profile $W'$ creates an influx of
particles from $-\infty$ with effective velocity $v\ge 0$, and the influence of the particles at $+\infty$ is not as
strong.
In particular, any $W\in\HH(v,v_+)$ with
$v\ge 0$ and $v_+> -v$ (e.g., $v_+=v$) satisfies \eqref{eq:flux_from_the_left_wins}. We will see that in this case the
long-term
behavior is described by the global solution $u_v$.

Condition~\eqref{eq:flux_from_the_right_wins} describes a situation symmetric to~\eqref{eq:flux_from_the_left_wins},
where
in the long run the system is dominated by the flux  of particles from $+\infty$.

The following precise statement supplements Theorem~\ref{thm:global_solutions} and describes the basins of attraction of
the global solutions $u_v$ in terms of
conditions~\eqref{eq:no_flux_from_infinity}--\eqref{eq:flux_from_the_right_wins}.

\begin{theorem}\label{thm:pullback_attraction} For every $v\in\R$,
there is a set $\hat \Omega\in\Fc$ with $\Pp(\hat \Omega)=1$ such that if
$\omega\in\hat \Omega$, $W\in \HH$, and one of
conditions~\eqref{eq:no_flux_from_infinity},\eqref{eq:flux_from_the_left_wins},\eqref{eq:flux_from_the_right_wins}
holds,
then $w=W'$ belongs to the domain of pullback attraction of $u_v$:
for any $n\in\R$ and any  $x\in\R$,
\[
\lim_{m\to-\infty} \Psi^{m,n}_\omega w(x) = u_{v,\omega}(n,x),
\]
and the convergence is uniform on compact sets.
\end{theorem}

The last statement of the theorem implies that for every $v\in\R$, the invariant
measure on $\HH'(v,v)$ described after Theorem~\ref{thm:global_solutions} is unique and for any initial condition
$w=W'\in\HH'$ satisfying one of
conditions~\eqref{eq:no_flux_from_infinity},\eqref{eq:flux_from_the_left_wins}, and~\eqref{eq:flux_from_the_right_wins},
the distribution of the random velocity profile at time $n$ weakly converges to the unique stationary
distribution on $\HH'(v,v)$ as $n\to\infty$, in the local uniform topology. However, our approach does not produce any convergence rate estimates.

We also note that, due to Lemma~\ref{lem:monotonicity_of_x-ux}, proving uniform convergence in this theorem amounts to proving
pointwise convergence.

\medskip

We prove Theorems~\ref{thm:global_solutions} and~\ref{thm:pullback_attraction} in 
Section~\ref{sec:global-solutions-and-ratios-of-partition-functions} using thermodynamic limits for directed polymers that
we introduce next.

\section{Directed polymers}\label{sec:polymers}
Our analysis of the Burgers equation is based on several new results concerning directed polymers and their thermodynamic limits. 
We provide the necessary background and state those results in this section.

Directed polymers in random environment are a class of random media models
given by random
Boltzmann--Gibbs distributions on paths with (i) free measure describing classical random walks and 
(ii) the energy function given by the potential accumulated from the random environment by the random walk. 

In the Burgers equation context, the directed polymers emerge naturally through the Feynman--Kac formula~\eqref{eq:F-K}. It can be understood as integration over the space of paths
endowed with appropriate polymer measures.

Namely, for any $m,n\in\Z$ with $m<n$ and any $x,y\in\R$, we can introduce the 
point-to-point polymer measure 
$\mu^{m,n}_{x,y,\omega}$ via its density
\begin{align*}
p^{m,n}_{x,y,\omega}(x_m,\ldots,x_n)
&=\frac{\prod_{k=m}^{n-1} \left[g_{2\visc}(x_{k+1}-x_k)e^{-\frac{F_k(x_k)}{2\visc}}\right]
}{Z^{m,n}_{x,y,\omega}},
\end{align*}
with respect to $\delta_x\times \Leb^{n-m-1} \times \delta_y$, where $\Leb$ is the Lebesgue
measure on $\R$.
The random normalizing quantity $Z^{m,n}_{x,y,\omega}$ in this formula has been defined 
in~\eqref{eq:Z}.
It is called the point-to-point partition function. The random measure 
$\mu^{m,n}_{x,y,\omega}$ can be viewed as the Gibbs measure obtained from the Gaussian random walk connecting $x$ to $y$ over time interval~$[m,n]_\Z=[m,n]\cap\Z$, by reweighting paths using the potential energy function
\begin{equation}
\label{eq:path-energy} 
 H_\omega^{m,n}(x_m,\ldots,x_n)=\sum_{k=m}^{n-1}F_{k,\omega}(x_k), \quad  (x_m,\ldots, x_n)\in \R^{n-m+1},
\end{equation}
and temperature $2\visc$. Note the asymmetry in the definition of $H_\omega^{m,n}$: we have to include $k=m$, but exclude $k=n$. All our results
on polymer measures are proved for this choice of path energy, but it is straightforward to obtain their counterparts for the version
of energy where $k=m$ is excluded and $k=n$ is included.

In our notation for partition functions, we omit the dependence on~$\visc$. In fact, in the remaining part of the paper, we will always assume $\visc=1/2$ to make the notation lighter. Adaptation of all the statements and proofs to the general $\visc>0$ is straightforward. We also often omit the $\omega$
argument in all the notations used above and often write $Z^{m,n}(x,y)$ for $Z^{m,n}_{x,y}$.

The polymer density can also be expressed as
\[
 p^{m,n}_{x,y,\omega}(x_m,\ldots,x_n)= \frac{(2\pi)^{-(n-m)/2}e^{-A_\omega^{m,n}(x_m,\ldots,x_n)}}{Z^{m,n}_{x,y,\omega}},
\]
where the total action 
\begin{equation}
 \label{eq:action}
A^{m,n}_\omega(x_m,\ldots,x_n)=H_\omega^{m,n}(x_m,\ldots,x_n)+I^{m,n}(x_m,\ldots,x_n)
\end{equation}
is defined via the potential energy introduced in~\eqref{eq:path-energy} and the kinetic energy:
\[
I^{m,n}(x_m,\ldots,x_n)=\frac{1}{2}\sum_{k=m}^{n-1}(x_{k+1}-x_k)^2.
\]

Asymptotic properties of directed polymer models similar to ours have been extensively studied in the literature, see, e.g., surveys \cite{Comets-Sh-Y-survey:MR2073332}, \cite{Hollander:MR2504175}, \cite{Giacomin:MR2380992}. Here, we will mention only results most tightly
related to ours. 

One of our first results is the existence of the infinite volume quenched density of the free energy:
\begin{theorem} \label{thm:shape-function}  There is a constant $\alpha_0\in\R$ 
such that for any $v\in\R$,
\[
 \frac{\ln Z^{0,n}_{0,nv}}{n}\ \asconv\ \alpha_0-\frac{v^2}{2}, \quad {n\to\infty}.
\]
\end{theorem}
We prove it in Section~\ref{sec:shape-function} after providing some useful properties of partition functions in 
Section~\ref{sec:properties-of-partition-function}.  The quadratic term $-\frac{v^2}{2}$ comes
  from the shear-invariance symmetry (see~(\ref{eq:def-of-shear})) of our model.

The existence of infinite volume normalized quenched free energy has been obtained for a variety of polymer models
using subadditivity arguments, 
see~\cite{Carmona-Hu:MR1939654}, \cite{Comets-Shiga-Yoshida:MR1996276}, \cite{Vargas:MR2299713},
\cite{Comets-Fukushima-Nakajima-Yoshida:MR3406700}, for 
lattice polymers under various assumptions and \cite{Comets-Yoshida:MR2117626}, \cite{Comets-Cranston:MR3038513},
\cite{Comets-Yoshida:MR3085670} for some continuous models. Variational characterizations of the free energy in terms of auxiliary skew-invariant functions (cocycles) were developed in  \cite{Yilmaz-2009:MR2531552}, \cite{Rassoul-Agha--Seppalainen:MR3176363}, \cite{Rassoul-Agha--Seppalainen--Yilmaz:MR2999296}, \cite{Rassoul-Agha--Seppalainen--Yilmaz:2013arXiv1311.3016G}, \cite{Rassoul-Agha--Seppalainen--Yilmaz:2014arXiv1410.4474R}.

Concentration inequalities for the finite volume free energy have been obtained for several polymer models, see, e.g., 
\cite{Mejane:MR2060455}, \cite{Carmona-Hu-2004:MR2073415}, \cite{Rovira-Tindel:MR2129770}. These estimates also imply that the fluctuation of the quenched energy for polymer of length $n$ are (roughly)
bounded by $n^{1/2}$ and the typical transversal fluctuations for polymer paths themselves in those settings are smaller than (roughly) $n^{3/4}$, although it is believed that for a large class of models including ours (KPZ universality class, see, e.g.,~\cite{Corwin:MR2930377}), the true scalings are $n^{1/3}$ and $n^{2/3}$, respectively. We prove new concentration inequalities for our model. Their statements, related results, and proofs are given in Section~\ref{sec:concentration}. 

In the proof of concentration inequalities and throughout the rest of the paper we heavily use the 1-dimensional character of the model. More
specifically, in Section~\ref{sec:monotonicity} we prove monotone dependence of point-to-point polymer measures on the endpoints, along with
several other monotonicity results.  

In Section~\ref{sec:delta-straightness-and-tightness}, we use the concentration inequalities obtained in Section~\ref{sec:concentration}
to derive so-called $\delta$-straightness which in turn is then used to prove tightness of sequences
of polymer measures. 
In Section~\ref{sec:infinite-volume-polymer-measure}, we study the thermodynamic limit for polymer measures. Namely, we use tightness to construct 
infinite volume polymer measures as limits of finite volume ones and prove a uniqueness statement. These infinite volume polymer
measures will be used in Section~\ref{sec:global-solutions-and-ratios-of-partition-functions} to construct global solutions
of the Burgers equation and prove their uniqueness and pullback attraction property.

The notion of $\delta$-straightness goes back to~ \cite{Ne}.
It can be derived from the concentration of finite volume free energy and the uniform curvature
assumption on the shape function, introduced in~\cite{Ne}.
It was later used in~\cite{Licea1996}, \cite{HoNe}, \cite{Wu}, \cite{ferrari2005}, \cite{CaPi}, \cite{BCK:MR3110798}
and~\cite{kickb:bakhtin2016} in the context of optimal paths in random environments.
In these papers, either the curvature assumption was assumed (as in \cite{Licea1996}) or the
  shape functions were explicitly known so that the curvature assumption was satisfied.

To state the thermodynamic limit results precisely, we need some notation.

For every $m,n\in\Z$ satisfying $m<n$, we denote the set of all paths
$\gamma:\{m,m+1,\ldots,n\}\to\R$ by $S^{m,n}$.  
If in addition a point $x\in\R$ is given, then $S_{x,*}^{m,n}$ denotes the set of all such paths that
satisfy $\gamma_m=x$. If $n=\infty$, then we understand the above spaces as the spaces of one-sided semi-infinite paths.

If points $x,y\in\R$ are given, then $S_{x,y}^{m,n}$  denotes the set of all such paths that
satisfy $\gamma_m=x$ and $\gamma_{n}=y$. The random point-to-point polymer measure $\mu_{x,y}^{m,n}=\mu_{x,y,\omega}^{m,n}$ is concentrated on  $S_{x,y}^{m,n}$ for 
all~$\omega\in\Omega$.

We call a measure $\mu$ on  $S^{m,n}_{x,*}$ a polymer measure if there is a probability measure~$\nu$ on~$\R$
such that $\mu=\mu_{x,\nu}^{m,n}$, where 
\[
\mu_{x,\nu}^{m,n}=\int_{\R}\mu_{x,y}^{m,n}\nu(dy). 
\] 
We call $\nu$ the terminal measure for~$\mu=\mu_{x,\nu}^{m,n}$. It is also natural to call $\mu$ a point-to-measure polymer
measure associated to~$x$ and~$\nu$. Later we will also introduce and study a related notion of point-to-line polymer measure.

For every $(m,x)\in\ZR$, we denote by $S_{x}^{m,+\infty}$ the set of paths $y:\{m,m+1,m+2,\ldots\}\to\R$
such that $y_{m}=x$.

A measure $\mu$ on $S_{x}^{m,+\infty}$ is called an infinite volume polymer measure if for any $n\ge m$
the projection of $\mu$ on $S^{m,n}_{x,*}$ is a polymer measure. This condition is equivalent to the Dobrushin--Lanford--Ruelle (DLR) condition
on the measure $\mu$. 

We will need the notions of the law of large numbers (LLN) and the strong law of large numbers (SLLN) for polymer measures. We begin with SLLN. For 
every $(m,x)\in\ZR$ and every $v\in\R$,  we denote
\[
S_x^{m,+\infty}(v)=\left\{\gamma\in S_x^{m,+\infty}:\ \lim_{n\to\infty}\frac{\gamma_n}{n}=v\right\} 
\]
and say that the strong law of large numbers (SLLN) with slope $v\in\R$ holds for a measure $\mu$ on $S_x^{m,+\infty}$ if $\mu(S_x^{m,+\infty}(v))=1$.

We say that LLN with slope $v\in\R$ holds for a sequence of Borel measures $(\nu_n)$ on $\R$ if
for all $\delta>0$,
\[
 \lim_{n\to\infty} \nu_n([(v-\delta)n,(v+\delta) n])=1.
\]
Finally, for any $(m,x)\in\ZR$, we say that a measure $\mu$ on $S_x^{m,+\infty}$ satisfies LLN with slope $v$ if the sequence of its marginals 
$\nu_k(\cdot)=\mu\{\gamma:\ \gamma_k\in\cdot\}$ does.

We denote by $\Pc_x^{m,+\infty}(v)$  the set of all polymer measures on $S_x^{m,+\infty}$ satisfying SLLN with slope $v$. The set of all polymer measures on $S_x^{m,+\infty}$ satisfying LLN with slope $v$ 
is denoted by $\widetilde \Pc_x^{m,+\infty}(v)$. These sets are random since they depend on the realization of the environment, but we suppress the dependence on $\omega\in\Omega$ in this notation.

\begin{theorem}\label{thm:thermodynamic-limit}
  Let $v\in\R$. Then there is a set $\Omega_v\in\Fc$  and such that
\begin{enumerate}[1.]
 \item $\Pp(\Omega_v)=1$.
 \item \label{it:existence-uniqueness-of-infinite-volume} For all $\omega\in\Omega_v$ and all $(m,x)\in\Z\times\R$, there is a polymer measure $\mu_{x}^{m,+\infty}(v)$ such that 
 \[\Pc_x^{m,+\infty}(v)=\widetilde\Pc_x^{m,+\infty}(v)=\{\mu_{x}^{m,+\infty}(v)\}.\]  The finite-dimensional distributions of $\mu_{x}^{m,+\infty}(v)$ are absolutely continuous. 
 \item \label{it:thermodynamic-limit} For all $\omega\in\Omega_v$, all $(m,x)\in\Z\times\R$, and 
 for every sequence of measures $(\nu_n)$ satisfying LLN with slope $v$, 
 finite-dimensional distributions of $\mu_{x,\nu_n}^{m,n}$ converge to $\mu_{x}^{m,+\infty}(v)$ in total variation. 
\end{enumerate}
\end{theorem}

In other words, with probability one, there is a unique infinite volume polymer measure with prescribed endpoint and slope. Moreover, this infinite volume measure can be obtained via a thermodynamic limit, i.e., as a limit of finite volume polymer measures.

A similar result was obtained in \cite{Georgiou--Rassoul-Agha--Seppalainen--Yilmaz:MR3395462} for a model called log-gamma polymer. 
Log-gamma polymer describes a random walk in a certain random potential on the lattice $\Z^2$. Compared to that model, the one that we study has several features that make the analysis harder. Namely, in our model, the space is continuous and the increments of 
the polymer paths are not uniformly bounded. Moreover, our model does not give rise to explicit computations that are possible for the log-gamma polymer, so we have to rely only on estimates. Of course, a very useful feature of our model is that the free energy
function is exactly computed in Theorem~\ref{thm:shape-function} (except an unknown additive constant), it is quadratic and thus strongly convex.  

Note that we 
prove the thermodynamic limit not just for point-to-point polymers, but also for more general point-to-measure polymers. This can be done for gamma-polymers as well. In~\cite{Georgiou--Rassoul-Agha--Seppalainen--Yilmaz:MR3395462} similar results
on point-to-line polymers are established for terminal conditions on the line given by a linear tilt function. 
Our results on pullback attraction in Section~\ref{sec:global-solutions-and-ratios-of-partition-functions} allow to state a version of such a result in our
setting,  with more general tilt functions that are required to be only asymptotically linear.

Tightly connected to the thermodynamic limit results in~\cite{Georgiou--Rassoul-Agha--Seppalainen--Yilmaz:MR3395462} are results on the limits of ratios of partition functions. 
Logarithms of these limiting ratios are polymer counterparts of Busemann functions that compare
actions of infinite geodesics to each other in zero temperature models such as first passage percolation (FPP), last passage percolation,
or zero-viscosity Burgers equation, see~\cite{HoNe}, \cite{CaPi-ptrf}, \cite{BCK:MR3110798},\cite{kickb:bakhtin2016},
 \cite{Rassoul-Agha--Seppalainen--Yilmaz:2013arXiv1311.3016G},
 \cite{Rassoul-Agha--Seppalainen--Yilmaz:2015arXiv151000859G}, \cite{DamronHanson2014},
 \cite{Damron2017} and~\cite{AHD:2015arXiv151103262A}, which is a recent survey on FPP. In~\cite{Rassoul-Agha--Seppalainen--Yilmaz:2013arXiv1311.3016G} and \cite{Rassoul-Agha--Seppalainen--Yilmaz:2014arXiv1410.4474R} a variational approach to ratios of partition functions 
 is described.
 It should be noted that in~\cite{Rassoul-Agha--Seppalainen--Yilmaz:2015arXiv151000859G}
  and~\cite{DamronHanson2014}, \cite{Damron2017}, some differentiablity assumptions on the shape
 function were used to study the semi-infinite geodesic and the Busemann function.

We also prove a result on limits of partition function ratios for our model:
\begin{theorem} 
  \label{thm:ratio-partition-function-intro}
For every $v \in \R$, there is a full measure set $\Omega'_{v}$ such that for all $\omega\in\Omega'_v$, for all
$(n_1,x_1),(n_2,x_2) \in \Z\times\R\ $,
and for every sequence of
numbers $(y_N)$ with $\lim\limits_{N \to  \infty} y_N/N = v$, we have
\begin{equation*}
\lim\limits_{N\to \infty}  \frac{Z_{x_1, y_N}^{n_1, N}}{Z_{ x_2, y_N}^{n_2, N}} = G,
\end{equation*}
where $G=G_{v,\omega}(n_1,x_1,n_2,x_1) \in(0,\infty)$ does not depend on the sequence $(y_N)$.
\end{theorem}
This theorem along with some related and more general results is proved in 
Section~\ref{sec:global-solutions-and-ratios-of-partition-functions}.  In fact, we also prove convergence of logarithmic derivatives of partition function ratios. 
This extension is nontrivial but useful in our space-continuous setting since, due to the Hopf--Cole transformation, solutions of the Burgers equation can be expressed through these logarithmic derivatives.
In particular, this enables us to obtain in Section~\ref{sec:global-solutions-and-ratios-of-partition-functions} our main 1F1S results on global solutions of the Burgers equation stated in Section~\ref{sec:main_results}. 

For these results, we need to work with time-reversed polymer measures on paths that are defined down to $-\infty$ in time. The subtlety here is that our results on forward polymer measures are not 
absolutely equivalent to those for backward polymer measures due to the slight asymmetry in the definition of the 
energy function~$H^{m,n}_\omega$. However, these results (and their proofs) are straightforward modifications of each other, and we often find it practical 
to ignore the difference between them.

\medskip

In Section~\ref{sec:overlap},
we use the result on convergence of partition function ratios to derive a version of hyperbolicity property for the positive temperature polymer case. Namely, we show that the marginals of any two polymer measures with the same slope 
are asymptotic to each other:

\begin{theorem}
\label{thm:convergence-total-variation}
Let $v \in \R$. On a full measure event $\Omega'_{v}$, for any $(n_1,x_1), (n_2,x_2) \in \Z
\times\R$, we have
\begin{equation*}
  \lim_{N \to \infty} \| \mu_{x_1}^{n_1,+\infty}(v) \pi^{-1}_N - \mu_{x_2}^{n_2,+\infty}(v)
  \pi^{-1}_N \|_{TV} = 0,
\end{equation*}
where $\mu \pi^{-1}_N$ is the projection of a measure $\mu$ on $N$-th coordinate and  
$\| \cdot \|_{TV}$ denotes the total variation distance.
\end{theorem}

In fact, we prove a time-reversed version of this result in Section~\ref{sec:overlap}.

Since the marginals $\mu_{x_i}^{n_i,+\infty}(v) \pi^{-1}_N$ define the entire measure $\mu_{x_i}^{n_i,+\infty}(v)$ uniquely due 
to the Markovian character of nearest neighbor interactions encoded in the action functional, a stronger statement on overlap of measures on paths also follows immediately.

\section{Properties of the partition function}\label{sec:properties-of-partition-function}
We recall that throughout the paper, the reasoning does not depend on the viscosity value, so from now on 
we set $\visc=1/2$ for brevity. Under this convention, $2\visc=1$, so we will often work with Gaussian kernel $g(\cdot)=g_1(\cdot)$.

Let us denote $Z^n(y)=Z^{0,n}(0,y)$ for brevity. The following lemma states how distributional properties
of partition functions behave under shear transformations of space-time. We write $\stackrel{d}{=}$ to
denote identity in distribution.

\begin{lemma}
 For any $m,n\in\Z$ satisfying $m<n$ and any points $x,y\in\R$,
 \[
  Z^{n+l,m+l}(x+\Delta,y+\Delta)\stackrel{d}{=} Z^{n,m}(x,y). 
 \]
Also, for any $v\in\R$,
\[
 Z^n(vn)\stackrel{d}{=} e^{-\frac{v^2}{2}n} Z^{n}(0).
\]
\end{lemma}
\bpf The first statement of the lemma follows from the space-time stationarity of~$F$. For the second claim,
let us make a change of variables $x_k=x'_k+kv$ for $k=0,\ldots, n$ in \eqref{eq:Z}, to obtain the following integral ($x_0=0$ and
$x_n=vn$ are fixed, i.e., $x'=0$ and $x'_n=0$ are fixed):
\begin{align}\notag
Z^{n}(vn)&=\int_{\R}\dots\int_\R \prod_{k=0}^{n-1} [g(x'_{k+1}-x'_k+v)e^{-F_k(x'_k+kv)}] dx'_{1}\ldots dx'_{n-1}\\
         &\stackrel{d}{=}\int_{\R}\dots\int_\R \prod_{k=0}^{n-1} [g(x'_{k+1}-x'_k+v)e^{-F_k(x'_k)}] dx'_{1}\ldots dx'_{n-1}
         \label{eq:applying-shear}
\end{align}
due to the \iid property  and the spatial stationarity of~$F$.
Now notice that
\begin{align*}
\prod_{k=0}^{n-1} g(x'_{k+1}-x'_k+v)
&=\frac{1}{(2\pi)^{n/2}} e^{-\frac{1}{2}\sum_{k=0}^{n-1}  (x'_{k+1}-x'_k+v)^2}
\\&=\frac{1}{(2\pi)^{n/2}} e^{-\frac{1}{2}\sum_{k=0}^{n-1}  (x'_{k+1}-x'_k)^2 -v\sum_{k=0}^{n-1}(x'_{k+1}-x'_k) -\frac{n}{2}v^2}
\\&= e^{-\frac{v^2}{2}n}\prod_{k=0}^{n-1} g(x'_{k+1}-x'_k)
\end{align*}
since
\[
  \sum_{k=0}^{n-1}(x'_{k+1}-x'_k)=x'_n-x'_0=0.    
 \]
Plugging this into~\eqref{eq:applying-shear}, we obtain
\[
 Z^{n}(vn)\stackrel{d}{=}e^{-\frac{v^2}{2}n} \int_{\R}\dots\int_\R \prod_{k=0}^{n-1} [g(x'_{k+1}-x'_k)e^{-F_k(x'_k)}] dx'_{1}\ldots dx'_{n-1}
 = e^{-\frac{v^2}{2}n} Z^{n}(0),
\]
and the proof is completed.
\epf

It is easy to extend this proof and obtain a more general lemma on shift-invariance and shear-invariance. To state this result,
in addition to the existing notation $\theta^{n,x}$ and $L^v$ for transformations of the probability space, we define transformations of $\ZR$:
\[
 \theta_{n,x}(m,y)=(m-n,y-x),\quad\quad L_v(m,y)=(m,y-mv).
\]
 Throughout the paper we write  $Pf^{-1}$ to denote the pushforward of a measure~$P$ under a map $f$.

\begin{lemma} The following holds true:
\label{lem:shear-for-Z_v}  
\begin{enumerate}
\item Let $(n,x),(m,y)\in \ZR$. Then, for all $\omega\in \Omega$,
\[
 \mu^{n,m+n}_{x,y+x,\omega}\theta_{n,x}^{-1}=\mu^{0,m}_{0,y,\theta^{n,x}\omega}.
\]
 \item 
Let $Z_v(n)=e^{\frac{v^2}{2}n}Z^{n}(vn)$, $n\in\N$, $v\in\R$.
Then the distribution of the process $Z_v(\cdot)$ does not depend on $v$.
\item For every $n\in\N$, the process $\bar Z_n(x)=e^{\frac{x^2}{2n}}Z^{n}(x)$, $x\in\R$, is stationary in $x$.
\item Let $n\in\N$, $v\in\R$. Then, for all $\omega\in\Omega$,
\[
 \mu^{0,n}_{0,nv,\omega}L_{v}^{-1}=\mu^{0,0}_{0,0,L^v\omega}.
\]
\end{enumerate}
\end{lemma}

\begin{lemma}\label{lem:expectation-of-Z}
 For any $m,n\in\Z$ satisfying $m<n$ and any $x,y\in\R$,
\[\E Z^{m,n}(x,y)=\lambda^{n-m}g_{n-m}(y-x),\]
where $\lambda=\E e^{-F_0(0)}<\infty$ according to~\ref{item:exponential-moment-with-viscosity}.
\end{lemma}
\bpf We can use Fubini's theorem and the \iid property of $(F_k)$ to write 
\begin{align*}
\E Z^{m,n}(x,y)&=\int_{\R}\dots\int_\R \prod_{k=m}^{n-1} [g(x_{k+1}-x_k)\E e^{-F_k(x_k)}] dx_{m+1}\ldots dx_{n-1}\\&=\lambda^{n-m}g_{n-m}(y-x),
\end{align*}
where we also used the convolution property of Gaussian densities.\epf

\begin{lemma}
\label{lem:smoothness-of-partition-function}
Let $m < n$. For all~$\omega$, the point-to-point partition function $Z^{m,n}(x,y)$ is $C^{\infty}$ in $y$ and as smooth as
$F_m(x)$ in $x$. Moreover, partial derivatives of  $Z^{m,n}(x,y)$ can be obtained by differentiation
under the integral in~\eqref{eq:Z}.
\end{lemma}

\begin{proof}
  If $n-m=1$, the claim is obvious.  
If $n -m \ge 2$, it suffices to show that
\begin{align*}
e^{F_m(x)} Z^{m,n}(x,y) = \int f(x,y,x_{m+1},\dotsc, x_{n-1}) \, dx_{m+1} \dots dx_{n-1}
\end{align*}
is smooth in $x$ and $y$, where 
\[
 f(x,y,x_{m+1},\dotsc, x_{n-1})=e^{-\frac{1}{2}(x-x_{m+1})^2 -  \frac{1}{2}(x_{n-1}-y)^2}\prod_{k=m+1}^{n-1}
                          e^{-\frac{1}{2}(x_{k+1}-x_k)^2- F_k(x_k)}.
\]
By (\ref{eq:def-of-Omega}), we can find a constant $c$ such that if $m+1 \le k \le n-1$, then $|F_k(z)| \le c(|z|+1)$  for all~$z$.
The lemma follows from 
\begin{align*}
&\int \bigg| \frac{\partial^i}{\partial x^i} \frac{\partial^j}{ \partial y^j} f(x,y,x_{m+1}, \dotsc,
  x_{n-1}) \bigg|  \, dx_{m+1} \dots dx_{n-1} \\
  \le&  \int c_i c_j \big(|x-x_{m+1}|^i + 1\big) \big(|y-x_{n-1}|^j + 1\big) \\
  &  \cdot   e^{-\frac{1}{2}(x-x_{m+1})^2 -
    \frac{1}{2}(x_{n-1}-y)^2} \prod_{k=m+1}^{n-1}  e^{-\frac{1}{2}(x_{k+1}-x_k)^2 + c(|x_k| + 1)} \, dx_{m+1}\dots
    dx_{n-1} 
< \infty,
\end{align*}
where $c_i$ are absolute constants.
\end{proof}

As a corollary, we have
\begin{lemma}
\label{lem:smoothness-of-partition-ratio}
Let $m,n > k$.
Then $Z^{k,m}(x,y) /Z^{k,n}(x,z)$ is $C^{\infty}$ in $x$, $y$ and $z$.
\end{lemma}
\section{The shape function or free energy density }\label{sec:shape-function}
The goal of this section is to study the linear growth of $\ln Z^{m,n}(x,y)$ over long time intervals. It is useful to
 introduce an auxiliary function 
\[
Z_*^{m,n}(x,y)=\min_{|\Delta x|,|\Delta y|<1/2}Z^{m,n}(x+\Delta x, y+\Delta y).
\]

\begin{lemma}\label{lem:super-additive} The process $Z_*$ is super-multiplicative, i.e.,
\[
Z_*^{n_1,n_3}(x,z)\ge Z_*^{n_1,n_2}(x,y)Z_*^{n_2,n_3}(y,z).
\]
Equivalently, $\ln Z_*$ is super-additive, i.e.,
\[
\ln Z_*^{n_1,n_3}(x,z)\ge \ln Z_*^{n_1,n_2}(x,y) + \ln Z_*^{n_2,n_3}(y,z).
\]
\end{lemma}
\bpf Let $|x'-x|,|z'-z|<1/2$. Then
\begin{align*}
Z^{n_1,n_3}(x',z')&= \int_{y'\in\R} Z^{n_1,n_2}(x',y')
Z^{n_2,n_3}(y',z')dy'\\
&\ge \int_{y':|y'-y|<1/2} Z^{n_1,n_2}(x',y')
Z^{n_2,n_3}(y',z')dy'
\\
&\ge Z^{n_1,n_2}_*(x,y)Z_*^{n_2,n_3}(y,z),
\end{align*}
and the lemma follows.
\epf

\begin{lemma}\label{lem:shape-function-for-aux} For any $v\in\R$, there is $\alpha(v)\in\R$ such that
\[
 \frac{\ln Z_*^{0,n}(0,nv)}{n}\asconv \alpha(v),\quad n\to\infty.
\]
\end{lemma}
\begin{proof} Due to Lemma~\ref{lem:super-additive} and Kingman's sub-additive ergodic theorem, it suffices
to check that for every $v\in\R$, there is $C(v)>0$ such that
\[
\E \ln  Z_*^{0,n}(0,nv)< C(v)n,\quad n \in\N.
\]
This follows from Jensen's inequality and Lemma~\ref{lem:expectation-of-Z}:
\[
\E \ln  Z^{0,n}_*(0,nv)\le \ln \E  Z^{0,n}(0,nv)\le n\ln \lambda -\frac{(nv)^2}{2n}-\frac{1}{2}\ln(2\pi)-
\frac{1}{2}\ln n,
\]
and the proof is completed.
\end{proof}

\begin{lemma}\label{lem:same-shape-function-for-true-Z} For any $v\in\R$,
\[
\frac{\ln Z^{0,n}(0,nv)}{n}\asconv\alpha(v),\quad n\to\infty,
\]
where $\alpha(v)$ has been introduced in Lemma~\ref{lem:shape-function-for-aux}.
\end{lemma}
\begin{proof} Due to 
Lemma~\ref{lem:shear-for-Z_v},  it is sufficient to prove the lemma for $v=0$.
Lemma~\ref{lem:shape-function-for-aux} and the inequality $Z^{0,n}(0,0)\ge Z^{0,n}_*(0,0)$ imply that  it suffices to check 
\begin{equation}
\label{eq:discrepancy-between-Z-and-min-Z}
\limsup_{n\to\infty}\left(\frac{\ln Z^{0,n}(0,0)}{n}-\frac{\ln Z^{0,n}_*(0,0)}{n}\right)\le 0.
\end{equation}
For this, we need to see that $Z^{0,n}(0,0)/Z^{0,n}_*(0,0)$ is bounded by a function that grows
sub-exponentially in $n$.

First we note that  there is $q>0$ such that
\begin{equation}
 \liminf_{n\to\infty} \frac{Z^{0,n}_*(0,0)}{q^n}\stackrel{\rm a.s.}{>}0.
\label{eq:exp-growth-guaranteed} 
\end{equation}
To see this, it is sufficient to notice that for every $x,y\in[-1/2,1/2]$,
\[
 Z^{0,n}(x,y)\ge \int_{[-1/2,1/2]}\ldots\int_{[-1/2,1/2]} \bar{g}^n e^{-\sum_{k=0}^{n-1}\bar{F}_{k}}dx_1\ldots dx_{n-1},
\]
where
\[
\bar{g}=g(1)=\min_{|z_1|,|z_2|<1/2} g(z_1-z_2),
\]
\[
 \bar{F}_{k}=\max_{|z|<1/2} F_{k}(z),\quad k\ge 0,
\]
and apply the SLLN to the partial sums of \iid sequence $(\bar{F}_{k})_{k\ge 0}$.

To compare $Z^{0,n}(0,0)$ to $Z^{0,n}_*(0,0)$, let us take $r_n=n^{3/4}$, introduce sets 
$A_{-1}=A_{-1}(n)=(-\infty,r_n]$, $A_{0}=A_0(n)=[-r_n,r_n]$, $A_{1}=A_1(n)=[r_n,\infty)$,
and write
\[
 Z^{0,n}(0,0)=\sum_{i,j\in\{-1,0,1\}}B_{ij}^n(0,0),
\]
where
\begin{align*}
 B^n_{ij}(x,y)&=\int_{x_1\in A_i}\int_{x_{n-1}\in A_j} Z^{0,1}(x,x_1)Z^{1,n-1}(x_1,x_{n-1})Z^{n-1,n}(x_{n-1},y)dx_1dx_{n-1}.
\end{align*}

We need to estimate $B_{ij}^n(0,0)/Z^{0,n}_*(0,0) = B_{ij}^n(0,0)/Z^{0,n}(x_*,y_*)$, where points $x_*$ and~$y_*$ provide
minimum in the definition of $Z^{0,n}_*(0,0)$.

Let us estimate $B_{11}^n(0,0)$ and $B_{10}^n(0,0)$ first. 

By the Fubini theorem and the convolution
property of Gaussian densities,
\[
 \E [B_{11}^n(0,0)+B_{10}^n(0,0)]\le \lambda^n  \int_{A_1}\int_{A_{1}\cup A_0} g(x_1)g_{n-2}(x_{n-1}-x_1)g(-x_{n-1})\, dx_1\,dx_{n-1}.
\]
Since $g_{n-2}(z)\le 1$ for all $z\in\R$ and $g$ is a probability density, we conclude that
\begin{align*}
 \E [B_{11}^n(0,0)+B_{10}^n(0,0)] &\le \lambda^n  \int_{A_1}\int_{A_0\cup A_{1}} g(x_1)g(-x_{n-1})\, dx_1\,dx_{n-1}\\
 &\le \lambda^n  \int_{A_1} g(x)\, dx
\le \lambda^n  \Pp\{\Nc>r_n\} \le \lambda^n  \frac{1}{(2\pi)^{1/2} r_n } e^{-r_n^2/2},
\end{align*}
where $\Nc$ is a standard Gaussian random variable.

So, for any $\rho>0$, 
\[
 \Pp\{B_{11}^n(0,0)+B_{10}^n(0,0)>\rho^n\}\le \rho^{-n}\E [B_{11}^n(0,0)+B_{10}^n(0,0)]\le  \frac{\lambda^n}{\rho^n}  \frac{1}{(2\pi)^{1/2} r_n } e^{-r_n^2/2}.
\]
Here, the last factor decays super-exponentially,
and the Borel--Cantelli Lemma implies that for any $\rho>0$,
\begin{equation}
\label{eq:super-esp-decay}
 \lim_{n\to\infty}\frac{B_{11}^n(0,0)+B_{01}^n(0,0)}{\rho^n}\stackrel{\rm a.s.}{=}0.
\end{equation}
Combining \eqref{eq:super-esp-decay} with \eqref{eq:exp-growth-guaranteed} and applying the same reasoning
to all terms $B_{ij}^n$ with $|i|+|j|\ne 0$, we obtain
\begin{equation}
 \label{eq:B_11}
 \lim_{n\to\infty}\frac{\sum_{|i|+|j|\ne 0}B_{ij}^n(0,0)}{Z_*^{0,n}(0,0)}\stackrel{\rm a.s.}{=}0.
\end{equation}

It remains to estimate $B_{00}(0,0)$:
\begin{align}\notag
  \frac{B_{00}^n(0,0)}{Z^{0,n}(x_*,y_*)}
&\le \frac{B_{00}^n(0,0)}{B_{00}^n(x_*,y_*)}\le \max_{x_1,x_{n-1}\in A_0(n)} \frac{Z^{0,1}(0,x_1)Z^{n-1,n}(x_{n-1},0)}{Z^{0,1}(x_*,x_1)Z^{n-1,n}(x_{n-1},y_*)}
\\ \notag&\le \max_{x_1,x_{n-1}\in A_0(n)}
\frac{g(x_1)e^{-F_0(0)} g(-x_{n-1})e^{-F_{n-1}(x_{n-1})}}{g(x_1-x_*)e^{-F_0(x_*)} g(y_*-x_{n-1})e^{-F_{n-1}(x_{n-1})}}
\\ \notag &\le \max_{x_1,x_{n-1}\in A_0(n)}e^{-F_0(0)+F_0(x_*)} e^{(x_*^2+y_*^2)/2} e^{-x_*x_1-y_*x_2}
\\&\le C_1(\omega)e^{r_n}
\label{eq:random-ratio-of-densities}
\end{align}
for some random constant $C_1(\omega)$ and all $n\ge 2$.

Combining~\eqref{eq:B_11} and \eqref{eq:random-ratio-of-densities}, we obtain~\eqref{eq:discrepancy-between-Z-and-min-Z}
and finish the proof of Lemma~\ref{lem:same-shape-function-for-true-Z}.\epf

\begin{theorem} There is a constant $\alpha_0\in\R$ such that for any $v\in\R$,
\[
 \lim_{n\to\infty} \frac{\ln Z^n(vn)}{n}\aseq\alpha_0-\frac{v^2}{2},
\]
i.e., the shape function $\alpha(\cdot)$ introduced in Lemmas~\ref{lem:shape-function-for-aux} 
and~\ref{lem:same-shape-function-for-true-Z}, satisfies
\[
 \alpha(v)=\alpha_0-\frac{v^2}{2}.
\]
\end{theorem}
\bpf 
The theorem follows directly from Lemmas~\ref{lem:same-shape-function-for-true-Z} and~\ref{lem:shear-for-Z_v}.
\end{proof}

\section{Monotonicity}
\label{sec:monotonicity}
The order on the real line plays an important role in our  analysis. The goal of this section is to establish 
monotonicity of polymer measures with respect to endpoints, along with some related results. We begin with an auxiliary lemma on a monotonicity
property of the Gaussian kernel. 

\begin{lemma}
\label{lem:abstract-monotonicity}
 Suppose $\nu$ is a Borel $\sigma$-finite measure such that
\[
Z(x)=\int_{\R}g(z-x)\nu(dz)
\]
is finite for all $x\in\R$, and  
let \[
G(x,y)=\frac{\int_{(-\infty,y]}g(z-x)\nu(dz)}{Z(x)},\quad x,y\in\R. 
\]
Then $G(x,y)$ is nondecreasing in~$y$. If $\nu\{(y,\infty)\}>0$ and $\nu\{(-\infty,y]\}>0$,
then $G(x,y)$ is strictly decreasing in $x$.
\end{lemma}
\begin{proof} The monotonicity in $y$ is obvious.  Due to
\[
\frac{1}{G(x,y)}= \frac{\int_{\R}g(z-x)\nu(dz)}{\int_{(-\infty,y]}g(z-x)\nu(dz)}=1+\frac{\int_{(y,\infty)}g(z-x)\nu(dz)}{\int_{(-\infty,y]}g(z-x)\nu(dz)},
\]
it remains to prove that for all $z\in(y,\infty)$, 
\[
H(x,y,z)=\frac{\int_{(-\infty,y]}g(z'-x)\nu(dz')}{g(z-x)}
\]
decreases in $x$.
We rewrite
\[
 H(x,y,z)=\int_{(-\infty,y]}e^{\frac{-(x-z')^2+(x-z)^2}{2}}\nu(dz')=e^{z^2/2}\int_{(-\infty,y]}e^{x(z'-z)-\frac{z'^2}{2} }\nu(dz').
\]
 Since $z'-z<0$,  the integrand $e^{x(z'-z)-\frac{z'^2}{2} }$ decreases in $x$ and so does the integral on the right-hand side.
\end{proof}

For any $d\in\N$, we denote by ${\preceq}$ 
the natural partial order on $\R^d$, i.e., we write $x\preceq y$ iff
$x_k\le y_k$ for all $k=1,\ldots,d$. A function $f:\R^d\to \R$ is coordinatewise nondecreasing if $x\preceq y$ implies $f(x)\le f(y)$.
For two Borel probability measures $\nu_1,\nu_2$ on $\R^d$, we write $\nu_1\preceq \nu_2$  (and say
that  $\nu_1$ is {\it stochastically dominated} by $\nu_2$)
iff for any bounded coordinatewise nondecreasing
function $f:\R^d\to\R$ 
\[
 \int_{\R^d} f(x)\nu_1(dx)\le \int_{\R^d} f(x)\nu_2(dx).
\]
For $d=1$, $\nu_1\preceq\nu_2$ is equivalent to $\nu_1\{(-\infty,x]\}\ge \nu_2\{(-\infty,x]\}$ for all $x\in\R$. There is also
a coupling characterization of stochastic dominance usually called Strassen monotone coupling theorem 
(see Theorems~7 and~11 in~\cite{Strassen:MR0177430}
and a discussion in~\cite{Lindvall:MR1711599}). To state this theorem and our results on stochastic dominance, we introduce notation that will be used in various
contexts throughout the paper: we use $\pi_k x$ to denote the $k$-th coordinate of $x$, where $x$ is either a vector
or an infinite sequence. We also use $\pi_{m,n}x=(x_{m},\ldots,x_{n})$.

\begin{lemma}[Monotone coupling]\label{lem:dominance-by-coupling}
 Borel measures $\nu_1,\ldots,\nu_n$ on $\R^d$ satisfy $\nu_1\preceq\ldots\preceq \nu_n$  iff there is a measure 
 $\nu$ on $(\R^d)^n$ such that $\nu_k$ is the $k$-th marginal of $\nu$, i.e., $\nu_k=\nu \pi_k^{-1}$, $k=1,\ldots,n$,
and
\[
\nu\{(x^{(1)},\ldots,x^{(n)})\in(\R^d)^n:\ x^{(1)}\preceq\ldots\preceq x^{(n)}\}=1. 
\]
\end{lemma}

\begin{lemma}\label{lem:dominance-1} Let $x\le x'$. Then for any $m,n$ with $m<n$, any $y\in\R$, and all $\omega$, the polymer measure $\mu_{x,y}^{m,n}$ is stochastically dominated by $\mu_{x',y}^{m,n}$.
\end{lemma}
\begin{proof} The reasoning does not depend on $m$, so we set $m=0$ for brevity.
  We prove by induction in $k$ that for all $x < x'$ and for any $k\in (0,n)\cap\N$, there is a
  measure $\nu_{k}$ on $(\R^k)^2$ such that 
  \begin{align*}
  \nu_k(\cdot \times \R^k) &= \mu_{x,y}^{0,n} \pi_{1,k}^{-1},\\
  \nu_k(\R^k \times \cdot) &= \mu_{x',y}^{0,n} \pi_{1,k}^{-1}, 
  \end{align*}
  and 
  \begin{equation}
    \label{eq:dominance-coupling}
\nu_{k} \{ (x, x'): x \preceq x'\} = 1.    
  \end{equation}
In particular, taking $k = n-1$ we obtain the conclusion of the lemma.

 Let us check the case $k=1$ first. 
\begin{align*}
 \mu_{x,y}^{0,n}\pi^{-1}_{1} \bigl( (-\infty, r] \bigr)
&=\frac{1}{Z_{x,y}^{0,n}}\int_{(-\infty,r]}Z_{x,{s}}^{0,1}Z_{{s},y}^{1,n} d{s}
=\frac{\int_{(-\infty,r]}g({s}-x)e^{-F_0(x)}Z_{{s},y}^{1,n} d{s}}{\int_{\R}g({s}-x)e^{-F_0(x)}Z_{{s},y}^{1,n} d{s}}
\\
&=\frac{\int_{(-\infty,r]}g({s}-x)Z_{{s},y}^{1,n} d{s}}{\int_{\R}g({s}-x)Z_{{s},y}^{1,n} d{s}}.
\end{align*}
Introducing $\nu(d{s})=Z_{{s},y}^{1,n} d{s}$, we can apply Lemma~\ref{lem:abstract-monotonicity} to see that
$\mu_{x,y}^{0,n} \pi_{1}^{-1}\bigl( (-\infty,r] \bigr)$ is decreasing in $x$.
Therefore
$\mu_{x,y}^{0,n}\pi_{1}^{-1}\preceq\mu_{x',y}^{0,n}\pi_{1}^{-1}$ for $x < x'$, which finishes the argument
for the basis of induction.

Suppose for $k \ge 1$ the desired $\nu_k$ have been constructed.
We will construct $\nu_{k+1}$ using $\nu_k$. The basis of induction (the claim for 1-dimensional marginals) implies that for any $z,z'\in\R$ satisfying $z \le z'$, there is a measure $\nu_{z,z'}$ on
$\R \times \R$ such that 
\begin{align*}
\nu_{z,z'}(\cdot \times \R) &= \mu_{z,y}^{k,n}\pi_{k+1}^{-1}(\cdot),\\
\nu_{z,z'}(\R \times \cdot) &= \mu_{z',y}^{k,n}\pi_{k+1}^{-1}(\cdot), 
\end{align*}
and $\nu_{z,z'}\{(w,w'): w
\le w' \} = 1$.
Then the measure $\nu_{k+1}$ defined by
\begin{multline*}
  \nu_{k+1}\bigl( (A_{1}\times\dotsm \times A_{k+1})\times (A_{1}' \times \dotsm \times A_{k+1}')
  \bigr) \\
  = \int_{x_i \in A_i, x_i' \in A_i',i \le k} \nu_k \bigl( dx_{1}, ..., dx_{k},
  dx_{1}', ..., dx_{k}' \bigr)  \nu_{x_{k}, x_{k}'} \bigl( A_{k+1}\times A_{k+1}' \bigr)
\end{multline*}
satisfies (\ref{eq:dominance-coupling}) with $k$ replaced by $k+1$.
To see that $\nu_{k+1}$ has correct marginals, it suffices to notice that from the definition of
polymer measures, we have
\[
  \mu_{x,y}^{0,n}(A_{1}\times \dotsm \times A_{n-1})
  = \int_{x_i \in A_i, i\le k} \mu_{x,y}^{0,n}\pi_{1,k}^{-1} (dx_{1}, ..., dx_{k})
  \mu_{x_{k}, y}^{k,n} ( A_{k+1} \times \dotsm \times A_{n-1})  
\]
for any $x, y$ and $k \le n-1$.
\end{proof}

One can also easily obtain a time-reversed version of Lemma~\ref{lem:dominance-1}:

\begin{lemma}\label{lem:dominance-2} Let $y\le y'$. Then for any $m,n$ with $m<n$, any $x\in\R$, and all $\omega$, the polymer measure $\mu_{x,y}^{m,n}$ is stochastically dominated by $\mu_{x,y'}^{m,n}$. 
\end{lemma}

We can now state the main result of this section. It easily follows from Lemmas~\ref{lem:dominance-by-coupling},~\ref{lem:dominance-1}, and~\ref{lem:dominance-2}. 
\begin{lemma}[Main monotonicity lemma]\label{lem:main-monotonicity} The following holds for all $\omega\in\Omega$:
\begin{enumerate}[1.] 
\item  Let $x\le x'$ and $y\le y'$. Then for any $m,n$ with $m<n$, the polymer measure $\mu_{x,y}^{m,n}$ is stochastically dominated by $\mu_{x',y'}^{m,n}$.

\item If  two distributions $\nu_1,\nu_2$ on $\R$ satisfy $\nu_1\preceq\nu_2$, then, for any $x\in\R$ and any $m,n\in\Z$ satisfying $m\le n$, we 
have $\mu_{x,\nu_1}^{m,n}\preceq \mu_{x,\nu_2}^{m,n}$. 

\item If $x\le x'$, then for any distribution $\nu$ on $\R$ and any $m,n\in\Z$ satisfying $m\le n$, we have $\mu_{x,\nu}^{m,n}\preceq \mu_{x',\nu}^{m,n}$.
\end{enumerate}
\end{lemma}

\section{Concentration inequality for free energy}\label{sec:concentration}
The aim of this section is to prove a concentration inequality of the free energy around its
asymptotic value.
It is natural to expect that our model
belongs to the KPZ universality class. In particular, the fluctuations of $\ln Z^n$ around the mean are expected to be of order~$n^{\chi}$,   where~$\chi=1/3$.  The main result of this section may be interpreted 
as~$\chi\le 1/2$.

\begin{theorem}
  \label{thm:concentration-of-free-energy}
  There are positive constants $c_0, c_1, c_2, c_3$ such that for all $v \in \R$, all $n > c_0$ and all $u \in (c_3n^{1/2} \ln^{3/2}n,  n \ln n]$,
  \begin{equation*}
    \Pp \{ |\ln Z^n(0,vn) - \alpha(v)n | > u \} \le c_1 \exp
    \left\{ -c_2 \frac{u^2}{n \ln^2n} \right\}.
  \end{equation*}  
\end{theorem}
The right endpoint of the interval of applicability of this inequality is chosen to be $ n \ln n$. In fact, similar
inequalities with appropriate adjustments hold for different choices of the right endpoint.

Due to the invariance under shear transformations (Lemma~\ref{lem:shear-for-Z_v}), it suffices to prove this theorem for $v = 0$.
The proof can be divided into two steps:
the first step is Lemma~\ref{lem:free-energy-concentration-around-expectation}, where we obtain a concentration expectation of $\ln Z^n$ around its expectation
$\E \ln Z^n$ using Azuma's inequality after certain truncation; the second step is to estimate the speed of convergence, i.e., how far $\E \ln Z^n$ deviates from the linear function $\alpha_0n$, using Lemmas~\ref{lem:doubling-argument-inequality}
and~\ref{lem:doubling-argument-lemma}.

Such concentration inequalities have been proved for various FPP/LPP and polymer models with different tails.  The first such result appeared in~\cite{Kesten:MR1221154} on FPP, with a tail of
$e^{-cu/\sqrt{n}}$.  Using Talagrand's inequality, this can be improved to~$e^{-cu^2/n}$.
In \cite{benjamini2003}, the authors proved that for FPP with edge weight distribution~$\Pp(w_e = a
) =\Pp(w_e = b) = 1/2$, the variance of $\ln Z^n$ is $O(\frac{n}{\log n})$, which is sublinear.
The result was later strengthened to a concentration inequality with a tail $e^{-cu \sqrt{\ln n/n}}$
  for more general distributions,  see~\cite{benaim2008} and \cite{damron2014}.
  In~\cite{Alexander2013}, similar concentration inequality was obtained for a polymer model.

Our method in this section is more elementary and will not lead to a sharper subgaussian
concentration as mentioned above, but it is sufficient, in conjunction with the convexity of the shape function, to help us to establish straightness
estimates in section~\ref{sec:delta-straightness-and-tightness}. Moreover, 
in a forthcoming paper on
the zero-temperature limit of polymer measures, we strengthen our estimates proving
a functional  concentration inequality on the free energy process indexed by temperature (or viscosity) in uniform topology.  

\begin{lemma}
  \label{lem:free-energy-concentration-around-expectation}
  There are positive constants $b_0,b_1,b_2,b_3$ such that for all $n \ge b_0$ and all $u\in (b_3, n\ln n]$, 
  \begin{equation*}
    \Pp\{ |\ln Z^n - \E \ln Z^n| > u \}
    \le
    b_1 \exp \left\{ -b_2 \frac{u^2}{n\ln^2n} \right\}.
  \end{equation*}
\end{lemma}

Here are some notational conventions. For a Borel set $B \subset \R^{n-1}$, we define
\begin{equation*}
Z^n(B) = \int_{\R\times B\times \R} \prod_{j=0}^{n-1} g(x_{j+1}-x_j)e^{-F_j(x_j)}
    \delta_0(dx_0)dx_{1}...dx_{n-1}\delta_0(dx_n).
\end{equation*}
It is often practical to use the same description for events that are technically defined on paths over different time intervals. To that end we
introduce more notation. We recall that $\pi_{m,n}$ denotes the restriction of a vector or sequence onto the time interval $[m,n]_{\Z}$.
For a Borel set $D \subset \R^{\infty} =  S^{-\infty,\infty}_{*,*}$, we
define 
\begin{equation*}
  \mu_{x,y}^{m,n}(D) = \mu_{x,y}^{m,n}(\pi_{m,n} D), \quad
  Z^{m,n}_{x,y}(D) = Z^{m,n}(x,y,D) = Z^{m,n}_{x,y} \mu_{x,y}^{m,n}(D).
\end{equation*}
For example, if $m \le k \le l \le n$ and 
\begin{equation*}
D = \{ \gamma: \gamma_k \in A_k, \dotsc, \gamma_l \in A_l\},
\end{equation*}
then 
\begin{equation*}
Z^{m,n}_{x,y}(D) =  \int_{x_k \in A_k, \dotsc, x_l \in A_l} \prod_{j=m}^{n-1} g(x_{j+1}-x_j)e^{-F_j(x_j)}
    \delta_x(dx_m)dx_{m+1}...dx_{n-1}\delta_y(dx_n).
  \end{equation*}
If $x=y=0$, then we write $Z^{m,n}_{x,y}(D) = Z^{m,n}(D)$.
  
  \begin{lemma}
    \label{lem:partition-function-not-too-small}
    There are constants $\beta, r_0, \rho_0 \in (0,1)$ such that for all $n\in\N$ and any Borel set $B$ that satisfies $[0,1]^{n-1} \subset B \subset \R^{n-1}$, 
    \begin{equation*}
      \Pp \left\{ Z^n(B) \le \rho^n \right\} \le \beta^n \left(\frac{\rho}{\rho_0}\right)^{nr_0},\quad \rho>0,
    \end{equation*}
or, equivalently,
\begin{equation*}
      \Pp \left\{ Z^n(B) \le z \right\} \le \beta^n \frac{z^{r_0}}{\rho_0^{nr_0}},\quad z>0.
    \end{equation*}
  \end{lemma}

\begin{proof}
For any $\rho>0$,
\begin{align*}
&\Pp\{Z^n(B)<\rho^n\} 
\\ \le & \Pp\left\{\int_{\R\times [0,1]^{n-1}\times \R} \prod_{k=0}^{n-1} [g(x_{k+1}-x_k)e^{-F_k(x_k)}] \delta_0(dx_0)dx_{1}\ldots dx_{n-1}\delta_0(dx_n)<\rho^n \right\}
\\
\le & \Pp\left\{ g^{n}(1)e^{-\sum_{k=0}^{n-1} F^{*}_k(0)} <\rho^n\right\}
\le \Pp\left\{- \sum_{k=0}^{n-1} F^{*}_k(0)<n\ln(\rho g^{-1}(1))\right\}.
\end{align*}
Let us use the standard Cram\'er large deviation approach. Taking any number $r>0$ and using the Markov inequality, we can write
\begin{align*}
& \quad\Pp\bigg\{ -\sum_{k=0}^{n-1} F^*_k(0) <n\ln\big(\rho g^{-1}(1)\big)\bigg\} \\
&=
 \Pp\bigg\{ -\sum_{k=0}^{n-1} \Big(F^*_k(0)-\E F^*_k(0) \Big) <n\ln \Big(\rho g^{-1}(1)e^{\E
  F^*_k(0) } \Big) \bigg\}
  \\
&\le      \exp \Big( rn\ln\big(\rho g^{-1}(1)e^{\E F^*_k(0) }\big)  \Big)  \Big(\E e^{r \big(F^{*}_0(0)-\E F^{*}_0(0)\big)}\Big)^n 
\\
&\le \exp \Big( n\big(r\ln(\rho g^{-1}(1)e^{\E F^*_k(0) }) + \Lambda(r) \big) \Big)
\\
&\le \exp\Big(  n\big(r\ln(\rho_0 g^{-1}(1)e^{\E F^*_k(0)}) + \Lambda(r)  \big)  \Big) \Big(\frac{\rho}{\rho_0}\Big)^{nr}
\le e^{nH(r,\rho_0)}\left(\frac{\rho}{\rho_0}\right)^{nr},
\end{align*}
where
\[
 \Lambda(r)=\ln \E e^{r (F^{*}_0(0)-\E F^{*}_0(0))},\quad r \le \eta,
\]
and
\[
H(r,\rho_0)= r\ln(\rho_0 g^{-1}(1)e^{\E F^*_k(0)}) + \Lambda(r) ,\quad r,\rho_0>0.
\]
Here we use the assumption \ref{item:exponential-moment-for-maximum}. There is $\sigma>0$ such that 
\[
\Lambda(r)=\frac{\sigma^2}{2} r^2+o(r^2),\quad r\to 0. 
\]
Therefore, choosing $\rho_0\in(0,1)$ small enough to guarantee that
\[
 \rho_0 g^{-1}(1) e^{\E F^*_k(0)}<1,
\]
and choosing $r \in (0,1)$ sufficiently small we guarantee that $H(r,\rho_0)<0$,  
and the lemma follows.
\end{proof}

We recall the standard tail estimate for normal distribution:
\begin{lemma}
\label{lem:tail-estimate-for-normal}
For any $a, b > 0$, 
\begin{equation*}
  \int_{x \ge b } e^{-\frac{x^2}{a}} dx
  \le \frac{a}{2b} e^{-\frac{b^2}{a}};\quad\quad\quad   \int_{|x| \ge b } e^{-\frac{x^2}{a}} dx
  \le \frac{a}{b} e^{-\frac{b^2}{a}}.
\end{equation*}
\end{lemma}
  Let
  \begin{equation*}
    \Gamma(R) = \{ (x_1, ..., x_{n-1}) \in \R^{n-1}: |x_i| \le Rn, i=1,...,n-1 \}.
  \end{equation*}
\begin{lemma}\label{lem:large-deviation-for-Z-bar}
  There are constants $R, r_1 > 0$ such that for sufficiently large $n$, 
  \begin{equation*}
    \Pp \left\{ \frac{Z^n(\Gamma(R)^c)}{Z^n} > 2^{-n} \right\} \le e^{-r_1 n}, 
  \end{equation*}
or, in terms of polymer measures,
\begin{equation*}
  \Pp \Big\{ \mu_{0,0}^{0,n} \big\{ \gamma: \max_{1\le k \le n-1}|\gamma_k| \ge Rn \big\} \ge 2^{-n}
  \Big\}  \le e^{-r_1n}.
\end{equation*}

\end{lemma}
\begin{proof}
   Suppose  $1 \le k \le n-1$.
  By \ref{item:exponential-moment-with-viscosity} and the Fubini theorem, 
\begin{equation*} 
\begin{split}
  \E \int_{|y| \ge Rn} Z^{0,k}(0,y) Z^{k,n}(y,0) dy
  &= \lambda^n \int_{|y| \ge Rn} g_k(y)g_{n-k}(y) dy \\
  &= \lambda^n \int_{|y| \ge Rn} \frac{1}{2\pi \sqrt{k(n-k)}} e^{-\frac{y^2}{2k} -
    \frac{y^2}{2(n-k)}} dy. \\
\end{split}
\end{equation*}
Since $k(n-k) \ge n-1$ and 
\begin{equation*}
 \frac{y^2}{2k} + \frac{y^2}{2(n-k)} \ge \frac{y^2}{n},
\end{equation*}
we can continue the computation to find 
\begin{equation*}
  \E \int_{|y| \ge Rn} Z^{0,k}(0,y) Z^{k,n}(y,0) dy
  \le  \frac{\lambda^n}{2\pi \sqrt{n-1}} \int_{|y| \ge Rn} e^{- \frac{y^2}{n}} dy \le  \frac{\lambda^n}{2\pi  R\sqrt{n-1}} e^{-R^2n}. 
\end{equation*}
  Then by  Markov inequality,
  \begin{equation}
    \label{eq:nominator-not-too-big}
\begin{split}
  \Pp \left\{ Z^n\big(\Gamma(R)^c\big) \ge \left(\frac{\rho_0}{2}\right)^n \right\}
  &\le \frac{\E Z^n \bigl( \Gamma(R)^c \bigr)}{(\rho_0/2)^{n}}
\\  &
\le \frac{1}{(\rho_0/2)^n} \sum_{k=1}^{n-1} \E \int_{|y| \ge Rn} Z^{0,k}(0,y) Z^{k,n}(y,0) dy
\\ &
    \le \frac{n\lambda^n e^{- R^2n}}{2\pi R \sqrt{n-1} (\rho_0/2)^n}
 \le e^{-\theta_1n}
\end{split}
  \end{equation}
for some $\theta_1>0$ if $R$ is chosen large enough.  Now, taking $\rho=\rho_0$ in Lemma \ref{lem:partition-function-not-too-small}, we can find $\theta_2>0$ such that
  \begin{equation*}
    \Pp \{ Z^n \le \rho_0^n \} \le e^{-\theta_2n},
  \end{equation*}
  and the lemma follows.
\end{proof}

\begin{lemma}
  \label{lem:second-moment-growth-of-partition-function}
  There is a constant $d_1$ such that for any Borel set $B$ that satisfies $[0,1]^{n-1} \subset B \subset \R^{n-1}$,
  \begin{equation*}
    \E \ln^2 Z^n(B) \le d_1 n^2.
  \end{equation*}
  In particular, $\E \ln^2 Z^n \le d_1 n^2$ for all n.
\end{lemma}
\begin{proof}
  We will apply Lemma \ref{lem:partition-function-not-too-small} to control $\ln^2 Z^n(B)$ when $Z^n(B)$ is small, and when $Z^n(B)$ is large we will use Markov inequality.
  Taking any $\mu > \lambda \vee 1$, we split $\E \ln^{2}Z^n(B)$ into three parts
  \begin{align*}
    \E \ln^2Z^n(B) &= A_1 +A_2 + A_{3 } \\
    &= \E \ln^2Z^n(B)\ONE_{\{Z^n(B) \le \rho_0^n\}}
    + \E \ln^2Z^n(B) \ONE_{\{ \rho_0^n < Z^n(B) < \mu^n \}}\\
    &\quad + \E \ln^{2 } Z^n(B) \ONE_{\{ Z^n(B) \ge \mu^n   \}}.
  \end{align*}
  Clearly $A_2 \le n^2 (\ln^2\rho_0 \vee \ln^{2}\mu)$.
To estimate $A_1$, we apply Lemma \ref{lem:partition-function-not-too-small}:
\begin{align*}
  \E \ln^2 Z^2(B)\ONE_{\{  Z^n(B) \le \rho_0^n    \}}
  &= \sum_{m=0}^{\infty} \E \ln^2 Z^n(B) \ONE_{\{ \rho_0^n2^{-m-1}  < Z^n(B) \le  \rho_0^n2^{-m}   \}} \\
  &\le \sum_{m=0}^{\infty} \ln^2(\rho_0^n2^{-m-1}) \Pp \{ Z^n(B) \le \rho_0^n2^{-m} \}\\
  &\le \sum_{m=0}^{\infty}(n \ln \rho_0 - (m+1)\ln 2)^2 \beta^n 2^{-mr_0}\\
  &\le \beta^{n} \sum_{m=0}^{\infty} 2(n^2 \ln^2\rho_0 + (m+1)^2 \ln^2 2) 2^{-mr_0}\le Cn^2\beta^n,
\end{align*}
for some $C>0$ and all $n\in\N$.
For $A_3$, noting that $Z^{n}(B) \le Z^{n}$, we have
\begin{align*}
  A_3 &\le \sum_{k=1}^{\infty} \E \left[ \ln^2 Z^n(B) \ONE_{\{ \mu^{nk} \le Z^n(B) < \mu^{n(k+1)}\}}
        \right] \le \sum_{k=1}^{\infty} \ln^2\mu^{n(k+1)}\, \Pp \{ Z^n > \mu^{nk}\} \\
      & \le n^2 \sum_{k=1}^{\infty} (k+1)^2 \ln^2\mu \cdot \frac{\lambda^n}{\mu^{nk}} 
\le Cn^2,
\end{align*}
for some $C>0$ and all $n\in\N$, which completes the proof.
\end{proof}

Let us fix $R \in \N$ such that Lemma \ref{lem:large-deviation-for-Z-bar} holds and denote $Z^n(\Gamma(R))$ by $\tilde{Z}^n$.
\begin{lemma} \label{lem:truncated-log-partition-function-expectation}
For some constant $d_2 > 0$, 
\[
 0 \le \E \ln Z^n - \E \ln \tilde{Z}^n \le d_2,\quad n\in\N.
\] 
\end{lemma}
\begin{proof}
The first inequality is obvious since $\tilde{Z}^n \le Z^n$. 
Let $\Lambda = \{ \tilde{Z}^{n}/Z^n \le 1 - 2^{-n} \}$. By Lemma~\ref{lem:large-deviation-for-Z-bar}, $\Pp(\Lambda) \le e^{-r_1n}$.
We also have $\E \ln^2 \tilde{Z}^n \le d_1n^2$ and $\E \ln^2 Z^n \le d_1n^2$ by Lemma \ref{lem:second-moment-growth-of-partition-function}.
Therefore,
\begin{align*}
  \E \ln Z^n - \E\ln \tilde{Z}^{n}
  &\le \E \big(-\ln(\tilde{Z}^n/Z^n) \ONE_{\Lambda^c} \big) + \E (|\ln Z^n| + |\ln \tilde{Z}^n|)\ONE_{\Lambda} \\
  &\le -\ln(1-2^{-n}) + \sqrt{2(\E \ln^2 Z^n + \E \ln^2 \tilde{Z}^n)   } \sqrt{\Pp (\Lambda)} \\
  & \le \ln 2 + \sqrt{4d_1n^2 e^{-r_1n}}  \le d_2,
\end{align*}
and the lemma is proved.
\end{proof}

To obtain a concentration inequality for $\ln \tilde{Z}^n$, we will
apply Azuma's inequality:
\begin{lemma}
  \label{lem:azuma}
  Let $(M_k)_{0\le k\le N}$ be a martingale with respect to a
  filtration $(\Fc_k)_{k=0}^{N}$.
  Assume there are constants $c_k$, $1\le k\le N$ such that the
  increments $\Delta_k = M_k - M_{k-1}$ satisfy
  \begin{equation*}
    | \Delta_k  | \le c,\quad k=1,\dots, N.
  \end{equation*}
Then
\begin{equation*}
  \Pp \left\{ |M_N - M_0| \ge x \right\}  \le 2 \exp \left( \frac{-x^2}{ 2Nc^2} \right).
\end{equation*}
\end{lemma}

To apply the Azuma inequality, we need to introduce an appropriate martingale with bounded increments.
The modified partition function $\tilde{Z}^{n}$ depends only on the potential process on $\{0,1,\dotsc,n-1\}\times [-Rn, Rn]$,
and we need an additional truncation of the potential on this set.

Let $b>4/\eta$, where $\eta$ is taken from condition \ref{item:exponential-moment-for-maximum}.
For $0 \le k \le n-1$ and $x \in [-Rn, Rn]$, we define (suppressing the dependence on $n$ for brevity)
\[\xi_k = \max\{F^{*}_k(j): j=-Rn, -Rn+1, ..., Rn-1 \},\quad k=0,\ldots,n,\] 
\begin{equation*}
  \bar{F}_k(x)
  = \begin{cases}
0, & \xi_k \ge b\ln n,\\
F_k(x), & \text{otherwise},
\end{cases}
\end{equation*}
and, setting  $x_0=x_n=0$,
\begin{equation*}
\tilde{Z}^n(\bar{F}) = \int_{|x_i| \le Rn} \prod_{j=0}^{n-1} g(x_{j+1}-x_j)
e^{-\bar{F}_j(x_j)} dx_1\ldots dx_{n-1}.
\end{equation*}

\begin{lemma}
  \label{lem:partition_function_of_omega_bar_and_omega}
 For sufficiently large $n\in\N$, the following holds true:
\begin{gather}
 \label{eq:item:3}
\E \exp \Bigl(  \frac{\eta}{2} \xi_k \ONE_{\{\xi_k \ge b\ln n\}}   \Bigr) \le 2,\\
\label{eq:item:4} \E \xi_k \le b\ln  n + 4/\eta,\\
\label{eq:item:8}  \Pp \{  | \ln \tilde{Z}^{n} - \ln \tilde{Z}^n(\bar{F}) | > x \}
    \le 2e^{-\eta x/2},\quad x > 0,\\
\label{eq:item:9} |\E \ln \tilde{Z}^n - \E\ln \tilde{Z}^n(\bar{F})| \le 4/\eta.
\end{gather}
\end{lemma}

\begin{proof}  Since $\xi_k$ is the maximum of $2Rn$ random variables with the same distribution, we have
  \begin{align}\notag
 \E \exp \Bigl( \frac{\eta}{2}\xi_k  \ONE_{\{\xi_k > b\ln n\}}\Bigr) 
  &\le  1 + \E e^{\frac{\eta}{2}\xi_k} \ONE_{\{\xi_k > b\ln n\}} \le 1 + \E \sum_{j=-Rn}^{Rn-1} e^{\frac{\eta}{2} F^{*}_k(j)} 
\ONE_{\{F^{*}_k(j) > b\ln n\}} \\
\notag
  &\le 1 + 2Rn  \E e^{ \frac{\eta}{2} F^{*}_k(0)} \ONE_{\{ F^{*}_k(0) > b\ln n\}} \le 1 + 2Rn \frac{\E e^{\eta F^{*}_k(0)}}{e^{ \frac{b\eta}{2} \ln n}} \\
  & \le 1 + \frac{c}{n^{\frac{b\eta}{2}-1}},
\label{eq:exp-moment-above-bln}
  \end{align}
  where $c = 2R \E e^{\eta F^{*}_k(0)}$ is a constant. Now \eqref{eq:item:3} follows from $b>4/\eta$.

  If $x > b\ln n$, then by Markov inequality and~\eqref{eq:item:3}, we have
\[
  \Pp \{\xi_k \ge x \} \le \Pp \{\xi_k \ONE_{\{ \xi_k \ge b\ln n\}} \ge x\} \le e^{-\eta x/2} 
\E \exp \Bigl( \frac{\eta}{2} \xi_k \ONE_{\{ \xi_k \ge b \ln n\}} \Bigr) \le 2e^{-\eta x/2}
\] for sufficiently large $n$.
This implies~\eqref{eq:item:4}:
\[
  \E \xi_k \le b \ln n + \E \xi_k \ONE_{\{\xi_k \ge b\ln n\}} 
  \le b \ln n  + \int_{b\ln n}^{\infty} \Pp \{\xi_k  \ge x\} \, dx \le b\ln
  n + \frac{4}{\eta}.
\]
It follows from the definition  of $\tilde{Z}(\bar{F})$ that
  \begin{equation*}
    |\ln \tilde{Z}^{n} - \ln \tilde{Z}^{n}( \bar{F})| \le
    \sum_{k=0}^{n-1} \xi_k\ONE_{\{\xi_k > b\ln n\}}.
  \end{equation*}
  By Markov inequality, the \iid property of $(\xi_k)$, and~\eqref{eq:exp-moment-above-bln}, we have
  \begin{align*}
\Pp \left\{ | \ln \tilde{Z}^{n} - \ln \tilde{Z}^n(\bar{F}) | > x \right\} 
  & \le \Pp \Bigl\{ \frac{\eta}{2}\sum_{k=0}^{n-1} \xi_k \ONE_{\{\xi_k > b\ln n\}} > \frac{\eta x}{2} \Bigr\} \\
  & \le e^{-\eta x/2} \E \exp \Bigl(\frac{\eta}{2} \sum_{k=0}^{n-1} \xi_k  \ONE_{\{\xi_k> b\ln n\}} \Bigr)\\
  & =  e^{-\eta x/2} \Big(\E \exp \big(  \frac{\eta}{2}\xi_0  \ONE_{\{\xi_0> b\ln n\}}\big) \Big)^{n} \\
    & \le e^{-\eta x/2} (1+ c/n^{\eta b/2-1})^{n}.
\end{align*}
Since $b>4/\eta$,~\eqref{eq:item:8} follows. It immediately implies
\begin{equation*}
\begin{split}
  |\E  \ln \tilde{Z}^{n} - \E\ln \tilde{Z}^n(\bar{F}) |
  &\le \E | \ln \tilde{Z}^n - \ln \tilde{Z}^n(\bar{F})|\\
  &= \int_0^{\infty} \Pp \{ | \ln \tilde{Z}^{n} - \ln \tilde{Z}^n(\bar{F}) | > x \} \, dx 
  \le 4/\eta,
\end{split}
\end{equation*}
so \eqref{eq:item:9} is also proved.
\end{proof}  

\begin{lemma}
\label{lem:martingale-concentration-for-ZF-bar}
For all $n\in\N$ and $x>0$,
  \begin{equation*}
    \Pp \left\{ |\ln \tilde{Z}^{n}(\bar{F}) - \E \ln
      \tilde{Z}^n(\bar{F}) | > x \right\} \le 2\exp \left\{
      -\frac{x^2}{8nb^2\ln^2 n} \right\}.
  \end{equation*}
\end{lemma}

\begin{proof}
    Let us introduce the following martingale $(M_k,\Fc_k)_{-1\le k \le n-1}$:
\begin{equation*}
  M_k = \E ( \ln \tilde{Z}^n(\bar{F})\,|\, \Fc_k ), \quad -1 \le k \le n-1,
\end{equation*}
where 
\begin{equation*}
  \Fc_{-1}= \{\emptyset,\Omega \}, \quad
  \Fc_k = \sigma \big( F_{i}(x) : 0 \le i \le k \big), \quad  k=0,\dotsc,n-1.
\end{equation*}
If we can show that $|M_{k} - M_{k-1}| \le 2b\ln n$, $ 0 \le k \le n-1$, 
then the conclusion of the lemma follows immediately from Azuma's inequality (Lemma~\ref{lem:azuma}).

For a process $\bar G$, an  independent distributional copy of $\bar F$, we define 
\begin{multline*}
  \tilde{Z}^n ([\bar{F},\bar{G}]_k) = \int_{|x_i|\le Rn}
\prod_{i=0}^k g(x_{i+1}-x_i) e^{-\bar{F}_i(x_i)} \\ \cdot \prod_{i=k+1}^{n-1} g(x_{i+1}-x_i)
e^{-\bar{G}_i(x_i)} \delta_0(d x_0) dx_1\dotsm dx_{n-1} \delta_0(d x_n).
\end{multline*}
Denoting by $P_k$ the distribution of $\bar{F}_{k}(\cdot)$, we obtain 
\begin{equation*}
\begin{split}
&|M_{k} - M_{k-1}| 
\\=& \Bigg| \int \ln \tilde{Z}^n([\bar{F},\bar{G}]_k) \prod_{i=k+1}^{n-1} P_i\big( d
\bar{G}_i \big) -
\int \ln \tilde{Z}^n([\bar{F},\bar{G}]_{k-1}) \prod_{i=k}^{n-1} P_i\big( d \bar{G}_i \big) \Bigg|
\\
\le& \int \Big|  \ln \tilde{Z}^n([\bar{F},\bar{G}]_k) - \ln \tilde{Z}^n([\bar{F},\bar{G}]_{k-1}) \Big|
\prod_{i=k}^{n-1} P_i\big( d \bar{G}_i \big) \\
\le& \int \Big( \sup_{|x| \le Rn} |\bar{F}_k(x)| + \sup_{|x| \le Rn} | \bar{G}_k(x)| \Big)
\prod_{i=k}^{n-1} P_i\big( d \bar{G}_i \big) \le 2b\ln n,
\\
\end{split}
\end{equation*}
since $|\bar{F}_k(x)|$ and $|\bar{G}_k(x)|$ are bounded by  $b \ln n$.
This completes the proof.
\end{proof}

\begin{proof}[Proof of Lemma \ref{lem:free-energy-concentration-around-expectation}]
  Suppose $u \in \big( 3(d_2+4/\eta + 3), n \ln n \big] $, where $d_2$ is introduced in Lemma \ref{lem:truncated-log-partition-function-expectation}.
  Then 
\begin{align*}
  \Pp \big\{ |\ln Z^n -& \E \ln Z^n| \ge u \big\}
  \le \Pp \bigl\{ |\ln Z^n - \ln \tilde{Z}^{n}| \ge 1 \bigr\} \\
  &+  \Pp \bigl\{ | \ln \tilde{Z}^n - \ln \tilde{Z}^n(\bar{F}) | \ge \frac{u}{3} \bigr\}
  +  \Pp \bigl\{ |\ln \tilde{Z}^n(\bar{F}) - \E \tilde{Z}^n(\bar{F}) | \ge \frac{u}{3}\bigr\}\\
  &+  \Pp \bigl\{  |\E \ln \tilde{Z}^n(\bar{F}) - \E \ln \tilde{Z}^n| \ge 4/\eta +1 \bigr\}\\
  &+ \Pp \bigl\{ |\E \ln \tilde{Z}^n - \E \ln Z^n| \ge d_2 +1  \bigr\}.
\end{align*}
By \eqref{eq:item:9} and Lemma
\ref{lem:truncated-log-partition-function-expectation}, the last two terms equal 0.
By Lemma \ref{lem:large-deviation-for-Z-bar}, we have
\begin{equation*}
\Pp \big\{ |\ln Z^n - \ln \tilde{Z}^n | \ge |\ln (1-2^{-n}) | \; \big\} \le 2^{-r_1n},
\end{equation*}
which bounds the first term.
The second term is bounded by $2e^{-\eta u/6}$ by \eqref{eq:item:8}
and 
the third term is bounded by
$2e^{-u^2/72b^2n\ln n}$
by Lemma \ref{lem:martingale-concentration-for-ZF-bar}.
Combining all these estimates together, we obtain
\begin{equation*}
  \Pp \big\{ |\ln Z^n - \E \ln Z^n| \ge u \big\} \le 2^{-r_1n} + 2e^{-\frac{\eta u}{6}} + 2e^{ -
    \frac{u^2}{72b^2n \ln^2 n}}
  \le b_1 e^{-b_2 \frac{u^2}{n\ln^2 n}},
\end{equation*}
for some constants $b_1,b_2 > 0$, where in the last inequality we use $u \le n\ln n$.
\end{proof}

To prove Theorem~\ref{thm:concentration-of-free-energy}, we need to estimate $|\E \ln Z^n - \alpha_0n|$.

\begin{lemma}
\label{lem:small-probability-for-single-square-root-jump}
There are constants $R', r_2 > 0$ such that for sufficiently large $n$,
\begin{equation*}
  \Pp \Bigl\{ \mu_{0,0}^{0,n} \{ \gamma:  | \gamma_k -
    \gamma_{k+1}| \ge R' \sqrt{n} \} \ge 2^{-n}
  \Bigr\} \le e^{-r_2n},\quad 0 \le k \le n-1.
\end{equation*}
\end{lemma}
\begin{proof}  Let $A_k = \{ \gamma: | \gamma_k - \gamma_{k+1}| \ge R' \sqrt{n} \}$.
  The random variables $\mu_{0,0}^{0,n}(A_{n-1})$ and $\mu_{0,0}^{0,0}(A_0)$ are identically
  distributed.
  So let us assume $k \neq 0$.
  
  By Markov inequality, for $\rho_0$ defined in Lemma \ref{lem:partition-function-not-too-small}, 
\begin{equation*}
\begin{split}
  &\quad \Pp \Big\{ Z^{0,n}_{0,0}(A_k) \ge \Big(\frac{\rho_0}{2} \Big)^n \Big\} 
 \le \Big( \frac{2}{\rho_0} \Big)^n \E Z^{0,n}_{0,0}(A_k) \\
  &\le \Big( \frac{2}{\rho_0}\Big)^n \lambda^n
  \int_{|x_k-x_{k+1}| \ge R' \sqrt{n}} \prod_{j=0}^{n-1} g(x_j-x_{j+1}) \delta_0(dx_0)dx_1\cdots
  dx_{n-1} \delta_0(dx_n) \\
  &\le \Big( \frac{2}{\rho_0}\Big)^n \lambda^n \frac{1}{\sqrt{2\pi}}
  \int_{|y_k| \ge R' \sqrt{n}} \prod_{j=1}^{n-1} g(y_j) dy_1 \cdots dy_{n-1} \\
  &=  \Big( \frac{2}{\rho_0}\Big)^n \lambda^n \frac{1}{\sqrt{2\pi}}
  \int_{|y_k| \ge R' \sqrt{n}} g(y_k) dy_k 
  \le \Big( \frac{2}{\rho_0}\Big)^n \lambda^n \frac{1}{2\pi R' \sqrt{n}}e^{-R'^2n/2}
  \le e^{-Cn}
\end{split}
\end{equation*}
for some constant $C>0$ and all $n$
if  $R'$ is chosen sufficiently large.
Here, in the third line, we use $g(x_0-x_1) \le 1/\sqrt{2\pi}$ and then a
change of variables $y_i = x_i-x_{i+1}$.
The lemma now follows from 
\begin{equation*}
  \Pp \big\{  \mu_{0,0}^{0,n}(A_k) \ge 2^{-n} \}
  \le \Pp \big\{ Z^{0,n}_{0,0} \le \rho_0^n \big\} + \Pp \Big\{ Z^{0,n}_{0,0}(A_k) \ge \Big(
  \frac{\rho_0}{2} \Big)^n \Big\}
\end{equation*}
and Lemma \ref{lem:partition-function-not-too-small}.
\end{proof}

\begin{lemma}
  \label{lem:doubling-argument-inequality}
  There is positive constant $b_4$ such that for sufficiently large $n$, 
  \begin{equation*}
    |\E \ln Z^{0,2n} - 2\E \ln Z^{0,n} | \le b_4 n^{1/2} \ln^2 n.
  \end{equation*}
\end{lemma}

\begin{proof}
  For $R, R'$ in Lemma \ref{lem:large-deviation-for-Z-bar} and Lemma
  \ref{lem:small-probability-for-single-square-root-jump},  define
  \begin{align*}
  B &= \{ \gamma:   \max_{1 \le i \le 2n-1}|\gamma_i| \le 2Rn \}, \\
  C &= \{ \gamma : |\gamma_n-\gamma_{n+1}| \le R'\sqrt{2n},\
      |\gamma_n-\gamma_{n-1}| \le R' \sqrt{2n}  \}.
\end{align*}
Using an argument similar to the proof of
Lemma \ref{lem:truncated-log-partition-function-expectation}, we  can obtain
\begin{equation}
\label{eq:truncation_does-not-change-much}
  |\E \ln Z^{0,2n} (B\cap C) - \E \ln Z^{0,2n}| \le d_3,
\end{equation}
for some constant $d_3$.
In fact, introducing
$
  \Lambda = \left\{ \frac{Z^{0,2n}(B\cap C)}{Z^{0,2n}} \le 1-3\cdot 2^{-n}
  \right\}$,
we can write for large $n$:
\begin{align*}
  & \E \ln  Z^{0,2n}  - \E \ln Z^{0,2n}(B\cap C) \\
  \le & \E \big( - \ln (Z^{0,2n}(B
  \cap C) / Z^{0,2n}) \ONE_{\Lambda^c} \big)
  + \E \big(|\ln Z^{0,2n}(B\cap C)| + | \ln Z^{0,2n}|\big) \ONE_{\Lambda} \\
  \le & -\ln(1-3\cdot 2^{-n}) + \sqrt{2 \big( \ln^2 Z^{0,2n}(B \cap C) + \ln^2 Z^{0,2n} \big)
  } \sqrt{\Pp(\Lambda)} \\
  \le & -\ln(1-3\cdot 2^{-n}) + \sqrt{4d_1n^2 e^{-K_2n}} \le d_3,
\end{align*}
where in the last line we used Lemma
\ref{lem:second-moment-growth-of-partition-function}
and the estimate $\Pp(\Lambda) \le e^{-K_2n}$ for some constant $K_2$ implied by Lemmas~\ref{lem:large-deviation-for-Z-bar} 
and~\ref{lem:small-probability-for-single-square-root-jump}.
Clearly, $\E \ln Z^{0,2n} > \E \ln Z^{0,2n}(B \cap C)$, so
(\ref{eq:truncation_does-not-change-much}) follows.

To prove the lemma, it is sufficient now to obtain upper and lower bounds on $\E \ln Z^{0,2n} (B\cap C)$.
First, we bound $\E \ln Z^{0,2n}(B\cap C)$ from below. Since $|\gamma_n| \le Rn$ on~$B$, we have 
\begin{align*}
  Z^{0,2n}(B\cap C)
\ge Z^{0,2n}(B\cap C \cap \{\gamma_n \in [0,1) \}). 
\end{align*}
Let us now compare the action of every path $\gamma$ in $B\cap C \cap \{\gamma_n \in [0,1) \}$ to the action
of the modified path $\bar \gamma$ defined by $\bar{\gamma}_n = 0$ and 
$\bar{\gamma}_j = \gamma_j$ for $j \neq n$. We recall that action was defined in~\eqref{eq:action}.
Since $|\gamma_{n+1} - \gamma_{n}| \le R' \sqrt{2n}$, $|\gamma_n-\gamma_{n-1}| \le R' \sqrt{2n}$, and $|\gamma_n| \le 1$,
we get
\begin{equation*}
\begin{split}
  |A^{0,2n}(\gamma) - A^{0,2n}(\bar \gamma)|
   &\le \frac{1}{2}\left| (\gamma_{n+1}-\gamma_n)^2-\gamma_{n+1}^2 +
    (\gamma_{n-1}-\gamma_n)^2 - \gamma_{n-1}^2
  \right|
  + 2F_{n}^{*}(0)
  \\ &\le 2R' \sqrt{2n} + 1 + 2F_{n}^{*}(0).
\end{split}
\end{equation*}
So, there is a constant $K_3>0$ such that 
\begin{equation}
\label{eq:Z-almost-a-product}
  Z^{0,2n}(B\cap C)\ge  Z^{0,n}(D^-)Z^{n,2n}(D^+) e^{-K_3 \sqrt{n} - 2F_{n}^{*}(0)},
\end{equation}
where
\begin{align*}
  D^+ &= \{\gamma: |\gamma_{n+1}| \le R' \sqrt{2n} + 1,\
        |\gamma_i| \le 2Rn,\ n+1 \le  i \le 2n-1 \}, \\
    D^- &= \{ \gamma: |\gamma_{n-1}| \le R' \sqrt{2n} + 1,\
        |\gamma_i| \le 2Rn,\ 1 \le  i \le n-1\}.
\end{align*}
Similarly to~\eqref{eq:truncation_does-not-change-much}, there is a constant $d_4$ such that
\begin{equation*}
  |\E \ln Z^{n}(D^{+}) - \E \ln Z^{n} | \le d_4, \quad
  |\E \ln Z^n(D^-) - \E \ln Z^n | \le d_4.
\end{equation*}
Combining this with~\eqref{eq:Z-almost-a-product}, we obtain
\begin{align*}
\E \ln  Z^{0,2n}(B \cap C)
  &\ge  \E \ln Z^{0,n}(D^-) + \E \ln Z^{n,2n}(D^+) -K_3 \sqrt{n} - 2\E F_{n}^{*}(0)
     \\
  &\ge  2\E \ln Z^{0,n} - 2d_4-K_3 \sqrt{n} - 2\E F_{n}^{*}(0),
\end{align*}
the desired lower bound.

Next, we bound $\E \ln Z^{0,2n}(B\cap C)$ from above. 
 Similarly to~\eqref{eq:Z-almost-a-product}, we compare actions of generic paths in $B\cap C$  to the actions of the modified paths that take integer
values at time $n$:
\begin{align*}
  Z^{0,2n}(B\cap C)
  &= \sum_{k=-2Rn}^{2Rn-1}  Z^{0,2n}(B\cap C\cap \{\gamma_n \in [k,k+1) \} ) \\
  &\le \sum_{k=-2Rn}^{2Rn-1} Z^{0,n}(0,k)Z^{n,2n}(k,0) e^{K_3 \sqrt{n} + 2F_{n}^{*}(k)} \\
  &\le 4Rn \max_k [Z^{0,n}(0,k)Z^{n,2n}(k,0)]
    e^{K_3 \sqrt{n} + 2\max_k F_{n}^{*}(k)}, 
\end{align*}
where the maxima are taken over $-2Rn \le k \le 2Rn-1$.
Taking logarithm and then expectation of both sides, we obtain
\begin{align*}
  &\quad \E \ln Z^{0,2n}(B\cap C) \\
  &\le \E \max_k \ln Z^{0,n}(0,k) + \E \max_k \ln Z^{n,2n}(k,0)
    + \ln(4Rn) + K_3 \sqrt{n} + 2 \E \max_k F_{n}^{*}(k) \\
  &\le \max_k \E \ln Z^{0,n}(0,k) + \E \max_k X_k
    + \max_k \E \ln Z^{n,2n}(k,0) + \E \max_k Y_k \\ & \hspace{9cm}+ K_{4}(\ln n + \sqrt{n} + 1) \\
  &\le 2\E\ln Z^{0,n} + \E \bigl[\max_k X_k+\max_k Y_k\bigr] + K_{4} (\ln n + \sqrt{n} + 1),
\end{align*}
for some constant $K_4>0$, where
\begin{equation*}
  X_k = \ln Z^{0,n}(0,k) - \E\ln Z^{0,n}(0,k), \quad
  Y_k = \ln Z^{n,2n}(k,0) - \E\ln Z^{n,2n}(k,0).
\end{equation*}
 In the second inequality, we used~\eqref{eq:item:4} to
 conclude 
\begin{equation*}
  \E \max_{-2Rn\le k \le 2Rn-1}F_{n}^{*}(k) \le b \ln (2n) + 4/\eta,
\end{equation*}
and in the third inequality, we used the fact that
\begin{equation*}
  \E \ln Z^{0,n}(0,k) \le \E \ln Z^{0,n}, \quad
  \E \ln Z^{n,2n}(k,0) \le \E \ln Z^{n,2n} = \E \ln Z^{0,n}.
\end{equation*}
It remains to bound $\E \max_k X_k$ and $\E \max_k Y_k$.
By the shear invariance, all $X_k$ and~$Y_k$  have the same distribution, so
\begin{equation*}
  \E X_n^2 = \E Y_n^2 = \E \ln^2 Z^n \le K_5n^2.
\end{equation*}
Let
\begin{equation*}
  \Lambda = \left\{ \max_k X_k \le rn^{1/2}\ln^{3/2} n, \quad \max_k Y_k \le rn^{1/2}\ln^{3/2}n \right\},
\end{equation*}
with $r$ to be determined. We have
\begin{align*}
  \E \left[\max_k X_k + \max_k Y_k\right]
  &\le \E \ONE_{\Lambda} (\max_kX_k + \max_kY_k) + \E \ONE_{\Lambda^{c}} (\max_kX_k+\max_kY_k) \\
  & \le 2rn^{1/2}\ln^{3/2}n + \sqrt{2\Pp(\Lambda^c) \E( \max_kX_k^2+\max_kY_k^2)} \\
  & \le 2rn^{1/2} \ln^{3/2}n + \sqrt{16\Pp(\Lambda^c) K_5Rn^3}.
\end{align*}
To bound the second term by a constant, we use Lemma~\ref{lem:free-energy-concentration-around-expectation}:
\begin{align*}
  \Pp (\Lambda^c) &\le \sum_{k=-2Rn}^{2Rn-1} \Biggl[\Pp \left\{ |\ln Z^{0,n}(0,k) - \E\ln Z^{0,n}(0,k)| \ge rn^{1/2}\ln^{3/2}n \right\}  \\
  &\quad + \Pp \left\{ |\ln Z^{n,2n}(k,0) - \E\ln Z^{n,2n}(k,0)| \ge rn^{1/2} \ln^{3/2}n \right\} \Biggr]\\
                  &\le 8Rn \Pp \left\{ |\ln Z^n - \E\ln Z^{n}| \ge rn^{1/2}\ln^{3/2}n \right\} \\
  & \le 8Rn b_1 \exp \{-b_2 r^2 \ln n \},
\end{align*}
and choose $r$ to ensure $b_2r^2>4$. This completes the proof.\end{proof}

We can now use the following straightforward adaptation of~Lemma~4.2 of~\cite{HoNe}  from real argument functions to sequences:
\begin{lemma}
\label{lem:doubling-argument-lemma}
Suppose that number sequences $(a_n)$ and $(g_n)$ satisfy the following conditions: $a_n/n \to \nu$ as $n
\to \infty$, $|a_{2n} - 2a_n| \le g_n$ for large $n$ and $\lim_{n\to \infty} g_{2n}/g_n = \psi < 2$.
Then for any $c > 1/(2-\psi)$ and for all large $n$ (depending on $c$), 
\begin{equation*}
|a_n - \nu n | \le c g_n.
\end{equation*}
\end{lemma}

\begin{proof}
  Let $b_n = a_n/n$, $h_n = g_n/(2n)$.
  Then $|b_{2n} - b_n| \le h_n$ for large $n$ and $\lim_{n\to \infty } h_{2n}/h_n = \psi/2$.

  Since $\psi/2 < 1 - \frac{1}{2c} $, there is $N$ such that $h_{2m}/ h_m \le 1 -\frac{1}{2c}$ for all
  $m > N$.
  We can assume further that for $m > N$, the inequality $|b_{2m}-b_m| \le h_m$ holds.
  Let us now fix $n > N$.  Then for $k \ge 0$ we have $h_{2^kn} \le \bigl( 1-\frac{1}{2c} \bigr)^k h_{n}$.
  Therefore,
\begin{equation*}
  | b_n - b_{2^kn}| \le \sum_{i=0}^{k-1} |b_{2^{i+1}n} - b_{2^in} | \le \sum_{i=0}^{k-1} h_{2^in}
  \le 2c h_{n}.
\end{equation*}
We complete the proof by letting $k \to \infty$ in this estimate.
\end{proof}

\begin{proof}[Proof of Theorem \ref{thm:concentration-of-free-energy}]
It suffices to prove the inequality for $v = 0$.
Thanks to Lemma \ref{lem:doubling-argument-inequality} and Theorem \ref{thm:shape-function}, we can apply Lemma
\ref{lem:doubling-argument-lemma} to $a_n = \E Z^n$, $g_n = b_4 n^{1/2} \ln^2n$, $\nu = \alpha_0$
and $\psi =\sqrt{2}$ to obtain 
\begin{equation}
  \label{eq:expectation-deviation-from-linear}
|\E Z^n - \alpha_0n| \le c n^{1/2} \ln^2n
\end{equation}
for some constant $c$ and sufficiently large $n$.

Then, for all sufficiently large $n$ and $u \in (2cn^{1/2}\ln^2n, n\ln n]$,  we have
\begin{equation*}
  \Pp \bigl\{  |Z^n - \alpha_0n| \ge u
  \bigr\}
  \le \Pp \bigl\{ |Z^n - \E Z^n| \ge u/2 \bigr\} \le b_1 \exp \left(  - b_2 \frac{u^2}{4n\ln^2n} \right),
\end{equation*}
by Lemma~\ref{lem:free-energy-concentration-around-expectation} and~(\ref{eq:expectation-deviation-from-linear}).
\end{proof}

\section{Straightness and tightness}
\label{sec:delta-straightness-and-tightness}

\subsection{Straightness} \label{sec:delta-straightness}
The notion of $\delta$-straightness of paths in random environments
was introduced in~\cite{Ne}. In this section, we prove an analogous straightness property for the positive temperature polymer measure case.  
Some of the results of this section may be interpreted in terms of bounds on the transversal fluctuation exponent. It is expected that the fluctuations
of a typical polymer path of length~$n$ are of order $n^\xi$ for some $\xi$, and the KPZ scalings predict $\xi=2/3$.
Our results show that $\xi\le 3/4$.

We begin with some notation. For a path $\gamma$ and $n\in \Z$, we define
\[ \gamma^{\rm out}(n) = \{ (m,\gamma_m):\ m>n\}.\]
For $v>0$, we define
\[\Co(v) = \{ (n,x)\in\N\times \R:\ |x|\leq n v\}.\]
For  $(n,x)\in \N\times \R$ and $\eta>0$, we define:
\begin{equation*}
{\rm Co}(n,x,\eta) = \{ (m,y)\in \N\times \R\ :\ |y/m - x/n| \le\eta \}.
\end{equation*}

\medskip

Let us fix $\delta \in (0,1/4)$.
\begin{theorem}
  \label{thm:delta-straightness-for-polymer-measure} 
  For any $v > 0$ and $\alpha \in (0,1-2\delta)$,
  there are events $\Omega_1^{(n)} = \Omega_1^{(n)}(\delta,v,\alpha)$ and a constant $Q=Q(\delta)$ with the following properties: 
\begin{enumerate}
\item\label{item:13} For all $\beta \in (0,  1-4\delta)$ and sufficiently large $n$ (depending on $\beta$), 
  \begin{equation*}
\Pp (\Omega_1^{(n)}) \ge 1 - e^{-n^{\beta}};
\end{equation*}

\item\label{item:14} on the event $\Omega_1^{(n)}$ the following is true: for any terminal measure $\nu$ and for all $N\in\N$
satisfying $N/2 > n$, 
  \begin{equation}
    \label{eq:delta-straightness-for-polymer-measure}
    \mu_{0,\nu}^{0,N} \big\{ \gamma : \exists (k,\gamma_k) \in \mathrm{Co}(v),\ N/2 \ge k \ge n,\ \gamma^{\text{out}} \not\subset \mathrm{Co}(k,\gamma_k,Qk^{-\delta}) \big\}
    \le e^{-n^{\alpha}}.
  \end{equation}
\end{enumerate}
\end{theorem}

For $(m,x), (n,y) \in \Z \times \R$ with $m<n$, we define $[(m,x),(n,y)]$ to be the constant velocity path 
connecting $(m,x)$ and $(n,y)$, i.e.,  $[(m,x),(n,y)]_k=x+\frac{k-m}{n-m}(y-x)$ for $k\in [m,n]_{\Z}$.
For any $\alpha > 0$, let us define the event $A_{x,y}^{m,n} = A_{x,y}^{m,n}(\alpha)$:
\begin{equation*}
  A_{x,y}^{m,n} = 
  \Big\{
    \mu_{x,y}^{m,n} \bigl\{ \max_{k \in I(m,n)} |\gamma_k - [(m,x),(n,y)]_k| \ge (n-m)^{1-\delta} \bigr\} \le \exp\big(-(n-m)^{\alpha}\big) \Big\},
\end{equation*}
where $I(m,n) = [\frac{3m+n}{4}, \frac{m+3n}{4}]_{\Z}$.
By translation and shear invariance  all $A_{x,y}^{m,n}$ have the same probability for fixed $m$ and $n$.
We also define the event $\bar{A}_{p,q}^{m,n} = \bar{A}_{p,q}^{m,n}(\alpha)$ for $p, q \in \Z$ to be 
\begin{equation*}
\bar{A}_{p,q}^{m,n} = \bigcap \big\{  A_{x,y}^{m,n} : x \in [p,p+1], y \in [q,q+1]
\big\}. 
\end{equation*}
This is a measurable event since, by continuity, the intersection on the right-hand side may be restricted to
 rational points $x,y$.
Also, all $\bar{A}_{p,q}^{m,n}$ have the same probability for fixed $m$ and $n$.

\begin{lemma}\label{lem:concentration-of-polymer-measure-in-delta-rectangle}
  Let $\alpha \in (0,1-2\delta)$ and $\beta \in (0,1-4\delta)$. Then
  \begin{equation*}
    \Pp \big(   \bar{A}_{0,0}^{0,n}(\alpha) \big) \ge 1- e^{-n^{\beta}}
  \end{equation*}
for sufficiently large $n$.
\end{lemma}
\begin{proof} Lemma~\ref{lem:main-monotonicity} implies
$\mu_{0,0}^{0,n}\preceq\mu_{x,y}^{0,n}\preceq \mu_{1,1}^{0,n}$ for $x,y\in[0,1]$. 
Therefore, it suffices 
to check that for any $\alpha' \in (0,1-2\delta)$, $\beta' \in (0,1-4\delta)$, all
sufficiently large~$n$ and all~$k\in [n/4 ,3n/4]$, 
  \begin{equation*}
    \Pp \Bigl\{ \mu_{0,0}^{0,n} \{ \gamma  : |\gamma_k| > n^{1-\delta}-1  \}
     \ge e^{-n^{\alpha'}} 
     \Bigr\} \le e^{-n^{\beta'}}.
   \end{equation*}
So let us prove this estimate.
Let 
   \begin{align*}
            B &= \{ \gamma : \max_{1\le i \le n-1} |\gamma_i| \le Rn \},\\
       C &= \{ \gamma : |\gamma_k - \gamma_{k+1} | \le R' \sqrt{n},\
       |\gamma_k - \gamma_{k-1} | \le R' \sqrt{n} \},
     \end{align*}
     for $R, R'$ given in Lemma \ref{lem:large-deviation-for-Z-bar} and Lemma
     \ref{lem:small-probability-for-single-square-root-jump}.
     By these two lemmas, the complement of the event 
        $A=\{\mu_{0,0}^{0,n}(B^c\cup C^c)  \le 3\cdot 2^{-n} \}$ 
     is small:
     \begin{multline}
       \label{eq:control-for-A}
       \Pp(A^c) \le  \Pp \big\{ \mu_{0,0}^{0,n}(B^c) \ge 2^{-n} \big\} + \Pp \big\{
       \mu_{0,0}^{0,n}(C^c) \ge 2\cdot 2^{-n} \big\}  \le
       2^{-r_1n} + 2\cdot 2^{-r_2n}.              
     \end{multline}
     Let $D = \bigl\{ \max\limits_{|j|\le Rn} F_{k}^{*}(j) \le \sqrt{n} \bigr\}$.
     By \ref{item:exponential-moment-for-maximum} and Markov inequality,
\begin{equation}
  \label{eq:1}
  \Pp (D^c) \le \sum_{|j|\le Rn} \Pp \{ F_{k}^{*}(j) \ge \sqrt{n} \}
  \le (2Rn+1)e^{-\eta \sqrt{n}} \E e^{\eta F_{k}^{*}(0)}
  \le e^{-k_1 \sqrt{n}},
\end{equation}
for some constant $k_1>0$ and all sufficiently large $n$.
Let us also fix a number $\theta \in \Big(  \frac{\beta'+1}{2}, 1-2\delta \Big)$, and 
for $|j|\le Rn$, define 
\begin{equation*}
\begin{split}
  F^{+}_j &= \{ |\ln Z^{k,n}(j,0) - \alpha(n-k,j) | \le n^{\theta} \}, \\
  F^-_j   & = \{ |\ln Z^{0,k}(0,j) - \alpha(k,j) | \le n^{\theta} \}, \\
  F &= \{ |\ln Z^{0,n}(0,0) - \alpha(n,0) | \le  n^{\theta} \},\\
 \hat F&= F\cap \bigcap_{|j|\le Rn} F^{+}_j \cap \bigcap_{|j|\le Rn} F^{-}_j,
\end{split}
\end{equation*}
where 
\begin{equation}
\label{eq:space-time-shape}
\alpha(n,x) = \alpha_0n - \frac{x^2}{2n}.
\end{equation}
Let  $\hat{k} = \min\{k,n-k\}$.  By Theorem~\ref{thm:concentration-of-free-energy},
if $\hat{k}\ge n/4$, then
for some $\theta'\in(\beta',2\theta - 1)$ and sufficiently large $n$,
\begin{equation}
  \label{eq:3}
  \Pp(\hat F) \le (4Rn+3)e^{-n^{\theta'}}.
\end{equation}
On the event
$\Lambda = A\cap D \cap \hat F$, we can control $\mu_{0,0}^{0,n}$.
From (\ref{eq:control-for-A}), (\ref{eq:1}), (\ref{eq:3}) we get
\begin{equation*}
  \Pp (\Lambda^c) \le e^{-r_1n } + 2\cdot e^{-r_2n} + e^{-k_1 \sqrt{n}} + (4Rn+3)e^{-n^{\theta'}}
  \le e^{-n^{\beta'}}.
\end{equation*}
Now it suffices to show that for sufficiently large $n$,
\begin{equation}
  \label{eq:Lambda-means-straightness}
  \Lambda \subset \bigl\{ \mu_{0,0}^{0,n} \{ |\gamma_k| \ge n^{1-\delta} - 1 \}
  < e^{-n^{\alpha'}} \bigr\}.
\end{equation}
On $A$, we have
\begin{equation}
  \label{eq:5}
     \mu_{0,0}^{0,n}\{ |\gamma_k| \ge n^{1-\delta}-1 \} 
\le \mu_{0,0}^{0,n} \big( \{ |\gamma_k| \ge n^{1-\delta}-1 \} \cap B \cap C \big) + 3\cdot 2^{-n}.
\end{equation}
For every path in $B \cap C$, we can move $\gamma_k$ to $[\gamma_k]$ to obtain a new path.
The difference between integrating over the old paths and over the new ones can be estimated as follows.
Since all the paths are in $C$, the kinetic action will not change by more than $2R'\sqrt{n}+1$, and
since we are on the set $D$, the potential will change by at most $2\max_{|j| \le Rn}
F_{k}^{*}(j)\le 2 \sqrt{n}$.
The integral over modified paths can be viewed as a sum of partition functions over integer endpoints.
Therefore, we obtain
\begin{equation*}
  \mu^{0,n}_{0,0}\big\{ \{|\gamma_k| \ge n^{1-\delta} -1 \} \cap B \cap C \big\}
  \le e^{(2R'+3) \sqrt{n}} \bigl( Z^{0,n}_{0,0} \bigr)^{-1} \sum_{j \in J } Z^{0,k}(0,j) Z^{k,n}(j,0),
\end{equation*}
where $J = [n^{1-\delta}-1, Rn]_\Z$.
On $\hat F$, we have
\begin{align*}
&\quad   \mu_{0,0}^{0,n} \big\{ \{ |\gamma_k| \ge n^{1-\delta} -1\} \cap B \cap C \big\} \\
&\le e^{(2R'+3) \sqrt{n} + n^{\theta} - \alpha_0 n} \sum_{j \in J} 
    \exp \Big(2n^{\theta} +  \alpha_0k - \frac{j^2}{2k} + \alpha_0 (n-k)- \frac{j^2}{2(n-k)} \Big) \\
  & \le e^{ (2R'+3) \sqrt{n} + 3n^{\theta} } \sum_{|j| \ge n^{1-\delta}-1} \exp \Big( -\frac{j^2}{2k}- \frac{j^2}{2(n-k)} \Big) \\
  &\le e^{(2R'+3) \sqrt{n} + 3n^{\theta}} \sum_{|j| \ge n^{1-\delta}-1} \exp \Big(-\frac{2j^2}{n}\Big) \\
&\le k_2\exp\big( (2R'+3) \sqrt{n} + 3n^{\theta} - 2n^{1-2\delta} \big)\\
  &\le \exp (-k_3 n^{1-2\delta} )
\end{align*}
for some positive constants $k_2$, $k_3$ and sufficiently large $n$.
In the last inequality, we use $\theta <1-2\delta$.
We obtain~\eqref{eq:Lambda-means-straightness} from this and (\ref{eq:5}), and the lemma follows.
\end{proof}

Let us introduce
\begin{equation*}
  \mathcal{C}(n,x,L) := \left\{ (m,y) \in \ZR:\ m \in \{n+1,...,2n\},\ \left| y - \frac{m}{n}x \right| \le L
  \right\},
\end{equation*}
a parallelogram of width $2L$ with one pair of sides parallel to the $x$-coordinate axis and the other one parallel to $[(n,x),(2n,2x)]$.
We define the lateral sides by 
\begin{equation*}
  \partial_S^{\pm} \mathcal{C}(n,x,L) := \left\{
    (m,y) \in \ZR:\ m \in \{ n+1, ...,2n \} \text{ and } y - \frac{m}{n}x = \pm L
  \right\},
\end{equation*}
and let $\partial_S \mathcal{C}(n,x,L) = \partial_S^{+} \mathcal{C}(n,x,L) \cup \partial_S^{-}
\mathcal{C}(n,x,L)$.
Let us define $H(n,x)$ to be 
\begin{equation*}
  H(n,x) = \left\{ \gamma: \gamma_n \in [x,x+1],\ \exists m \in \{ n+1, ..., 2n\},\ \left| \gamma_m - \frac{m}{n}x \right| > 6n^{1-\delta} \right\}.
\end{equation*}
Then for $N \ge 2n$, $\mu_{0,\nu}^{0,N}\big( H(n,x) \big)$ is the polymer measure of the paths that start at $(0,0)$, 
pass near $(n,x)$ but exit the parallelogram $\mathcal{C}(n,x,6n^{1-\delta})$ 
at a time between $n+1$ and $2n$.
The constant 6 is chosen to be compatible with Lemma
\ref{lem:concentration-of-polymer-measure-in-delta-rectangle}. 

\begin{lemma}
  \label{lem:polymer-measure-of-Hnx-is-small}
Given $\alpha \in (0, 1-2\delta)$,
if $A_{0,y}^{0,m}(\alpha)$ holds for all $(m,y) \in \mathcal{C}(n,x,6n^{1-\delta}) \cap \{ m \ge 3n/2 \}$, then for any $N \ge 2n$ and any 
terminal measure $\nu$,
  \begin{equation}
    \label{eq:polymer-measure-of-Hnx-is-small}
    \mu_{0,\nu}^{0,N} \big( H(n,x) \big) \le  2n\exp(-n^{\alpha} ).
  \end{equation}
\end{lemma}
\begin{proof}
  Let
\begin{align*}
  B_k^+ &= \big\{ \gamma:\ \gamma_n \in [x,x+1],\ \gamma_k -\frac{k}{n}x > 4n^{1-\delta} \big\}, \quad 3n/2 \le k \le 2n, \\
    B_k^- &= \big\{ \gamma:\ \gamma_n \in [x,x+1],\ \gamma_k -\frac{k}{n}x < -4n^{1-\delta} \big\}, \quad 3n/2 \le k \le 2n , \\
  C_k^+ &= \big\{ \gamma:\ \gamma_n \in [x,x+1],\ \gamma_{2n} -2x \le 4n^{1-\delta},\ \gamma_k - \frac{k}{n}x > 6n^{1-\delta} \big\}, \quad n < k < \frac{3n}{2}, \\
  C_k^{-} &= \big\{ \gamma:\,\gamma_n \in [x,x+1],\,\gamma_{2n} -2x \ge - 4n^{1-\delta},\,\gamma_k - \frac{k}{n}x <- 6n^{1-\delta} \big\},\  n< k < \frac{3n}{2}.  
\end{align*}
Then $H(n,x) \subset \bigcup_k B_k^{\pm} \cup \bigcup_k C_k^{\pm}$.
We need to estimate $\mu_{0,\nu}^{0,N}(B_k^{\pm})$ and $\mu_{0,\nu}^{0,N}(C_k^{\pm})$.

First, we look at $B_k^+$.
By monotonicity, if  $ y  > kx/n + 4n^{1-\delta} := f(k)$, then 
\begin{equation*}
  \mu_{0,y}^{0,k}(B_k^+)\le \mu_{0,y}^{0,k}\{\gamma: \gamma_n \le x+1 \} \le \mu_{0,f(k)}^{0,k}\{
  \gamma: \gamma_n \le x+1\}.
\end{equation*}
Since
\begin{equation*}
 \big[ (0,0),(k,f(k) ) \big]_n - (x +1) = \frac{4n^{2-\delta}}{k} - 1 \ge \frac{4n^{2-\delta}}{2n} - 1 > (2n)^{1-\delta} \ge k^{1-\delta}
\end{equation*}
for sufficiently large $n$,
we have
\begin{equation*}
  \{ \gamma:  \gamma_n \le x+1 \}
  \subset
  \Big\{ \gamma  : \max_{j \in [k/4, 3k/4]} \Big|\gamma_j - \big[(0,0),(k,f(k))\big]_{j}\Big| \ge k^{1-\delta} \Big\}.
\end{equation*}
By assumption, $A_{0,f(k)}^{0,k}$ holds. Therefore, if $y \ge f(k)$, then
\begin{equation*}
  \mu_{0,y}^{0,k}(B_k^+)\le e^{-k^{\alpha}} \le e^{-n^{\alpha}},
\end{equation*}
and hence 
\begin{equation*}
  Z^{0,k}(0, y,B_k^+) \le \exp( -n^{\alpha} ) Z^{0,k}(0,y).
\end{equation*}
Since every path in $B_k^+$ visits some $y \ge f(k)$ at time $k$, we have
\begin{align*}
  Z^{0,N}(0, z, B_k^+)
  &= \int_{y \ge f(k)} Z^{0,k}(0,y,B_k^+) Z^{k,N}(y,z) d y \\
  &\le \int_{y \ge f(k)} \exp(-n^{\alpha}) Z^{0,k}(0,y) Z^{k,N}(y,z) dy \\
  & \le \exp(-n^{\alpha}) Z^{0,N}(0, z).
\end{align*}
Therefore, $\mu_{0,z}^{0,N}(B_k^+) \le \exp(-n^{\alpha})$ for all $z$.
Integrating over $z$ we obtain
\begin{equation*}
  \mu_{0,\nu}^{0,N}(B_k^+) = \int \mu_{0,z}^{0,N}(B_k^+) \nu(dz) \le \exp(-n^{\alpha}).
\end{equation*}
Next we look at $C_k^+$.
By monotonicity, if $y \le f(2n)$, then
\begin{equation*}
  \mu_{0,y}^{0,2n}(C_k^+) \le \mu_{0,y}^{0,2n}\{\gamma: \gamma_k \ge g(k) \}
  \le \mu_{0,f(2n)}^{0,2n}\{ \gamma: \gamma_k \ge g(k) \},
\end{equation*}
where $g(k) = kx/n+6n^{1-\delta}$.
Since
\begin{equation*}
  \big[(0,0), (2n,f(2n))\big]_k - g(k) = 2k n^{-\delta} - 6n^{1-\delta} <  -(2n)^{1-\delta},
\end{equation*}
we have  
\begin{equation*}
  \{ \gamma : \gamma_k \ge g(k) \}
  \subset
  \Big\{ \gamma  : \max_{j \in [n/2, 3n/2]} \Big|\gamma_j -
  \big[(0,0),(2n,f(2n) )\big]_j \Big| \ge (2n)^{1-\delta} \Big\}.
\end{equation*}
Since $A_{0,f(2n)}^{0,2n}$ holds, we have 
\begin{equation*}
\mu_{0,y}^{0,2n}(C_k^+) \le \exp\big(-(2n)^{\alpha}\big) \le \exp( -n^{\alpha} )
\end{equation*}
for all $y \le f(k)$.
Similarly to the case of $B_k^+$, it follows that $\mu_{0,\nu}^{0,N}(C_k^+) \le \exp(-n^{\alpha})$.

The same argument shows that 
\begin{equation*}
  \mu_{0,\nu}^{0,N}(B_k^-) \le \exp(-n^{\alpha}), \quad
  \mu_{0,\nu}^{0,N}(C_k^-) \le \exp( -n^{\alpha}).
\end{equation*}
We can now write 
\begin{align*}
  & \quad   \mu_{0,\nu}^{0,N}\big(H(n,x)\big) \\
  & \le \sum_{n< k < 3n/2} \left[
  \mu_{0,\nu}^{0,N}(C_k^+) + \mu_{0,\nu}^{0,N}(C_k^-) \right] + \sum_{3n/2 \le k \le 2n} \left[ \mu_{0,\nu}^{0,N}(B_k^+) +
     \mu_{0,\nu}^{0,N}(B_k^-) \right] \\
  &    \le 2n \exp (-n^{\alpha}),
\end{align*}
and the lemma follows.
\end{proof}

The following lemma is a geometric fact.
\begin{lemma}
  \label{lem:geometric-condition-of-staying-in-a-cone}
  Let $N/2 \ge M$.
  If a path $\gamma\notin H(n,[\gamma_n])$ (where $[\cdot]$ denotes the integer part) for all $n\in [M,N/2]_\Z$, then
  \begin{equation*}
    \gamma^{\text{out}}(n) \subset \mathrm{Co}(n, \gamma_n, Qn^{-\delta}),\quad n\in [M,N/2]\cap\Z,
  \end{equation*}
  for  a constant $Q$ depending only on $\delta$.
\end{lemma}
\begin{proof}
  Let $k = [\log_2 N/n] $  so that $2^k n \le N < 2^{k+1}n$.
  Since $N \ge 2n$, we have $k \ge 1$.
  Let $x_l = \gamma_{2^ln}$ and $i_l = [x_l]$ for $ 0 \le l \le k$.

If $0 \le l' \le k-1$, then,
  since $\gamma \not\in H(2^{l'}n, i_{l'})$, we have
  \begin{equation*}
    \bigg| \frac{i_{l'}}{2^{l'}n} - \frac{x_{l'+1}}{2^{l'+1}n} \bigg|
    \le \frac{6(2^{l'}n)^{1-\delta}}{2^{l'+1}n} =  3\cdot 2^{-l'\delta}n^{-\delta} .
  \end{equation*}
Since $|i_{l'+1} - x_{l'+1}| <  1$, we have
\begin{equation*}
  \left| \frac{i_{l'}}{2^{l'}n}  - \frac{i_{l'+1}}{2^{l'+1}n}\right|
  \le 3 \cdot 2^{-l'\delta} n^{-\delta} + \frac{1}{2^{l'+1}n}
  < 3 \cdot 2^{-l'\delta} n^{-\delta} + 2^{-l'\delta}n^{-\delta} = 
  4\cdot 2^{-l'\delta}n^{-\delta}.
\end{equation*}
Summing over $0 \le l' \le l-1$ and using the triangle inequality we obtain
\begin{equation}\label{eq:distance-between-diadic-points}
  \left| \frac{i_0}{n} - \frac{i_{l}}{2^{l}n} \right|
  \le  4 \cdot n^{-\delta}\sum_{l'=0}^{l-1} 2^{-l'\delta}
  \le c n^{-\delta}
\end{equation}
for all $0 \le l \le k$, where $c = 4\sum_{l'=0}^{\infty} 2^{-l'\delta}$.

If $n \le m \le 2^kn $, let $l = [\log_2 \frac{m-1}{n}]$ so that $2^ln < m \le 2^{l+1}n$.
Since $\gamma \not\in H(2^ln, i_l)$,  we have
\begin{equation*}
  \left| \frac{i_l}{2^ln} - \frac{\gamma_m}{m} 
  \right| \le
  \frac{6(2^ln)^{1-\delta}}{m}
\le 6  n^{-\delta}.
\end{equation*}
Combining this with (\ref{eq:distance-between-diadic-points}) we find that
\begin{equation}
\label{eq:distant-n-to-m-small}
  \left| \frac{i_0}{n} -  \frac{\gamma_m}{m} \right| \le (c+6) n^{-\delta}.
\end{equation}

If  $2^k n < m \le N$, let $\bar{m}= [\frac{m+1}{2}]$ so that $n \le \bar{m}\le 2^kn$ and $\bar{m} < m \le
2\bar{m}$.
Then $\gamma \not\in H(\bar{m}, [\gamma_{\bar{m}}])$ implies
\begin{equation*}
  \left| \frac{[\gamma_{\bar{m}}]}{\bar{m}} - \frac{\gamma_m}{m} \right|
  \le \frac{6(\bar{m})^{1-\delta}}{m}
  \le 6\bar{m}^{-\delta} 
  \le 6 n^{-\delta}.
\end{equation*}
Combining this with (\ref{eq:distant-n-to-m-small}) we obtain
\begin{align}\label{eq:distance-n-to-m-big}
 \left| \frac{i_0}{n} - \frac{\gamma_{m}}{m} \right|
  &\le \left| \frac{i_0}{n} - \frac{\gamma_{\bar{m}}}{\bar{m}}\right|  +
    \left| \frac{\gamma_{\bar{m}}}{\bar{m}} - \frac{[\gamma_{\bar{m}}]}{\bar{m}} \right|    +
    \left| \frac{[\gamma_{\bar{m}}]}{\bar{m}} - \frac{\gamma_m}{m} \right|
  \\ \nonumber
  &\le (c+6)n^{-\delta} + \frac{1}{\bar{m}} + 6n^{-\delta} 
  \le (c+13) n^{-\delta}.
\end{align}
Clearly,
\begin{equation}\label{eq:move-initial-point}
  \left| \frac{\gamma_n}{n} - \frac{i_0}{n} \right|  \le n^{-\delta}.
\end{equation}
Combining (\ref{eq:distant-n-to-m-small}), (\ref{eq:distance-n-to-m-big}), and
(\ref{eq:move-initial-point}) we obtain that for every $m\in [n, N]_{\Z}$,
\begin{equation*}
   \left| \frac{\gamma_n}{n}  - \frac{\gamma_m}{m} \right| \le (c+14)n^{-\delta}.
\end{equation*}
This completes the proof, with $Q=c+14$.
\end{proof}

\bigskip

\begin{proof}[Proof of Theorem \ref{thm:delta-straightness-for-polymer-measure}]
  We fix $\delta \in (0,1/4)$, $v > 0$, $\alpha \in (0,1-2\delta)$, and $\beta \in (0,1-4\delta)$.
  Let $\alpha' = (\alpha + 1-2\delta)/2$ and $\beta' = (\beta + 1-4\delta)/2$.
  Then Lemma~\ref{lem:concentration-of-polymer-measure-in-delta-rectangle} implies that for
  sufficiently large $n$,
  \begin{equation*}
    \Pp \Big( \big(\bar{A}_{0,l}^{0,n}(\alpha')\big)^c\Big) \le \exp(-n^{\beta'}),  \quad \forall l \in \Z.
  \end{equation*}
   Let $v'= v+ 7$.
  We have 
  \begin{equation*}
    \sum_{m \ge n} \sum_{|l| \le v'm+1} \Pp \Big( \big (\bar{A}_{0,l}^{0,m}(\alpha')\big)^c  \Big)
    \le \sum_{m \ge n} (2v'm+3) \exp(-m^{\beta'}) < e^{-n^{\beta}}
  \end{equation*}
for sufficiently large $n$. Therefore, if we define
\begin{equation*}
\Omega_1^{(n)} = \bigcap_{m \ge n} \bigcap_{|l| \le v'm+1} \bar{A}_{0,l}^{0,m}(\alpha')
\end{equation*}
depending only on $v$, $\delta$ and $\alpha$, then $\Pp (\Omega_1^{(n)}) \ge 1-e^{-n^{\beta}}$ for large $n$.

Now let us fix any  $\omega \in \Omega_1^{(n)}$. We have then $\omega \in A^{0,m}_{0,x}(\alpha')$ for all $|x| \le v'm$ and $m \ge n$.
Let us also fix $N > 2n$ and a terminal measure $\nu$.
We want to show that~(\ref{eq:delta-straightness-for-polymer-measure}) holds true.
  
If $(k,z) \in \Co(v) \cap \{k \ge n\}$, then $(k, [z] ) \in \Co(v+1)$, so $\mathcal{C}(k,[z],6n^{1-\delta}) \subset \Co(v') $.
Therefore, by Lemma \ref{lem:polymer-measure-of-Hnx-is-small}, for $\omega \in \Omega_1^{(n)}$ and such $(k,z)$, we have
  \begin{equation*}
    \mu_{0,\nu}^{0,N}\big( H(k,[z]) \big) \le 2k \exp(-k^{\alpha'}).
  \end{equation*}
By Lemma \ref{lem:geometric-condition-of-staying-in-a-cone}, the set of paths
  \begin{equation*}
    \left\{
      \gamma : \exists (k,\gamma_k) \in\Co(v), N/2 \ge k \ge n, \gamma^{\text{out}}(k) \not\subset \Co(k,\gamma_k,Qk^{-\delta})
      \right\}
  \end{equation*}
  is contained in
  \begin{equation*}
    \bigcup_{N/2\ge k\ge n, y = [z], |z| \le kv } H(k,y).
  \end{equation*}
  Therefore,
\begin{align*}
  &\quad \mu_{0,\nu}^{0,N} \left\{
    \gamma : \exists (k,\gamma_k) \in \mathrm{Co}(v), N/2 \ge k \ge n, \gamma^{\text{out}}(k) \not\subset \mathrm{Co}(k,\gamma_k,Qk^{-\delta})
    \right\} \\
  &\le \sum_{N/2 \ge k\ge n, y = [z], |z| \le kv} \mu_{x,\nu}^{0,N}\big(  H(k,y) \big) \\
  &\le \sum_{k=n}^{\infty}  (2kv+2) \cdot 2k \exp(-k^{\alpha'})\le \exp(-n^{\alpha})
\end{align*}
for sufficiently large $n$.
This completes the proof.
\end{proof}

\subsection{Tightness of polymer measures}\label{sec:compactness-of-polymer-measures}
In this section, we will use Theorem~\ref{thm:delta-straightness-for-polymer-measure}  
on $\delta$-straightness to prove tightness
of families of polymer measures constructed for a fixed endpoint and a well-behaved sequence of terminal measures.

 Let $(m,x)\in\ZR$. For a sequence  $(n_k)$  and  a family of Borel probability measures $(\nu_k)$ on $\R$, 
we say that the family of polymer measures $ \big( \nu_{x,\nu_k}^{m,n_k} \big)$ is tight if for every
$n\ge m$ and every 
$\varepsilon>0$, there is a compact set $K \subset \R^{n-m+1}$ such that for all $n_k > n$,
\begin{equation*}
\mu_{x,\nu_k}^{m,n_k} \pi_{m,n}^{-1} (\R^{n-m+1}\setminus K) \le \varepsilon.
\end{equation*}

\begin{theorem}
  \label{thm:tightness-of-polymer-measure} There is a full measure set $\Omega'$ such that for 
all~$\omega \in \Omega'$ the following holds: if a  sequence  $(n_k)$ and a family of probability measures $(\nu_k)$ satisfy
\begin{equation}
\label{eq:tightness-for-terminal-measures}
\lim_{c \to \infty} \sup_k \nu_k \big( [-cn_k,cn_k]^c ) = 0,
\end{equation}
then  for all $(m,x)\in\Z\times\R$,
$\big(\mu_{x,\nu_k}^{m,n_k}\big)$ is tight.
\end{theorem}

We will derive this theorem from the following result:

\begin{theorem}
  \label{thm:all-space-time-point-compactness}
 There is a full measure set $\Omega'$ such that for every $\omega \in \Omega'$ the
 following holds:  if $(m,x) \in \Z \times \R$,  $v'\in \R$, and $0\le u_0 <u_1$,
  then there is a random constant
\begin{equation*}
n_0= n_0\big(\omega,m,[x], [|v'|+u_1], [(u_1-u_0)^{-1}] \big)
\end{equation*}
(where $[\cdot]$ denotes the integer part) such that
\begin{equation}
  \label{eq:remote-control-all-space-time-all-space-time} 
    \mu_{x,\nu}^{m,N} \big\{ \gamma: |\gamma_{m+n}  - v'n| \ge u_1n \big\}
  \le 4 \nu \big( [ (v'-u_0)N, (v'+u_0)N]^c \big) + 6e^{-\sqrt{n}}
\end{equation}
and 
\begin{multline}
  \label{eq:compactness-control-by-terminal-measure-all-space-time}
  \mu_{x,\nu}^{m,N} \big\{ \gamma: \max_{1 \le i \le n}|\gamma_{m+i} - v'i| \ge (u_1+R+1)n \big\} \\
  \le 4 \nu \big( [ (v'-u_0)N, (v'+u_0)N]^c \big) + 8e^{-\sqrt{n}},
\end{multline}
hold true for any terminal measure $\nu$ and $(N-m)/2 \ge n \ge n_0$.
Here, we use~$R$ that has been introduced in Lemma \ref{lem:large-deviation-for-Z-bar}.
\end{theorem}

\begin{proof}[Derivation of Theorem~\ref{thm:tightness-of-polymer-measure} from Theorem~\ref{thm:all-space-time-point-compactness}]
  We take $\Omega'$ from the statement of Theorem~\ref{thm:all-space-time-point-compactness} and fix an arbitrary  $\omega\in\Omega'$.
 Given any $\varepsilon>0$,  by (\ref{eq:tightness-for-terminal-measures}), there is $c$ such that
$\nu_{k}\big( [-cn_k,cn_k]^c \big) \le \varepsilon$
for all $k$.
Choosing $v'=0$, $u_0=c$, $u_1=2c$ in Theorem~\ref{thm:all-space-time-point-compactness}, we see that if 
\begin{equation*}
n \ge n_0\big(\omega, m,[x], [2c], [c^{-1}] \big) \vee \ln^2 \varepsilon,
\end{equation*}
then, due to (\ref{eq:compactness-control-by-terminal-measure-all-space-time}),
\begin{equation*}
  \mu_{x,\nu_k}^{m,n_k} \Bigl\{ \gamma: \max_{m \le i \le m+n} |\gamma_i| \ge (2c+R+1)n \Bigr\}
  \le 4 \nu_k \big( [-cn_k, cn_k]^c \big) + 8 e^{-\sqrt{n}}
  \le 12 \varepsilon
\end{equation*}
for all $n_k \ge m+ 2n$, and tightness follows.
\end{proof}

Now we turn to the proof of Theorem \ref{thm:all-space-time-point-compactness}.

Using the constant $R$ introduced in Lemma \ref{lem:large-deviation-for-Z-bar}, we define the events
\begin{equation*}
B_{x,y}^{m,n} = \Big\{ \mu_{x,y}^{m,n} \big\{  \gamma: \max_{m\le k \le n} \big| \gamma_k - [(m,x),(n,y)]_k\big|
    \ge (R+1)(n-m) \big\}  \le 2^{-n+1} \Big\}
\end{equation*}
and 
\begin{equation*}
\bar{B}_{p,q}^{m,n} = \bigcap \big\{  B_{x,y}^{m,n}: x \in [p,p+1], y \in [q,q+1]\}.
\end{equation*}
These are measurable events since, by continuity, the intersection on the right-hand side may be restricted to
 rational points $x,y$.
 Also, for fixed $m$ and $n$, all $\bar{B}_{p,q}^{m,n}$ have the same probability.

\begin{lemma}
  \label{lem:uniform-small-probability-outside-a-linear-box}
If $n$ is   sufficiently large, then
\begin{equation*}
\Pp   \Big(  \big( \bar{B}_{0,0}^{0,n}\big)^c \Big) \le 2e^{-r_1n}.
\end{equation*}
\end{lemma}

\begin{proof}
The lemma follows from Lemma \ref{lem:large-deviation-for-Z-bar} and the fact that
$\mu_{0,0}^{0,n}\preceq\mu_{x,y}^{0,n}\preceq \mu_{1,1}^{0,n}$ for $x,y\in[0,1]$. 
\end{proof}

\begin{lemma}
  \label{lem:delta-straighness-imply-compactness}
Suppose $\delta \in (0,1/4)$, $n \in \N$, $v' \in \R$, $0\le v_0<v_1$, and $v_1 > v_0 + Qn^{-\delta}$.
If $\omega$ satisfies the following properties:
\begin{enumerate}
  \item\label{item:6} $\omega \in A_{0,z}^{0,m}(\frac{1}{2})$ for all  $m \ge n$ and $|z-v'm| \le v_0m$,
\item\label{item:7} $\omega \in B_{0,z}^{0,m}$ for all $m \ge n$ and $|z-v'm| \le v_1m$,
\item\label{item:5} $\omega \in \Omega_1^{(n)}(\delta,|v'|+v_1+R+2,\frac{1}{2})$,
\end{enumerate}
then for any $N \ge 2n$ and terminal measure $\nu$, 
\begin{equation*}
  \mu_{0,\nu}^{0,N} \big\{ \gamma: |\gamma_n-v'n| \ge v_1n \big\}
  \le 2\nu \big([(v'-v_0)N, (v'+v_0)N]^c\big) + 3e^{-\sqrt{n}}.
\end{equation*}
\end{lemma}

\begin{proof}  We will only give a proof for the case of $v'=0$. The extension to a general $v'$ is straightforward.

Let $v_2 = v_1 + R + 2$ and fix $N \ge 2n$.
Define $T=T(\gamma)=\inf\{ m \ge n: (m,\gamma_m) \in \Co(v_1)\}\wedge N$. 
Since  $|\gamma_n| \ge nv_1$ implies $(T-1,\gamma_{T-1}) \not\in \Co(v_1)$, we can
partition the set of paths
$B = \{ \gamma: |\gamma_n| \ge nv_1 \}$
into the union of $A_i$, where
\begin{align*}
  A_1 &= \big\{ \gamma: T \le [N/2],\ (T-1,\gamma_{T-1}) \in \Co(v_2)\setminus \Co(v_1),\ \gamma_N \in
        [- v_0N, v_0N]   \big\}, \\
  A_2 &= \big\{ \gamma: T \le [N/2],\ (T-1,\gamma_{T-1}) \in \Co(v_2)\setminus\Co(v_1),\ \gamma_N \in [-v_0N, v_0N]^{c}   \big\}, \\
  A_3 &= \big\{ \gamma: T \le [N/2],\ (T-1,\gamma_{T-1}) \not\in \Co(v_2) \big\}, \\
  A_4 &= \big\{ \gamma: T > [N/2],\ \gamma_N \in [-v_0N, v_0N] \big\},\\
  A_5 &= \big\{ \gamma: T > [N/2],\ \gamma_N \in [-v_0N, v_0N]^c \big\}. 
\end{align*}
Let us estimate weights assigned to sets $A_i$ by the polymer measure.
Inequalities
\begin{equation}
  \label{eq:A2-and-A5}
  \mu_{0,\nu}^{0,N}(A_2) \le \nu \big( [-v_0N, v_0N]^c \big),
  \quad
    \mu_{0,\nu}^{0,N}(A_5) \le \nu \big( [-v_0N, v_0N]^c \big)
\end{equation}
are obvious. Since
\begin{equation*}
A_1
  \subset
  \left\{
    \gamma : \exists (k,\gamma_k) \in \Co(v_2),\ N/2 \ge k \ge n,\ \gamma^{\text{out}}(k) \not\subset \Co(k,z,Qk^{-\delta})
    \right\}
\end{equation*}
(the condition is satisfied by $k=T-1$), and $\omega \in \Omega_1^{(n)}(\delta,v_2,\frac{1}{2})$,
we have
\begin{equation}
  \label{eq:A1}
  \mu_{0,\nu}^{0,N}(A_1) \le e^{-\sqrt{n}} .
\end{equation}
Since $A_3 \subset \big\{|\gamma_T| \le v_1T \big\}$ and for $n\ge R+2$ and $|z| \le v_1m$ we have
\begin{align*}
  \mu_{0,z}^{0,m}(A_3 \cap \{T = m\}) &\le
  \mu_{0,z}^{0,m} \Big\{ \gamma: \max_{1 \le k \le m} \big| \gamma_k - [(0,0),(m,z)]_k \big| \ge (R+2)(m-1) \Big\}
\\ &\le
  \mu_{0,z}^{0,m} \Big\{ \gamma: \max_{1 \le k \le m} \big| \gamma_k - [(0,0),(m,z)]_k \big| \ge (R+1)m \Big\}
\end{align*}
(the condition is satisfied for $k=T-1$),
$\omega \in B_{0,z}^{0,m}$ will imply
\begin{equation*}
\mu_{0,z}^{0,m} \bigl( A_3 \cap \{ T=m \} \bigr) \le 2^{-m+1}, 
\end{equation*}
and hence
\begin{equation*}
\mu_{0, \nu}^{0,N} \bigl( A_{3} \cap \{ T=m \} \bigr) \le 2^{-m+1}.
\end{equation*}
Therefore,
\begin{equation}
  \label{eq:A3} 
  \mu_{0,\nu}^{0,N}(A_3) \le  \sum_{m=n+1}^{[N/2]} 2^{-m+1}  \le 2^{-n+1}.
\end{equation}
Since $A_4 \subset \big\{ |\gamma_N| \le v_0N \big\}$ and $T > [N/2]$ implies that
$([N/2],\gamma_{[N/2]}) \not\in \Co(v_1)$, we have
\begin{equation*}
  \mu_{0,z}^{0,N} (A_4)
  \le
 \mu_{0,z}^{0,N}   \Big\{ \gamma: \max_{N/4 \le k \le 3N/4} |\gamma_k - [(0,0),(N,z)]_k| \ge \frac{Q}{3}N^{1-\delta} \Big\}
\end{equation*}
for $|z| \le v_0N$
(the condition is satisfied for $k = [N/2]$ ).
Therefore, $\omega \in A^{0,N}_{0,z}(\frac{1}{2})$ will imply (due to the choice of~$Q$ in the proof of Lemma~\ref{lem:geometric-condition-of-staying-in-a-cone})
\begin{equation*}
\mu_{0,z}^{0,N} (A_4  ) \le e^{-\sqrt{N}} \le e^{-\sqrt{n}}
\end{equation*}
 and hence
 \begin{equation}
   \label{eq:A4}
  \mu_{0,\nu}^{0,N}(A_4) = \int_{-v_0N}^{v_0N} \mu_{0,z}^{0,N}(A_4) \nu(dz) \le e^{-\sqrt{n}}.
 \end{equation}
The conclusion of the lemma follows from  (\ref{eq:A2-and-A5}), (\ref{eq:A1}), (\ref{eq:A3}), and (\ref{eq:A4}).
\end{proof}

\begin{lemma}
\label{lem:sets-for-remote-control-to-holds}
Let $\delta \in (0,1/4)$,  $v > 0$ and $(m,q) \in \Z \times \Z$.
Then there is an event $\Omega_2^{(n)} = \Omega_2^{(n)}(\delta,v, m, q)$ with the following properties.
\begin{enumerate}
\item\label{item:12} For all  $\beta \in (0,
  1-4\delta)$ and sufficiently large $n$ (depending on $\beta$), we have
\begin{equation}
\label{eq:Omega-2-small-probability}
\Pp \big( (\Omega_{2}^{(n)} )^{c} \big ) \le e^{-n^{\beta}}.
\end{equation} 
\item\label{item:15} Let $v' \in \R$, $0 \le v_0<v_1$.
  If $\omega \in \Omega^{(n)}_2$, $|v'| +
  v_1\le v$, $Qn^{-\delta} <  v_1-v_0$, $N/2 \ge s \ge n$, then for any terminal measure $\nu$,
  \begin{multline}
  \label{eq:remote-control}
    \mu_{q,\nu}^{m,m+N} \big\{ \gamma: |\gamma_{m+s} - q - v'n| \ge v_1n \big\} \\
  \le 2\nu \big( [ q+(v'-v_0)N, q+(v'+v_0)N]^c \big) + 3e^{-\sqrt{s}}
\end{multline}
and
\begin{multline}
  \label{eq:compactness-control-by-terminal-measure}
  \mu_{q,\nu}^{m,m+N} \big\{ \gamma: \max_{1 \le i \le s}|\gamma_{m+i} - q - v'i| \ge (v_1+R+1)s \big\} \\
  \le 2 \nu \big( [ q+(v'-v_0)N, q+(v'+v_0)N]^c \big) + 4e^{-\sqrt{s}}.
\end{multline}
\end{enumerate}
\end{lemma}

\begin{proof}
Without loss of generality, we can assume $(m,q) = (0,0)$. Let 
\begin{equation*}
  \Omega^{(n)}_2  = \Omega_1^{(n)}(\delta,v+R+2,1/2) \cap
  \bigg[  \bigcap_{s \ge n, |p| \le vs+1} \Big( \bar{A}_{0,p}^{0,s}(1/2) \cap \bar{B}_{0,p}^{0,s}
  \Big)
  \bigg].
\end{equation*}
We claim that such events have the desired properties.
In fact,  (\ref{eq:Omega-2-small-probability}) follows from Theorem \ref{thm:delta-straightness-for-polymer-measure}, Lemma
 \ref{lem:concentration-of-polymer-measure-in-delta-rectangle}, and Lemma
 \ref{lem:uniform-small-probability-outside-a-linear-box}; (\ref{eq:remote-control})
 follows from Lemma \ref{lem:delta-straighness-imply-compactness}.
It remains to show that if $\omega \in \Omega^{(n)}_2$, then
(\ref{eq:compactness-control-by-terminal-measure}) holds.

Fix $s \ge n$ and assume $\omega \in \Omega^{(n)}_2$.
Let 
\begin{equation*}
A = \big\{ \gamma : \ \max_{1 \le i \le s} |\gamma_i-v'i| \ge (v_1+R+1)s \big\}.
\end{equation*}
If $|z-v's| \le v_1s$,  then 
\begin{equation*}
\mu_{0,z}^{0,s} (A )
\le
\mu_{0,z}^{0,s}  \{ \gamma: \max_{1 \le i \le s} | \gamma_i - [(0,0), (s,z)]_i | \ge (R+1)s\}.
\end{equation*}
Since $\omega \in B_{0,z}^{0,s}$, we have  $\mu_{0,z}^{0,s} (A) \le 2^{-s+1}$ 
for all $|z-v's| \le v_1s $. Therefore,
\begin{equation*}
\mu_{0,\nu}^{0,N} \bigl(  A \cap \{ |\gamma_s - v's| \le v_1s  \} \bigr) \le 2^{-s+1} \le e^{-\sqrt{s}}.
\end{equation*}
Then (\ref{eq:compactness-control-by-terminal-measure}) follows from this and  (\ref{eq:remote-control}).
\end{proof}

\begin{proof}[Proof of Theorem \ref{thm:all-space-time-point-compactness}]
We fix $\delta \in (0, \frac{1}{4})$.
By (\ref{eq:Omega-2-small-probability}) and the Borel--Cantelli Lemma, with probability one, for all
$(m,q) \in \Z \times \Z$ and $M \in \N$, $\Omega_2^{(n)}(\delta, M, m, q)$ hold for $n > n_1$,
where $n_1=n_1(\omega, m,q, M)$ is a random constant depending on $m,q$ and $M$.

  Fix $0 \le u_0 < u_{1}$.
  Let 
\begin{equation*}
v_0 = \frac{2u_0 + u_1}{3}, \quad v_1 = \frac{u_0 + 2u_1}{3}.
\end{equation*}
Suppose $x \in [q,q+1]$ for some $q \in \Z$ and $ |v'| + u_1 \le M$ for some $M \in \N$.
Lemma~\ref{lem:main-monotonicity} implies $\mu_{q,\nu}^{m,N} \preceq \mu_{x,\nu}^{m,N} \preceq
 \mu_{q+1,\nu}^{m,N}$.
 Combined with Lemma \ref{lem:sets-for-remote-control-to-holds}, 
 we see that if 
\begin{equation*}
n \ge   n_1(\omega,m,q,M) \vee n_1(\omega,m,q+1,M)  \vee \Bigl( \frac{3Q}{|u_1-u_0|}
\Bigr)^{\frac{1}{\delta}} \vee \bigg( \frac{3(|q|+1 + M|m|)}{|u_1-u_0|} + |m| \bigg),
\end{equation*}
then for any $N - m \ge 2n $ and any terminal measure $\nu$, we have
\begin{align*}
   &\mu_{x,\nu}^{m,N}  \big\{ \gamma: |\gamma_{m+n} - v'n| \ge u_1n \big\}
    \\ \le& \sum_{p=q}^{q+1} \mu_{p,\nu}^{m,N}  \big\{ \gamma: |\gamma_{m+n} - p-v'n| \ge v_1n)
      \big\} \\
    \le & 2 \sum_{p=q}^{q+1} \nu \bigl( [p+(v'-v_0)(N-m), p+(v' + v_0) (N-m)]^c \bigr) + 6e^{-\sqrt{n}} \\
    \le &4 \nu \bigl( [(v'-u_0)N, (v'+u_0)N]^c \bigr) + 6e^{-\sqrt{n}}.
\end{align*}
Here, the first inequality uses $u_1n \ge v_{1}n + q+1$, which is implied by $n \ge \frac{|q|+1}{u_1-v_1} = \frac{3(|q|+1)}{|u_1-u_0|}$, the third inequality uses
$q+1 + (v'+v_0)(N-m) \ge (v'+u_0) N$, i.e., $q+1+(v_0-u_0)N \ge m (v'+v_0)$, which is implied by $N-m \ge n \ge \frac{mM + |q|+1}{v_0-u_0} = \frac{3(mM+|q|+1)}{|u_1-u_0|}$.
Similarly,
\begin{align*}
&\quad  \mu_{x,\nu}^{m,N} \big\{ \gamma: \max_{1 \le i \le n}|\gamma_{m+i} - v'i| \ge (u_1+R+1)n
\big\} \\
& \le \sum_{p=q}^{q+1} \mu_{p,\nu}^{m,N} \big\{ \gamma: \max_{1 \le i \le n}|\gamma_{m+i} - p-v'i| \ge (v_1+R+1)n
\big\}
\\
& \le 2\sum_{p=q}^{q+1} \nu \bigl( [p + (v'-v_0)(N-m), p+(v'+v_0)(N-m)]^c \bigr) + 8e^{-\sqrt{n}} \\
    &\le 4 \nu \bigl( [(v'-u_0)N, (v'+u_0)N]^c \bigr) + 8e^{-\sqrt{n}}.
\end{align*}
This completes the proof.
\end{proof}

\section{Infinite-volume polymer measures }
\label{sec:infinite-volume-polymer-measure}

In this section, we prove most claims of the Theorem~\ref{thm:thermodynamic-limit} on thermodynamic limits of polymer measures. 
We prove the existence and uniqueness part, and we prove that finite time
horizon polymer measures converge to the infinite volume polymer measures in the sense of weak convergence
of finite-dimensional distributions. We will prove the stronger total variation convergence in 
Section~\ref{sec:global-solutions-and-ratios-of-partition-functions}.

We begin with the following useful statement:

 \begin{lemma}\label{lem:limit-of-polymers-is-a-polymer} For all $\omega\in\Omega$,
 if a sequence of polymer measures has a weak limit, then the limiting measure is also a polymer measure.  
 \end{lemma}
\bpf It is sufficient to prove the statement of the lemma for finite volume polymer measures. 
We need to prove that $\mu^{m,n}_{x,\nu_k}$
weakly converges to $\mu^{m,n}_{x,\nu}$ if $m,n\in\Z$, $x\in\R$, and $(\nu_k)$ is a  sequence of distributions on~$\R$, weakly convergent to a distribution $\nu$. 

It suffices to check that if $f(x_{m+1},\ldots,x_n)=f_{m+1}(x_{m+1})\ldots f_{n-1}(x_{n-1})f_{n}(x_{n})$ for continuous nonnegative functions $f_{m+1},\ldots,f_n$ with bounded support,
then
\begin{multline*}
 \lim_{k\to\infty}\int\mu^{m,n}_{x,\nu_k}(x_m,\ldots, dx_n)f(x_{m+1},\ldots,x_n)\\=\int\mu^{m,n}_{x,\nu}(dx_m,\ldots, dx_n)f(x_{m+1},\ldots,x_n).
\end{multline*}
Since
\[
\int\mu^{m,n}_{x,\nu}(dx_m,\ldots, dx_n)f(x_{m+1},\ldots,x_n)=\int \nu(dx_n) G(x_n),
\]
where
\[
 G(x_n)=\int \mu_{x,x_n}^{m,n}(dx_m,\ldots, dx_n)f(x_{m+1},\ldots,x_n),
\]
we need to show that $G$ is a continuous function.  The latter follows from the definition of $\mu_{x,x_n}^{m,n}$,
continuity of  $Z^{m,n}_{x,x_n}$ (see Lemma~\ref{lem:smoothness-of-partition-function}) and $g(x_{n}-x_{n-1}) f_n(x_{n})$ with respect to~$x_n$,  and the bounded convergence theorem.
\epf

In addition to the terminology and notation from Section~\ref{sec:polymers}, we say that LLN with slope $v\in\R$ holds {\it for an increasing sequence of times} $(n_k)$ and a sequence 
of Borel measures $(\nu_k)$ on $\R$ if
for all $\delta>0$,
\[
 \lim_{k\to\infty} \nu_k([(v-\delta)n_k,(v+\delta) n_k])=1.
\]

\begin{lemma}\label{lem:existence-of-polymer-measures-and-convergence-along-subsequences} For every
  $\omega \in \Omega'$ (introduced in Theorem \ref{thm:all-space-time-point-compactness}) the
  following holds true. 
For any $(m,x)\in \Z \times\R$, for any $v\in\R$ , any time sequence $(n_k)$ and any sequence of measures $(\nu_k)$ 
satisfying LLN with slope $v$, there is an increasing subsequence $(k_i)_{i\in\N}$  such that $\mu^{m,n_{k_i}}_{x,\nu_{k_i}}$ 
converges in the sense of weak convergence of finite-dimensional distributions to a measure $\mu$ on $S_x^{m,+\infty}$.  The limiting measure
$\mu$ is a polymer measure supported on $S_x^{m,\infty}(v)$.
\end{lemma}

\begin{proof}
  Since  $(\nu_k)$ satisfies  LLN with slope $v$, 
   (\ref{eq:tightness-for-terminal-measures}) is satisfied.
  By Theorem \ref{thm:tightness-of-polymer-measure}, the sequence 
  $(\mu_{x,\nu_k}^{m,n_k})$ forms a  tight family,  so by the Prokhorov theorem, there is a converging subsequence of this sequence.
 Let $\mu$ be the limiting measure of some subsequence $(\mu_{x,\nu_{k_i}}^{m,n_{k_i}})$. It is an infinite volume polymer measure
due to Lemma~\ref{lem:limit-of-polymers-is-a-polymer}. Let us prove that for every $\varepsilon > 0$,
\begin{equation}
\label{eq:tails-summable}
\sum_{n=1}^{\infty} \mu\pi^{-1}_{m+n}\big([(v-\varepsilon)n , (v+\varepsilon) n]^c\big) < \infty.
\end{equation}
The Borel--Cantelli lemma will imply then that $\mu$ is supported on $S_x^{m,+\infty}(v)$. Fixing $\varepsilon>0$,  for
sufficiently large $n$ and $n_{k_i}-m > 2n$, we derive from  (\ref{eq:remote-control-all-space-time-all-space-time}):
\begin{equation*}
  \mu_{x,\nu_{k_i}}^{m,n_{k_i}} \pi^{-1}_{m+n} \big([(v-\varepsilon)n, (v+\varepsilon)n]^c \big)
  \le 4\nu_{k_i} \big(  [(v-\varepsilon/2)n_{k_i}, (v+\varepsilon/2) n_{k_i}]^c \big) + 6e^{-\sqrt{n}}.
\end{equation*}
Since $(\nu_k)$ satisfies LLN with slope $v$, taking the limit $k_i \to \infty$ and using the weak
convergence of finite-dimensional distributions of $(\mu_{x,\nu_{k_i}}^{n,n_{k_i}})$, we find 
\begin{equation*}
  \mu \pi^{-1}_{m+n} \big( [ (v-\varepsilon)n, (v+\varepsilon)n]^c \big)
  \le 6 e^{-\sqrt{n}}.
\end{equation*}
Therefore \eqref{eq:tails-summable} holds, and the proof is complete.
\end{proof}

\bigskip

Our next goal is uniqueness of a polymer measure with given endpoint and slope. 

Let $m\in\Z$ and let $\mu_1$ and $\mu_2$ be two measures on $S^{m,+\infty}$. We say that $\mu_1$ is stochastically dominated by $\mu_2$ if $\mu_1\pi_{m,n}^{-1}$ is stochastically dominated by $\mu_2\pi_{m,n}^{-1}$
for all finite $n>m$.

\begin{lemma}
\label{lem:infinite-volume-monotonicity-wrt-slope} 
Let $v_1<v_2$ and $(m,x)\in\ZR$.
If $\mu_1$ and $\mu_2$ are polymer measures on $S_x^{m,+\infty}$  satisfying LLN with slopes $v_1$ and $v_2$, respectively, then $\mu_2$ stochastically dominates $\mu_1$.
\end{lemma}
To prove this lemma, we need the following obvious auxiliary statement.
\begin{lemma}\label{lem:convergence-preserves-dominance} Suppose $(\mu_1^{k})_{k\in\N}$ and $(\mu_2^{k})_{k\in\N}$ are sequences of probability measures
converging weakly to probability measures $\mu_1$ and~$\mu_2$, respectively, and  such that 
$\mu_1^k$ is dominated by $\mu_2^k$ for all $k\in\N$. Then $\mu_1$ is dominated by $\mu_2$.
\end{lemma}

\bpf[Proof of Lemma~\ref{lem:infinite-volume-monotonicity-wrt-slope}]
Let us take any $\delta>0$ satisfying  $v_1+\delta< v_2-\delta$,
denote \[
 \mu_{i,k}:=\mu_i\pi_{k}^{-1},\quad i=1,2,\quad k>m, 
\]
and introduce
$\mu_{i,k,\delta}$ as $\mu_{i,k}$ conditioned on $[(v_i-\delta)k,(v_i+\delta) k]$.
Then  $\mu_{1,k,\delta}$ is dominated by $\mu_{2,k,\delta}$. Using Lemma~\ref{lem:main-monotonicity} on monotonicity, we obtain that 
$\mu_{x,\mu_{1,k,\delta}}^{m,k}$ is dominated by $\mu_{x,\mu_{2,k,\delta}}^{m,k}$.
Therefore,
$\mu_{x,\mu_{1,k,\delta}}^{m,k}\pi_{m,r}^{-1}$ is dominated by $\mu_{x,\mu_{2,k,\delta}}^{m,k}\pi_{m,r}^{-1}$, for any 
$r$ between $m$ and $k$. 
Since, in addition, the LLN assumption implies
\[
\lim_{k\to\infty}\|\mu_{i}\pi_{m,r}^{-1}-\mu_{x,\mu_{i,k,\delta}}^{m,k}\pi_{m,r}^{-1}\|_{TV}=\lim_{k\to\infty}\|\mu_{i,k}-\mu_{i,k,\delta}\|_{TV}=0,\quad  i=1,2,
\]
Lemma~\ref{lem:convergence-preserves-dominance} implies that  $\mu_{1}\pi_{m,r}^{-1}$ is dominated by $\mu_{2}\pi_{m,r}^{-1}$.
\epf

\begin{lemma} \label{lem:uniqueness-and-convergence-at-rationals}Let $v\in\R$.
Then there is a set $\tilde \Omega_v$ of probability~$1$ such that the following holds on $\tilde \Omega_v$:
\begin{enumerate}[1.]
 \item \label{it:uniqueness-at-rationals} For every point $(m,x)\in\Z\times\Q$, the set $\Pc_x^{m,+\infty}(v)$ 
 of all polymer measures on $S_x^{m,+\infty}$ satisfying SLLN with slope $v$,
 contains exactly one element that we denote by~$\mu_{x}^{m,+\infty}(v)$. 
 \item \label{it:convergence-of-polymer-measures-at-rationals} For every point $(m,x)\in\Z\times\Q$ and for every sequence of measures $(\nu_n)$ satisfying LLN with slope $v$, 
 $\mu_{x,\nu_n}^{m,n}$ weakly converges to $\mu_{x}^{m,+\infty}(v)$. 
\end{enumerate}
\end{lemma}
This lemma is weaker than Theorem~\ref{thm:thermodynamic-limit} in two ways: its statements hold only for rational spatial locations, and only weak convergence is claimed.
We study the irrational points later in this section, and prove the total variation convergence in Section~\ref{sec:global-solutions-and-ratios-of-partition-functions}.

\bpf
Let us fix a point $(m,x)$. By Lemma~\ref{lem:existence-of-polymer-measures-and-convergence-along-subsequences},
for each $v$, the set $\Pc_x^{m,+\infty}(v)$ is non-empty. For any $\mu\in \Pc_x^{m,+\infty}(v)$ and any $k>m$, the measure $\mu\pi_{k}^{-1}$ 
is equivalent to Lebesgue measure (in the sense of absolute continuity),
so for any $\alpha\in(0,1)$ the quantile $q_\alpha(\mu)$ at level $\alpha$ is uniquely defined 
by~$\mu\pi_{k}^{-1}(-\infty, q_\alpha(\mu)]=\alpha$.
So let us define
\begin{align*}
 q_\alpha^-(v)&=\inf\{q_\alpha(\mu): \mu\in \Pc_x^m(v)\},\\
 q_\alpha^+(v)&=\sup\{q_\alpha(\mu): \mu\in \Pc_x^m(v)\}.
\end{align*}
Let us prove that with probability 1, $q_\alpha^-=q_\alpha^+$. Due to 
Lemma~\ref{lem:infinite-volume-monotonicity-wrt-slope}, if $v_1<v_2$, then 
$q_\alpha^-(v_1)\le q_\alpha^+(v_1)\le q_\alpha^-(v_2)\le q_\alpha^+(v_2)$. Therefore, with probability~$1$, there may be at most
countably many nonempty intervals $I_\alpha(v)=(q_\alpha^-(v),q_\alpha^+(v))$. On the other hand, space-time shear transformations
map polymer measures into polymer measures (on finite or infinite paths), so 
$\Pp\{I_\alpha(v)\ne\emptyset\}=p$ does not depend on~$v$.
Therefore, we can apply arguments similar to those in~\cite{kickb:bakhtin2016} and going back to Lemma~6 in~\cite{HoNe3}. We take an arbitrary probability density $f$ on $\R$
and write
\[
p=\int_\R \Pp\{I_\alpha(v)\ne\emptyset\}f(v)dv
=\int_\R \E \ONE_{\{I_\alpha(v)\ne\emptyset\}}f(v)dv
=\E \int_\R \ONE_{\{I_\alpha(v)\ne\emptyset\}}f(v)dv =0,
\]
since $I_\alpha(v)\ne\emptyset$ can be true for at most countably many $v$. 
So, for any $v\in\R$, $\Pp\{I_\alpha(v)\ne\emptyset\}=0$. This immediately implies that for every $v\in\R$,
\[
 \Pp\{q_\alpha^-(v)= q_\alpha^+(v)~\text{for all}~\alpha\in\Q \}=1.
\]
So, for any $\mu_1,\mu_2\in \Pc_x^{m,+\infty}(v)$, the rational quantiles of~$\mu_1\pi_{k}^{-1}$ 
and~$\mu_2\pi_{k}^{-1}$ coincide. Therefore, $\mu_1\pi_{k}^{-1}=\mu_2\pi_{k}^{-1}$.
 In turn, this implies $\mu_1\pi_{m,k}^{-1}=\mu_1\pi_{m,k}^{-1}$. Since this is true for all $k$,
we conclude that $\mu_1=\mu_2$.

So we have proved that for a fixed point $(m,x)\in\ZR$, with probability~$1$, a polymer measure with specified asymptotic
slope is unique. We denote that measure by $\mu_x^{m,+\infty}(v)$. By countable additivity, this uniqueness statement holds true for all $(m,x)\in\Z\times\Q$ at once on a common set $\tilde\Omega_v$ of measure 1, and part~\ref{it:uniqueness-at-rationals} is proved.

To prove the second part, we fix any $\omega\in\tilde \Omega_v$ and use a compactness argument. 
Lemma~\ref{lem:existence-of-polymer-measures-and-convergence-along-subsequences} implies that from any subsequence
$(\mu_{x,\nu_n}^{m,n})$ one can choose a convergent subsubsequence. Part~\ref{it:uniqueness-at-rationals} 
of this lemma implies that all these partial limits must coincide with  $\mu_x^{m,+\infty}(v)$. Therefore, the entire sequence converges to
$\mu_x^{m,\infty}(v)$, which completes the proof of the lemma. \epf

\begin{lemma} Let $v\in\R$. On $\tilde\Omega_v$, for every $m\in\Z$ and points $x_1,x_2\in\Q$ satisfying $x_1<x_2$,  
$\mu_{x_1}^{m,+\infty}(v)$ is dominated by $\mu_{x_2}^{m,+\infty}(v)$.  
\end{lemma}
\bpf
By Lemma~\ref{lem:uniqueness-and-convergence-at-rationals}, for $i\in\{1,2\}$, the sequence of measures 
$(\mu_{x_i,vn}^{m,n})_{n>m}$ converges to $\mu_{x_i}^{m,+\infty}(v)$ as $n\to\infty$. Since
for every $n$, $\mu_{x_1,vn}^{m,n}$ is dominated by $\mu_{x_2,vn}^{m,n}$, the limiting measures are also related by
stochastic dominance.
\epf

\begin{lemma}\label{lem:squeezing-any-endpoint-of-one-sided-polymer-measures-between-rational-ones} 
For every $v$, the following holds on $\Omega_v$. For every $m\in\Z$, every $x\in\R$ and every $x_-,x_+\in\Q$ such that
$x_-<x<x_+$, every measure in $\Pc_x^{m,+\infty}(v)$ is dominated by the (unique) measure $\mu_{x_+}^{m,+\infty}(v)$ in $\Pc_{x_+}^{m,+\infty}(v)$ 
and dominates the (unique)
measure $\mu_{x_-}^{m,+\infty}(v)$ in $\Pc_{x_-}^{m,+\infty}(v)$.
\end{lemma}
\bpf We take an arbitrary measure $\mu\in\Pc_x^{m,+\infty}(v)$ and denote $\nu_n=\mu\pi_{n}^{-1}$, $n>m$. Since $\nu_n$ satisfy
LLN with slope $v$, $\mu_{x_-,\nu_n}^{m,n}$ and $\mu_{x_+,\nu_n}^{m,n}$ converge, 
by Lemma~\ref{lem:uniqueness-and-convergence-at-rationals}, to $\mu_{x_-}^{m,+\infty}(v)$ and $\mu_{x_+}^{m,+\infty}(v)$, respectively. Since
$\mu_{x,\nu_n}^{m,n}$ coincides with $\mu\pi_{m,n}^{-1}$, the lemma follows from the dominance relation 
on the pre-limiting measures.
\epf

So now we know that for any $x$, the measures in $\Pc_x^{m,+\infty}(v)$ are squeezed between measures 
$\mu_{x_-}^{m,+\infty}(v)$, $x_-\in\Q\cap(-\infty,x)$ and $\mu_{x_+}^{m,+\infty}(v)$, $x_+\in\Q\cap(-\infty,x)$. Now we need to show that
there is a unique measure with this property. 

\begin{lemma}\label{lem:equicontinuity-of-polymer-measures-wrt-endpoint}
Let $v\in\R$,  $m,k\in\Z$, $k>m$,  $r\in\N$, $y\in\Q$, and a sequence of measures $\nu_n$ satisfying
LLN with slope $v$. Then there is an event $\Omega_{v,m,k,r,y}$ of probability~$1$ such that on that event,
the family of functions $f_n:[-r,r]\cap\Q\to\R$, $n>k$, defined by
\begin{equation*}
f_n(x)=\mu_{x,\nu_n}^{m,n}\pi_{k}^{-1}((-\infty,y]) 
\end{equation*}
is uniformly equicontinuous on $[-r,r]\cap\Q$.

\end{lemma}

\bpf Without loss of generality, we assume that $m=0$. To prove the uniform equicontinuity, we will check that for every $\eps\in(0,1/2)$, there is
$\delta>0$ such that 
\begin{equation}
\label{eq:equicontinuity-wrt-endpoint}
f_n(x_0)-f_n(x'_0)\le 6\eps,\quad |x_0|,|x'_0|\le r,\ |x_0-x'_0|\le \delta.
\end{equation}
First, we use LLN for $(\nu_n)$ to find $L>0$ such that
\begin{equation}
\label{eq:sublinear-growth-of-nu_n}
 \nu_n(\R\setminus [-Ln,Ln])<\eps,\quad n\in\N.
\end{equation}
Then we use monotonicity and tightness to find $R>|y|$ such that
\begin{equation}
\label{eq:uniform-compactness-in-starting-point-and-linearly-growing-endpoint}
 \mu_{x,y}^{0,n}\pi_{1,k}^{-1}(\R^k\setminus B_R^k)<\eps,\quad x\in[-r,r],\ n\in\N,\ y\in [-Ln,Ln],
\end{equation}
where $B_R^k=[-R,R]^k$. 
Inequality~\eqref{eq:sublinear-growth-of-nu_n} implies
\begin{align*}
f_n(x_0)&=\int_{\R}\nu_n(dw) \mu_{x_0,w}^{0,n}\pi^{-1}_{1,k}(\R^{k-1}\times(-\infty,y])
\\
&\le \int_{[-Ln,Ln]}\nu_n(dw) \mu_{x_0,w}^{0,n}\pi^{-1}_{1,k}(\R^{k-1}\times(-\infty,y])+\eps.
\end{align*}
Introducing $B_R^k(y)=[-R,R]^{k-1}\times [-R,y]$ and $B_{Ln}=[-Ln,Ln]$, we can use~\eqref{eq:uniform-compactness-in-starting-point-and-linearly-growing-endpoint} to write
\begin{align*}
 f_n(x_0)&\le \int_{B_{Ln}}\nu_n(dw) \mu_{x_0,w}^{0,n}\pi^{-1}_{1,k}(B_R^k(y))+2\eps,
\\
& \le \int_{B_{Ln}}\nu_n(dw)\frac{
\int_{B_R^k(y)}  Z(x_0,\ldots,x_k,w)\,dx_1\ldots dx_k
}{\int_{B_R^k}  Z(x_0,\ldots,x_k,w)\,dx_1\ldots dx_k
}+2\eps,
\end{align*}
where
\[
 Z(x_0,\ldots,x_k,w)=e^{-F_0(x_0)}g(x_1-x_0)\cdot 
\prod_{i=1}^{k-1} e^{-F_i(x_i)}g(x_{i+1}-x_{i})\cdot Z_{x_k,w}^{k,n}.
\]

For every $\delta>0$, let us define 
\[
K_\delta=\sup\left\{\frac{e^{-F(x'_0)}g(x_1-x'_0)}{e^{-F(x_0)}g(x_1-x_0)}:\ 
|x_0|,|x'_0|\le r,\ |x_0-x'_0|\le \delta,\ |x_1|\le R\right\}. 
\]
Then $\lim_{\delta\downarrow 0}K_\delta =1$ with probability $1$. Also, we can continue the above sequence of inequalities, assuming
$|x_0-x'_0|\le \delta$:

\begin{align*}
f_n(x_0) &\le  
K^2_\delta
\int_{B_{Ln}}\nu_n(dw)\frac{
\int_{B_R^k(y)}  Z(x'_0,x_1\ldots,x_k,w)\,dx_1\ldots dx_k
}{\int_{B_R^k} Z(x'_0,x_1\ldots,x_k,w) \,dx_1\ldots dx_k
}+2\eps
\\ &\le   
K^2_\delta
\int_{\R}\nu_n(dw)\frac{
\int_{\R^{k-1}\times (-\infty,y]} Z(x'_0,x_1\ldots,x_k,w) \,dx_1\ldots dx_k
}{(1-\eps)\int_{\R^k}  Z(x'_0,x_1\ldots,x_k,w) \,dx_1\ldots dx_k
}+2\eps
\\ &\le \frac{K^2_\delta}{1-\eps}f_n(x'_0)+2\eps.
\end{align*}
Therefore, if $\delta$ is chosen so that $K^2_\delta\le 1+\eps$, we obtain
\[
 f_n(x_0)-f_n(x'_0)\le \left(\frac{K^2_\delta}{1-\eps}-1\right)f_n(x'_0)+2\eps
 \le \frac{K^2_\delta}{1-\eps}-1+2\eps\le \frac{1+\eps}{1-\eps}-1+2\eps\le 6\eps,
\]
and~\eqref{eq:equicontinuity-wrt-endpoint} holds.
\epf

\begin{lemma}\label{lem:continuity-of-one-sided-polymer-measures-wrt-endpoint}
Let $v\in\R$,  $m,k\in\Z$, $k>m$,  $r\in\N$, $y\in\Q$.
On $\tilde\Omega_v\cap \Omega_{v,m,k,r,y}$,
the function $f:[-r,r]\cap\Q\to\R$, defined by
\begin{equation}
f(x)=\mu_{x}^{m,+\infty}(v)\pi_{k}^{-1}((-\infty,y]), 
\end{equation}
is uniformly continuous on $[-r,r]\cap\Q$.
\end{lemma}
\bpf Let us choose any sequence $(\nu_n)$ satisfying LLN with slope $v$ and define $f_n$ as in Lemma~\ref{lem:equicontinuity-of-polymer-measures-wrt-endpoint}. 
The statement follows then from that lemma since $\lim_{n\to\infty}f_n(x)=f(x)$ for
$x\in [-r,r]\cap\Q$.
\epf

We can now prove the complete uniqueness and weak convergence claims of Theorem~\ref{thm:thermodynamic-limit}: 
\begin{lemma}\label{lem:uniqueness-of-polymer-measures}
  Let $v\in\R$. Then on $\Omega_v=\tilde\Omega_v\cap \bigcap_{m,k,r,y}\Omega_{v,m,k,r,y}$,
\begin{enumerate}[1.]
 \item \label{it:uniqueness-at-reals} For any point $(m,x)\in\Z\times\R$, the set $\Pc_x^{m,+\infty}(v)$ 
 of all polymer measures on $S_x^{m,+\infty}$ satisfying SLLN with slope $v$,
 contains exactly one element, $\mu_{x}^{m,+\infty}(v)$. 
 \item \label{it:convergence-of-polymer-measures-at-reals} For any point $(m,x)\in\Z\times\R$ and 
 for every sequence of measures $(\nu_n)$ satisfying LLN with slope $v$, 
 $\mu_{x,\nu_n}^{m,n}$ converges to $\mu_{x}^{m,+\infty}(v)$ weakly. 
\end{enumerate}

\end{lemma}
\bpf The second part follows from the first one and the compactness argument explained in the proof of 
Lemma~\ref{lem:uniqueness-and-convergence-at-rationals}.

To prove the first part, it is sufficient to fix $(m,x)\in\ZR$ and check that for every $k>m$, the marginal measure 
$\nu_k=\mu\pi_{k}^{-1}$ does not depend on $\mu\in\Pc_x^{m,+\infty}(v)$. For that, it suffices to see that for every choice of $y\in\Q$,
$\nu_k((-\infty,y])$
does not depend on $\mu\in\Pc_x^{m,+\infty}(v)$.

If $x_-<x<x_+$, then
$\mu_{x_-,\nu_n}^{m,n}$ is dominated by $\mu_{x,\nu_n}^{m,n}$ which
is dominated by  $\mu_{x_+,\nu_n}^{m,n}$.  Therefore, 
for every $y\in\R$,
\[
\mu_{x_-,\nu_n}^{m,n}\pi_{k}^{-1}((-\infty,y])\ge \mu_{x,\nu_n}^{m,n}\pi_{k}^{-1}((-\infty,y])\ge \mu_{x_+,\nu_n}^{m,n}\pi_{k}^{-1}((-\infty,y]).  
\]
Since $\mu_{x,\nu_n}^{m,n}\pi_{k}^{-1}=\nu_k$, we obtain
\begin{equation}
\mu_{x_-,\nu_n}^{m,n}\pi_{k}^{-1}((-\infty,y])\ge \nu_k((-\infty,y])\ge \mu_{x_+,\nu_n}^{m,n}\pi_{k}^{-1}((-\infty,y]).  
\label{eq:nu_k-squeezed-1}
\end{equation}

If additionally $x_-,x_+\in\Q$, then   
f.d.d.'s of $\mu_{x_-,\nu_n}^{m,n}$ and $\mu_{x_+,\nu_n}^{m,n}$ weakly converge to those of $\mu^{m,+\infty}_{x_-}(v)$ and $\mu^{m,+\infty}_{x_+}(v)$,
due to Lemma~\ref{lem:uniqueness-and-convergence-at-rationals} since $(\nu_n)_{n>m}$ satisfies LLN with slope~$v$. Since marginals of
both $\mu^{m,+\infty}_{x_-}(v)$ and $\mu^{m,+\infty}_{x_+}(v)$ are absolutely continuous, \eqref{eq:nu_k-squeezed-1} implies
\[
\mu^{m,+\infty}_{x_-}(v)((-\infty,y])\ge \nu_k((-\infty,y])\ge \mu^{m,+\infty}_{x_+}(v)((-\infty,y]).   
\]
Lemma~\ref{lem:continuity-of-one-sided-polymer-measures-wrt-endpoint} implies that
\[
 \inf_{x_-\in \Q\cap (-\infty,x)} \mu^{m,+\infty}_{x_-}(v)((-\infty,y]) = \sup_{x_+\in \Q\cap (x,+\infty)} \mu^{m,+\infty}_{x_+}(v)((-\infty,y]).
\]
Denoting this common value by $c$, we conclude that  the value of $\nu_k((-\infty,y])$ is uniquely defined and equals~$c$,  
which completes the proof.
\epf

\section{Existence and uniqueness of global solutions and basins of pullback
  attraction}
\label{sec:global-solutions-and-ratios-of-partition-functions}

The main goal of this section is to prove Theorems~\ref{thm:global_solutions} and~\ref{thm:pullback_attraction} on global solutions of the randomly kicked Burgers equation. These solutions will be constructed and studied via the pullback procedure with the help of backward polymer measures. 
We also prove the total variation convergence part in the thermodynamic limit Theorem~\ref{thm:thermodynamic-limit}.

Our main result on thermodynamic limit for polymer measures, Theorem~\ref{thm:thermodynamic-limit}, was naturally stated and proved in terms of forward polymer measures defined on $S^{m,+\infty}_x(v)$ for endpoints
$(m,x)\in\ZR$ and slopes $v$. A direct counterpart of this result holds for thermodynamic limit polymer measures
defined on spaces $S^{-\infty,m}_x(v)$ 
of backward infinite sequences $y:\{\ldots,m-2,m-1,m\}\to\R$ such that $y_m=x$
and $\lim_{n\to-\infty}(y_n/n)=v$.
 All the definitions, notation, and intermediate results can be straightforwardly adapted to this case. 
We do not introduce these adaptations one-by-one since all of them are unambiguously defined by the context.
When dealing with backward
polymer measures we will still refer to results obtained for forward polymer measures, although formally they
are not equivalent due to the asymmetry in the definition of path energy near the path endpoints in~\eqref{eq:path-energy}. 
Moreover, at times we abuse the notation by using the same symbols for objects constructed for forward and backward versions.

\smallskip
A function $u(n,x)=u_{\omega}(n,x)$ is a global solution of the Burgers equation if the version
of the Hopf--Cole transform defined by
\begin{equation*}
V(n,x) = e^{-U(n,x)} = e^{-\int_0^x u(n,y) dy},\quad (n,x)\in\ZR,
\end{equation*}
 satisfies, for all integers $m<n$ and all $x\in\R$,
\begin{equation}
\label{eq:evolution-for-V}
V(n,x) = C_{m,n}[\Xi_{\omega}^{m,n} V(m,\cdot)](x) := C_{m,n}\int Z^{m,n}(y,x) V(m,y) dy,
\end{equation}
 where $(C_{m,n})$ is a random family of constants such that
$C_{m,n}C_{n,k}=C_{m,k}$, $m<n<k$.
We need to introduce the normalizing constants $C_{m,n}$ for consistency with the identity $V(n,0)=1$ holding for all $n$, because we fix the the lower limit of integration to be zero when defining the Hopf--Cole transform.

The following computation shows that, given any $v\in\R$ and $N\in\Z$, the functions 
\begin{equation}
\label{eq:Hopf-Cole-for-prelimit}
V_{v}^{N}(n,x) =
Z^{N,n}_{Nv,x}/Z^{N,n}_{Nv,0},\quad n>N,\ x\in\R, 
\end{equation}
and constants 
\begin{equation}
\label{eq:Normalizing-constants-product-cocycle}
C^N_{v,m,n} = Z^{N,m}_{Nv,0}/Z^{N,n}_{Nv,0},\quad m<n, 
\end{equation}
satisfy (\ref{eq:evolution-for-V}) for $n> m >N$: 
\begin{equation*}
\begin{split}
  C_{v,m,n}^N [\Xi_{\omega}^{m,n}V^N_{v}(m,\cdot)](x)
  &= C_{v,m,n}^{N}\int Z^{m,n}(y,x) V_{v}^N(m,y) dy \\
  &= Z^{N,m}_{Nv,0}/Z^{N,n}_{Nv,0} \int  Z^{N,m}_{Nv,y}   Z^{m,n}_{y,x}
  /Z^{N,m}_{Nv,0}\, dy\\
  &= (Z^{N,n}_{Nv,0})^{-1} \int   Z^{N,m}_{Nv,y}   Z^{m,n}_{y,x}\, dy \\
  &= (Z^{N,n}_{Nv,0})^{-1}Z^{N,n}_{Nv,x} = V^N_{v}(n,x).
\end{split}
\end{equation*}
Therefore, a natural guess for the Hopf--Cole transform of global solutions will be 
$V(n,x)=V_v(n,x)=\lim\limits_{N\to -\infty}V^N_v(n,x)$, along with normalizing constants given by $C_{m,n}= C_{v,m,n}=\lim\limits_{N \to
  -\infty}C^N_{v,m,n}$.
This leads to the study of the limits of partition function ratios.
On the other hand,  let $u^N_{v}(n,x) = -\frac{\partial}{\partial x}\ln
V^N_v(n,x)$ be the inverse Hope--Cole transform of $V^{N}_v(n,x)$.
We find that
\begin{equation}
\label{eq:approx-solution-via-polymer-measures}
\begin{split}
  u^N_v(n,x) &= -\frac{\partial}{\partial x} \ln \int_{\R}
  \frac{1}{\sqrt{2\pi}}e^{-(x-y)^2/2 - F_{n-1}(y)} Z^{N,n-1}(Nv,y) dy \\
  &=  \frac{ \int_{\R} (x-y)\frac{1}{\sqrt{2\pi}}e^{-(x-y)^2/2-F_{n-1}(y)} Z^{N,n-1}(Nv,y)
    dy
  }{ Z^{N,n}(Nv,x)} \\
  &= \int_{\R} (x-y) \mu_{Nv,x}^{N,n} \pi_{n-1}^{-1}(dy).
\end{split}
\end{equation}
Taking the limit $N \to -\infty$, we expect the global solution to be 
\begin{equation}
\label{eq:global-solution-via-marginal-polymer-measures}
u_v(n,x) = \int_{\R} (x-y) \mu_x^{-\infty,n}(v) \pi_{n-1}^{-1}(dy),
\end{equation}
where $\mu_x^{-\infty,n}(v)$ is the backward polymer measure, the weak limit 
of $\mu_{Nv,x}^{N,n}\pi_{n-1}^{-1}$ as $N\to-\infty$.
To justify this answer, we actually need a stronger statement than weak convergence, namely,
a statement on convergence of the associated densities.

The convergence of densities is closely related to convergence of partition function ratios, since the density of $\mu_{Nv,x}^{N,n}\pi_m^{-1}$ is precisely
\begin{equation*}
\frac{d \mu_{\nu_N,x}^{N,n}\pi_m^{-1} }{ d\Leb }(y) =  \frac{Z_{Nv, y}^{N,m}}{Z_{Nv,x}^{N,n}}Z_{y,x}^{m,n}.
\end{equation*}
In Section~\ref{sec:limits-of-partition-function-ratios}, we will show that both convergences are uniform on compact sets.
The existence of global solutions is then established in Section~\ref{sec:existence-of-global-solution}.

The uniqueness of global solutions relies on the uniqueness of infinite volume polymer measures with
any given
slope $v$.

Suppose $u_v(n,x) \in \mathbb{H'}(v,v)$ is a global solution and $V_v(n,x)$ is its Hopf--Cole transform.
For fixed $(n,x) \in \Z \times \R$, we can define a ``backward'' point-to-line polymer measure $\bar{\mu}^{-\infty,n}_{x}$ on the set 
$S_x^{-\infty,n}$ of paths $\gamma:\{...,n-2,n-1,n\}
\to \R$ with $\gamma(n)=x$ 
:
\begin{align}  \label{eq:def-of-mu-bar}\nonumber
&\quad \bar{\mu}_x^{-\infty,n}(A_{n-k} \times ... \times A_{n-1} \times A_n) \\
&= \frac{ \int_{A_{n-k}} dx_{n-k} \cdots \int_{A_{n-1}} dx_{n-1} \int_{A_n} \delta_x(dx_n) \quad V_v(n-k,x_{n-k})
  \prod\limits_{i=n-k}^{n-1} Z^{i,i+1}_{x_i, x_{i+1}}
}{\int_{\R} V_v(n-k, x_{n-k}) Z^{n-k,n}_{x_{n-k},x} dx_{n-k}
}.
\end{align}
This definition is
consistent for different choices of $k$ since $V_v(n,x)$ satisfies (\ref{eq:evolution-for-V}).
Then the global solution $u_{v}(n,x)$ is uniquely determined by $\bar{\mu}_x^{-\infty,n}$ through
\begin{equation}\label{eq:solution-through-pt-to-line-measure}
u_{v}(n,x) = \int_{\R} (x-y) \bar{\mu}_x^{-\infty,n} \pi_{n-1}^{-1}(dy).
\end{equation}
We will show that the measures $\bar{\mu}_x^{-\infty,n}$ satisfy LLN with slope~$v$. This will allow us to conclude that they are are uniquely defined by the potential and coincide with~$\mu_x^{-\infty,n}(v)$, 
so the global solution in $\HH'(v,v)$ is also uniquely defined by the potential and coincides with 
$u_v$, see~(\ref{eq:global-solution-via-marginal-polymer-measures}). 
This is done in Section~\ref{sec:uniqueness-of-global-solutions}.

In Section~\ref{sec:pullback} we show that global solutions are also pullback attractors.
We also generalize the result on convergence of density functions to certain point-to-line polymer measures.

\subsection{Limits of partition function ratios}
\label{sec:limits-of-partition-function-ratios}

Let $\Omega'_v = \Omega' \cap \Omega_v$, where $\Omega'$ and $\Omega_v$ have been introduced in Theorems~\ref{thm:all-space-time-point-compactness} and~\ref{thm:thermodynamic-limit}.
In this section, we will prove the following two theorems on $\Omega'_v$:
\begin{theorem}
\label{thm:convergence-of-partition-function-ratios}
For all $\omega \in \Omega_v'$, there is a function $G_v\bigl( (n_1,x_1), (n_2,x_2) \bigr)>0$ such
that for any sequence $(y_N)$ satisfying $\lim\limits_{N \to  -\infty} y_N/N = v$, we have
\begin{equation*}
\lim\limits_{N\to -\infty}  Z_{y_N,x_1}^{N,n_1} /Z_ {y_N, x_2}^{N, n_2} = G_{v}\bigl(
(n_1,x_1), (n_2,x_2) \bigr).
\end{equation*}
For  $n_1,n_2$ fixed and  $x_1,x_2$ restricted to a compact set, the convergence is uniform. 
\end{theorem}
\begin{theorem}
\label{thm:convergence-of-density-functions}
For all $\omega \in \Omega_v'$ the following is true.
Suppose a family of probability measures $\bigl( \nu_N \bigr)$
satisfies the following conditions:
\begin{enumerate}
\item\label{item:23} for some $c>0$ and all sufficiently large $|N|$, 
  \begin{equation}
\label{eq:finite-support-assumption}
\nu_N \big( [-c|N|, c|N|]^c \big)=0;
\end{equation} 
\item\label{item:24} $\bigl( \nu_N \bigr)$
  satisfies LLN with slope $v$ as $N\to -\infty$.
\end{enumerate}
Let $N<m<n$ and let $f^N_{n,m}(x,\cdot)$ be the density of $\mu_{\nu_N,x}^{N,n} \pi_m^{-1}$, namely,
\begin{equation*}
f^N_{n,m}(x,y) = \int_{-c|N|}^{c|N|} \frac{Z^{N,m}(z,y)Z^{m,n}(y,x)}{Z^{N,n}(z,x)} \nu_N(dz).
\end{equation*}
Then $f_{n,m}^N(x,y)$ converges uniformly in $x$ and $y$ on compact sets to $f_{v,n,m}(x,y)$ as $N \to
-\infty$, where $f_{v,n,m}(x,\cdot)$ is the density of $\mu_x^{-\infty,n}(v) \pi_m^{-1}$ and can be expressed as 
\begin{equation}
  \label{eq:expression-of-polymer-measure-density}
  f_{v,n,m}(x,y) = Z^{m,n}_{y,x} G_v \bigl( (m,y), (n,x) \bigr).
\end{equation}
\end{theorem}

We can derive part~\ref{it:thermodynamic-limit} of Theorem~\ref{thm:thermodynamic-limit} from
Theorem~\ref{thm:convergence-of-density-functions}.
Combined with Section~\ref{sec:infinite-volume-polymer-measure}, this completes the proof of
Theorem~\ref{thm:thermodynamic-limit}.

\begin{proof}[Proof of (the time-reversed version of) part~\ref{it:thermodynamic-limit} of Theorem~\ref{thm:thermodynamic-limit} ] 
  Let us take the full measure set $\Omega'_v$.
  For every $\omega \in \Omega'_v$, our goal is to show that for any $(n,x) \in \ZR$ and $(\nu_N)$ satisfying LLN with slope $v$,
  $\mu_{\nu_N, x}^{N,n}\pi_m^{-1}$ converges to $\mu_x^{-\infty,n}(v)\pi_m^{-1}$ in total variation for all $m<n$.

 Let $c > |v| + 1$.  Denoting the conditioning of $\nu_N$ on $[-c|N|, c|N|]$ by $\tilde{\nu}_{N}$, we get
\begin{align*}
  &\| \mu_{\nu_N, x}^{N,n}\pi_m^{-1} - \mu_x^{-\infty,n}(v)\pi_m^{-1} \|_{\text{TV}}
\\
  \le &  \| \mu_{\nu_N, x}^{N,n}\pi_m^{-1} - \mu_{\tilde{\nu}_N, x}^{N,n}\pi_m^{-1}  \|_{\text{TV}} +   \|
        \mu_{\tilde{\nu}_N, x}^{N,n}\pi_m^{-1} - \mu_x^{-\infty,n}(v)\pi_m^{-1} \|_{\text{TV}} \\
  \le & \| \nu_N - \tilde{\nu}_N \|_{\text{TV}} +   \|  \mu_{\tilde{\nu}_N, x}^{N,n}\pi_m^{-1} - \mu_x^{-\infty,n}(v)\pi_m^{-1} \|_{\text{TV}}.
\end{align*}
The first term goes to $0$ since $(\nu_N)$ satisfies LLN with slope~$v$. 
To see that the second term goes to $0$, we notice that $\bigl( \tilde{\nu}_N \bigr)$ satisfies LLN with slope $v$
and~(\ref{eq:finite-support-assumption}), so we can apply Theorem~\ref{thm:convergence-of-density-functions} to conclude
that the densities of~$\mu_{\tilde{\nu}_N, x}^{N,n}\pi_m^{-1}$ converge to that of~$\mu_x^n(v)\pi_m^{-1}$, uniformly on
any compact set, which implies convergence in total variation.
This completes the proof.
\end{proof}

In what follows, we could use the spatial smoothness of $F$ to simplify some of the arguments, but we prefer to give a version of the proof that
does not even assume continuity of $F$.
Let us define
\begin{equation}\label{eq:definition-of-gN}
g^N_{n,m}(x,y) = \int_{-c|N|}^{c|N|} \frac{Z^{N,m}(z,y)}{Z^{N,n}(z,x)} \nu_{N}(dz),
\end{equation}
so that
\begin{equation}
  \label{eq:relation-between-f-and-g}
f_{n,m}^N(x,y) = Z^{m,n}_{y,x} g^N_{n,m}(x,y),
\end{equation}
then $g^N_{n,m}(x,y)$ is continuous in both $x$ and $y$ as the following lemma shows.
\begin{lemma}
\label{lem:continuous-of-g}
If  $\nu_N$ satisfies (\ref{eq:finite-support-assumption}), then  $g^N_{n,m}(x,y)$ is continuous in $x$ and $y$.
\end{lemma}
\begin{proof}
  By Lemma \ref{lem:smoothness-of-partition-ratio},  $h(x,y,z) = Z^{N,m}(z,y)/Z^{N,n}(z,x)$ is
  uniformly continuous in $x,y$, and $z$ on any compact set, so the lemma follows
  from~\eqref{eq:definition-of-gN}.
\end{proof}

The next lemma shows that  if $n<m$, then for any compact set $K \subset \R \times \R$, the family $\bigl( \ln g^N_{n,m}(\cdot,\cdot) \bigr)_{N<m}$ is precompact 
in $\mathcal{C}(K)$, the space of continuous functions on $K$.
From this we will derive Theorems~\ref{thm:convergence-of-partition-function-ratios} and~\ref{thm:convergence-of-density-functions}.

\begin{lemma}
\label{lem:precompactness-of-gN}
Let $\omega \in \Omega'$ and $K$ be a compact subset of $\R\times\R$. 
If  measures $\nu_{N}$ satisfy (\ref{eq:finite-support-assumption}), 
then for any $m,n\in\Z$ satisfying $m<n$, the family $\bigl( \ln g^N_{n,m}(\cdot,\cdot) \bigr)_{N<m}$ is precompact in $\mathcal{C}(K)$.
\end{lemma}
Note that for a compact set $K$ and a family of positive functions $h^N$, the precompactness of  $\bigl( \ln h^N \bigr)$ implies that of $\bigl( h^N \bigr)$ in $\mathcal{C}(K)$.

\begin{proof}[Proof of Theorems \ref{thm:convergence-of-partition-function-ratios} and
  \ref{thm:convergence-of-density-functions}] We fix $m < n$ and let $\omega \in \Omega'_v$.
  Using Lemma \ref{lem:precompactness-of-gN} and the standard diagonal procedure, we can find a
  subsequence $\bigl( g^{N_k}_{n,m}(x,y) \bigr)$  converging to some $\tilde{g}(x,y)$, uniformly on any compact set.
  Since $\ln Z^{m,n}(x,y)$ is bounded on every compact set,  by (\ref{eq:relation-between-f-and-g}), we see that $f^{N_k}_{n,m}(x,y)$
  converges to $\tilde{f}(x,y) = Z^{m,n}(x,y)\tilde{g}(x,y)$ uniformly on compact sets. 
  
  Let us now identify the limit of any  subsequence $\bigl( f^{N_k}_{m,n}(x,y) \bigr)$ converging uniformly on all compact
  sets.
  On $\Omega_v$, if $(\nu_{N_k})$ satisfies LLN with slope $v$, then $\mu_{\nu_{N_k}, x}^{N_k,n} \pi_m^{-1}$ converge weakly to $\mu_x^{-\infty,n}(v)\pi_m^{-1}$.
  Hence $ \tilde{f}(x,\cdot)$ must equal $f_{v,n,m}(x,\cdot)$,  the density of $\mu_x^{-\infty,n}(v) \pi_m^{-1}$. So, the only
  possible limit of any subsequence of $\bigl( f^N_{n,m}\bigr)$, is $f_{v,n,m}$, and we obtain that $f^N_{n,m}$ converges to $f_{v,n,m}$, uniformly on any compact set.
  By (\ref{eq:relation-between-f-and-g}), $g^N_{n,m}$ also converges to $\tilde{g}(x,y) =\bigl(
  Z^{m,n}_{y,x} \bigr)^{-1} f_{v,n,m}(x,y)$, uniformly on any compact set. 

  Next we will show  \eqref{eq:expression-of-polymer-measure-density} and Theorem \ref{thm:convergence-of-partition-function-ratios}.

  Let $(y_N)$ be such that $y_N/N\to v$.
Then $\nu_N = \delta_{y_N}$ satisfy (\ref{eq:finite-support-assumption}), so 
\begin{equation*}
g^N_{n,m}(x,y) = Z^{N,m}_{y_N,y}/ Z^{N,n}_{y_N,x} \to \bigl( Z^{m,n}_{y,x} \bigr)^{-1} f_{v,n,m}(x,y),
\end{equation*}
where the convergence is uniform on any compact set.
We denote the limit by $G_v \bigl( (m,y), (n,x) \bigr)$.
By Lemma \ref{lem:precompactness-of-gN}, we know that $\bigl(  \ln g^N_{n,m} \bigr)$ is uniformly
bounded, hence $G_v$ is strictly positive.
This proves \eqref{eq:expression-of-polymer-measure-density} and Theorem \ref{thm:convergence-of-partition-function-ratios} for $n_1 < n_2$.

For $n_1 \ge n_2$, we can simply use the following two identities:
\begin{equation*}
  \lim_{N\to -\infty} \frac{Z^{N,n_1}_{y_N,x_1}}{Z^{N,n_2}_{y_N,x_2}} =
   \Big( \lim_{N\to -\infty} \frac{Z^{N,n_2}_{y_N,x_2}}{Z^{N,n_1}_{y_N,x_1}}  \Big)^{-1}
\end{equation*}
and 
\begin{equation*}
  \lim_{N\to -\infty} \frac{Z^{N,n_1}_{y_N,x_1}}{Z^{N,n_2}_{y_N,x_2}}
  =   \lim_{N\to -\infty} \frac{Z^{N,n_1}_{y_N,x_1}}{Z^{N,n_3}_{y_N,x_3}}
    \lim_{N\to -\infty} \frac{Z^{N,n_3}_{y_N,x_3}}{Z^{N,n_2}_{y_N,x_2}}.
\end{equation*}
\end{proof}

To prove Lemma \ref{lem:precompactness-of-gN}, we need the following corollary of Theorem
\ref{thm:all-space-time-point-compactness}:
\begin{lemma}
\label{lem:one-step-control}
Let $\omega\in \Omega'$.
For $(l,q) \in \Z \times \Z$ and  $c > 0$, there is a constant $n_1=n_1(l,q,c)$ such that for any $ s > n_1$, 
any $x \in [q,q+1]$, any terminal measure~$\nu$, any $s'\in[1,s]_{\Z}$, and any $N\in \Z$ such that $N
\le l - 2s$, we have
\begin{equation}
  \label{eq:one-step-large-small-probability}
  \mu_{\nu, x}^{N,l} \pi_{l-s'}^{-1} \big( [-(2c+R+1)s, (2c+R+1)s]^c \big)
  \le 4\nu \big( [-c|N|,c|N|]^c \big) + 8e^{-\sqrt{s}}. 
\end{equation}
\end{lemma} 
\begin{proof}  We apply the backward version of Theorem~\ref{thm:all-space-time-point-compactness},
taking $v' = 0, u_0 = c, u_1 = 2c$ and letting $n_1(l,q,c) = n_0(\omega, l, q, [c],[c^{-1}])$, so
\eqref{eq:one-step-large-small-probability} is directly implied by the backward version of~\eqref{eq:compactness-control-by-terminal-measure-all-space-time}.    
\end{proof}

\begin{proof}[Proof of Lemma \ref{lem:precompactness-of-gN}]
Let us fix a compact set $K = [p,p+1]\times [-k,k]$ and times $m < n$.
We denote $r = 2c+R+1$ and, for $\varepsilon \in (0, 1/2)$, define
\begin{equation*}
s_1 = \max\Big\{ n-m,\, n_1 (n,p,c),\,   \frac{k}{r},\, \ln^2 \frac{\varepsilon}{64}\Big\}
\end{equation*}
and
\begin{equation*}
  s_2 = \max \Big\{ n_1(m,i,c):\ |i| \le rs_1 +1\Big\} \vee \ln^2\frac{\varepsilon}{64}.
\end{equation*}
We will need several truncated integrals: 
\begin{align*}
&\bar{Z}^{N,m}_{z,y} = \int_{-rs_2}^{rs_2} Z^{N,m-1}_{z,w}
  Z^{m-1,m}_{w,y} dw  =\int_{-rs_2}^{rs_2} Z^{N,m-1}_{z,w}
                        \frac{1}{\sqrt{2\pi}}e^{-\frac{(w-y)^2}{2} -  F_{m-1}(w)} dw,  
  \\
& \bar{Z}^{m,n}_{y,x} = 
\begin{cases}
 Z^{m,n}_{y,x},&  m=n-1,  \\
\int_{-rs_1}^{rs_1} Z^{m,n-1}_{y,w} Z^{n-1,n}_{w,x} dw
  ,&  m < n-1,
\end{cases}
\\ 
 & \bar{Z}^{N,n}_{z,x} = \int_{-rs_1}^{rs_1}
                        \bar{Z}^{N,m}_{z,y} \bar{Z}^{m,n}_{y,x} dy,
                         \\
& \bar{g}^N(x,y) = \int_{-c|N|}^{c|N|} \frac{ \bar{Z}^{N,m}_{z,y}}{
  \bar{Z}^{N,n}_{z,x}}  \nu_N(dz),
\end{align*}
where, according to \eqref{eq:Z}, $Z^{n-1,n}_{w,x}=\frac{1}{\sqrt{2\pi}} e^{-\frac{(w-x)^2}{2} - F_{n-1}(w)}$ in the definition of $\bar Z^{m,n}_{y,x}$.

We also let $h_{\varepsilon}^N = \ln \bar{g}^N$, $N<m$, and define  
  $\tilde{K} = [p,p+1]\times[-rs_1,rs_1] \supset K$.
If we prove that for every $\eps>0$,
\begin{equation}\label{eq:eps-approximation-of-g}
    |\ln g^N_{n,m}(x,y) - h_{\varepsilon}^N(x,y) | \le \varepsilon,\quad (x,y) \in \tilde{K},
  \end{equation}
and $\big( h_{\varepsilon}^N\big)$ is precompact in $\mathcal{C}(\tilde{K})$, then the lemma will follow  since, given any $\eps>0$, we will be able to use an $\eps$-net for $(h_\eps^N)$ to construct a $2\eps$-net for $(\ln g_{n,m}^N)$.

  Let $N \le \min\{n-2s_1, m-2s_2\}$.
  If $|y| \le rs_1$ and $|z| \le cN$, then, from  $\delta_z ([-c|N|, c|N|]^c) = 0$
and \eqref{eq:one-step-large-small-probability} with $(l,s',s,w,\nu)
= (m,1,s_2, y, \delta_z)$, we obtain
\begin{equation*}
  1 - \frac{\bar{Z}_{z,y}^{N,m}}{Z_{z,y}^{N,m}} 
  \le \mu_{z,y}^{N,m} \pi_{m-1}^{-1} ([-rs_2, rs_2]^c) \le
    8e^{-\sqrt{s_2}} \le \varepsilon/8.
  \end{equation*}
Then  using the elementary inequality $|\ln(1+x)| \le 2|x|$ for $|x| \le
1/2$ we find
\begin{equation}
  \label{eq:estimate-on-above}
e^{-\varepsilon/4} \le \bar{Z}_{z,y}^{N,m} / Z_{z,y}^{N,m}  \le 1.
\end{equation}
Let 
\begin{equation*}
\tilde{Z}^{N,n}_{z,x} = \int_{-rs_1}^{rs_1}Z^{N,m}_{z,y} \bar{Z}^{m,n}_{y,x} dy.
\end{equation*}
Then (\ref{eq:estimate-on-above}) implies
\begin{equation}
  \label{eq:ratio-Z-tilde-and-Z-bar}
1\le \tilde{Z}_{z,x}^{N,n} / \bar{Z}_{z,x}^{N,n} \le e^{\varepsilon/4}.
\end{equation}
Similarly, if $x \in [p,p+1]$ and $|z| \le cN$, by (\ref{eq:one-step-large-small-probability}) with
$(l,s',s,w,\nu) = (n,1,s_1, x, \delta_z)$ and $(n,n-m, s_1,x,\delta_z)$, we obtain
\begin{equation*}
 1- \frac{ \tilde{Z}_{z,x}^{N,n} }{Z_{z,x}^{N,n}}
  \le \mu_{z,x}^{N,n} \pi_{n-1}^{-1} ([-rs_1, rs_1]^c) + \mu_{z,x}^{N,n} \pi_m^{-1}([-rs_1,rs_1]^c)
 \le 16e^{-\sqrt{s_1}} \le \varepsilon/4.
\end{equation*}
Therefore, 
\begin{equation}
  \label{eq:estimate-on-below}
e^{-\varepsilon/2} \le \tilde{Z}_{z,x}^{N,n}/ Z_{z,x}^{N,n} \le e^{\varepsilon/2}.
\end{equation}
Combining (\ref{eq:estimate-on-above}), (\ref{eq:ratio-Z-tilde-and-Z-bar}) and (\ref{eq:estimate-on-below}) we obtain 
\begin{equation*}
e^{-\varepsilon} \le \bar{g}^N(x,y) / g^N_{n,m}(x,y) \le e^{\varepsilon},
\end{equation*}
and (\ref{eq:eps-approximation-of-g}) follows.

For any $|w| \le rs_2$ and $y,y' \in [-rs_1, rs_1]$, we have
\begin{equation*}
\bigg| \frac{(y-w)^2}{2} - \frac{(y'-w)^2}{2} \bigg| \le r(s_1+s_2) |y-y'|.
\end{equation*}
Hence 
\begin{equation*}
 \big| \ln \bar{Z}^{N,m}_{z,y}- \ln \bar{Z}_{z,y'}^{N,m}
\big| \le r(s_1+s_2) |y-y'|.
\end{equation*}
Similarly, for all $x,x' \in [p,p+1]$, we have
\begin{equation*}
\big| \ln \bar{Z}^{N,n}_{z,x} - \ln \bar{Z}^{N,n}_{z,x'} \big| \le     (rs_1 + |p| + 1 )|x-x'|. 
\end{equation*}
Combining these two inequalities we see that
\begin{equation}\label{eq:Lip-continuous}
  | h_{\varepsilon}^N(x,y) - h^N_{\varepsilon}(x',y') | \le L (|x-x'|+|y-y'|)
\end{equation}
for $L = r(s_1+s_2) + |p| + 1$.
So,  $h_{\varepsilon}^N$ are uniformly Lipschitz continuous and hence equicontinuous on $\tilde{K}$.
It remains to show that $h_{\varepsilon}^{N}$ are uniformly bounded.
Let
\begin{equation*}
 \bar{f}^N(x,y) = \bar{g}^N(x,y)  \bar{Z}^{m,n}_{y,x} =  \int_{-c|N|}^{c|N|} \frac{ \bar{Z}^{N,m}_{z,y} \bar{Z}^{m,n}_{y,x}}{
  \bar{Z}^{N,n}_{z,x}}  \nu_N(dz).
\end{equation*}
Then for fixed $x$, 
\begin{equation*}
\int_{-rs_1}^{rs_1} \bar{f}^N(x,y') dy' = 1.
\end{equation*}
By (\ref{eq:Lip-continuous}), we have for $y,y' \in [-rs_1,rs_1]$, 
\begin{equation*}
\bar{g}^N(x,y) e^{-L\cdot 2rs_1} \le \bar{g}^{N}(x,y') \le \bar{g}^N(x,y) e^{L\cdot 2rs_1}.
\end{equation*}
Let $M$ be the supremum of $|\ln \bar{Z}^{m,n}_{\,\cdot,\, \cdot}|$ on $\tilde{K}$.
Then
\begin{equation*}
\bar{f}^N(x,y) e^{-L\cdot2rs_1-2M} \le \bar{f}^{N}(x,y') \le \bar{f}^N(x,y) e^{L\cdot 2rs_1+2M}.
\end{equation*}
Integrating this inequality over $y'\in [-rs_1, rs_1]$ gives us
\[
2rs_1\bar{f}^N(x,y) e^{-L\cdot 2rs_1-2M} \le 
1 
\le 2rs_1 \bar{f}^N(x,y) e^{L\cdot 2rs_1+2M}.
\]
and hence $\ln \bar{f}^N(x,y)$  are uniformly bounded on $\tilde{K}$.
Therefore, $$h^N_{\varepsilon}(x,y) = \ln \bar{f}^N(x,y) - \ln \bar{Z}^{m,n}_{y,x} $$ are also uniformly bounded.
\end{proof}

\subsection{Existence of global solutions}
\label{sec:existence-of-global-solution}
In this section, for every $v\in\R$, we will prove the existence of global solutions on a full measure set
$\tilde{\Omega} \cap \Omega_v'$. Here, $\Omega_v'$ has been introduced in the beginning
of Section~\ref{sec:limits-of-partition-function-ratios} and $\tilde{\Omega}$ is introduced
in the following lemma controlling  the tail of $\mu_{Nv,x}^{N,n} \pi_{n-1}^{-1}$.
\begin{lemma}
  \label{lem:small-first-step-almost-surely-all-range}
  There is a full measure set $\bar{\Omega}$ on which for every $c > 0$ and $(n,q) \in \Z\times
  \Z$, there are constants $a_1, a_2,L_0 > 0$ and  $N_0$ 
  depending on $c$, $n$ and $q$ such that 
  \begin{equation}
\label{eq:small-first-step-almost-surely-all-range}
\mu_{\nu, x}^{N, n} \pi_{n-1}^{-1} \bigl( [-L,L]^c \bigr) \le 4\nu \bigl( [-c|N|,c|N|]^c \bigr) +
a_1 e^{-a_2 \sqrt{L}}
\end{equation}
for any $N \le N_0$, $L \ge L_0$, $x \in [q,q+1]$ and any terminal measure $\nu$.
\end{lemma}
A proof of the lemma will be given at the end of this section.

Let us fix $v\in\R$ and assume that $\omega \in \bar{\Omega} \cap \Omega_v'$ throughout this section.

Let us define $u^N_v(n,x)$, its Hopf--Cole
transform $V^N_v(n,x)$, and the constants $C^N_{v,m,n}$
by~\eqref{eq:Hopf-Cole-for-prelimit},~\eqref{eq:Normalizing-constants-product-cocycle}, and \eqref{eq:approx-solution-via-polymer-measures}.
We can use the function $G_v$ introduced in Theorem \ref{thm:convergence-of-partition-function-ratios} to define 
\begin{equation*}
  V_v(n,x) = G_v \bigl( (n,x), (n,0) \bigr) 
,\quad  C_{v,m,n} = G_v \bigl( (m,0), (n,0)  \bigr).
\end{equation*}
\begin{lemma}
  \label{lem:burgers-in-HC-satisfied}
The functions $V_v(n,x)$ and constants $C_{v,m,n}$ satisfy \eqref{eq:evolution-for-V}.
\end{lemma}

\begin{proof}
  Fix $m<n$ and $x$.
  We want to show 
\begin{equation*}
  G_v \bigl(  (n,x), (n,0) \bigr)=
  G_v \bigl( (m,0), (n,0) \bigr)  \int Z^{m,n}(y,x) G_v \bigl(  (m,y), (m,0) \bigr)dy,
  \end{equation*}
which, by Theorem \ref{thm:convergence-of-partition-function-ratios}, is equivalent to 
\begin{equation*}
1 = \int Z^{m,n}_{y,x} G_v \bigl( (m,y), (n,x) \bigr)dy.
\end{equation*}
This identity is true because by Theorem \ref{thm:convergence-of-density-functions}, the integrand
is the density of $\mu_x^{-\infty,n}(v) \pi_m^{-1}$.
\end{proof}

Let $f^N_{v,n,n-1}(x,y)$ be the density of $\mu_{Nv,x}^{N,n}\pi_{n-1}^{-1}$.
Then~\eqref{eq:approx-solution-via-polymer-measures} rewrites as
\begin{equation*}
u^N_v(n,x) = \int_{\R} (x-y) f_{v,n,n-1}^N(x,y) dy.
\end{equation*}
Recalling that we expect the global solution to be given by (\ref{eq:global-solution-via-marginal-polymer-measures}), we use the limiting density $f_{v,n,n-1}(x,y)$ from Theorem \ref{thm:convergence-of-density-functions}
to define
\begin{equation*}
u_v(n,x) = \int_{\R} (x-y) f_{v,n,n-1}(x,y) dy.
\end{equation*}
\begin{lemma}
\label{lem:convergence-of-velocity}
The functions $u_v^N(n,\cdot)$ converge  to  $u_v(n,\cdot)$ as
$N \to -\infty$, uniformly on compacts sets,  and the Hopf--Cole transform of $u_v(n,\cdot)$ is $V_v(n,\cdot)$.
\end{lemma}

\begin{proof}
  Let $q \in \Z$.
  Lemma \ref{lem:small-first-step-almost-surely-all-range} implies that for some constants $a_1,a_2, L_0$ 
and~$N_0$, 
\begin{equation*}
\mu_{Nv,x}^{N,n} \pi_{n-1}^{-1}\bigl( [-L, L]^c \bigr) = \int_{ |y| >  L} f^N_{v,n,n-1}(x,y) dy \le
a_1 e^{-a_2 \sqrt{L}},\quad x\in[q,q+1],
\end{equation*}
for all $N \le N_0$ and $L \ge L_0$, if we take $c > |v|$.
Moreover, by Theorem \ref{thm:convergence-of-density-functions}, 
$f_{v,n,n-1}^N(x,y)$ converges to $f_{v,n,n-1}(x,y)$ uniformly on compact sets.
Therefore $u_v^N(n,\cdot)$ converges to $u_v(n,\cdot)$ uniformly on $[q,q+1]$.

Since $u^N_v(n,\cdot)$ and $V^N_v(n,\cdot)$ converge to $u_v(n,\cdot)$ and $V^N(n,\cdot)$ on compact
sets, taking the limit  $N \to -\infty$ on both sides of 
\begin{equation*}
V^N_v(n,x) = e^{-\int_0^x u^N_v(n,x')dx' },
\end{equation*}
we see that $V_v(n,x)$ is the Hopf--Cole transform of $u_v(n,x)$.
\end{proof}

To show that $u_v(n,\cdot)\in \mathbb{H}'(v,v)$, we  need the following lemma which we will prove in the end of this section.
\begin{lemma}
\label{lem:uniform-integrability-of-uN}
Given $n\in\Z$ and a compact set $K\subset\R$, the family of random variables $\big\{ u^N_v(n,x):\ N<n,\ x\in K\big\}$ is uniformly
integrable.
\end{lemma}

\begin{proof}[Proof of the existence part of Theorem \ref{thm:global_solutions}]
By  Lemmas   \ref{lem:burgers-in-HC-satisfied} and  \ref{lem:convergence-of-velocity}, $u_v(n,x)$ is a
global solution.
It remains to show that $u_v(n,\cdot) \in \mathbb{H}'(v,v)$.
All the other properties are easy to check.

Lemma \ref{lem:uniform-integrability-of-uN} implies that 
\begin{equation}\label{eq:L-one-limit}
\lim_{N\to -\infty} \E u_v^N(n,x) = \E u_v(n,x).
\end{equation}
By Lemma \ref{lem:shear-for-Z_v}, for any $(m_1,x_1)$ and $(m_2,x_2)$
such that $m_1<m_2$, we have
\begin{equation*}
  Z^{m_1,m_2}(x_1,x_2) \stackrel{d}{=} e^{-\frac{(x_1-x_2)^2}{2(m_2-m_1)}} Z^{0,m_2-m_1}(0,0).
\end{equation*}
Taking logarithm and then expectation, we obtain
\begin{equation*}
\E \ln Z^{m_1,m_2}(x_1,x_2) = -\frac{(x_1-x_2)^2}{2(m_2-m_1)} + \E \ln Z^{0,m_2-m_1}(0,0),
\end{equation*}
so
\begin{equation}
  \label{eq:expectation-of-V}
\E \ln V^N_v(n,x)  = \E \ln Z^{N,n}(Nv,x) - \E \ln Z^{N,n}(Nv,0) = 
-\frac{x(x-2Nv)}{2(n-N)}.
\end{equation}
For any $N$, by Hopf--Cole transform we have 
\begin{equation*}
\int_0^{x} u_v^N(n,x')dx' = - \ln V^N_v(n,x).
\end{equation*} 
Taking expectation of both sides, using the Fubini theorem and (\ref{eq:expectation-of-V}), we obtain
\begin{equation*}
\int_0^x \E u_v^N(n,x') dx' = \frac{x(x-2Nv)}{2(n-N)}.
\end{equation*}
Taking the limit $N \to -\infty$ and using \eqref{eq:L-one-limit}, we obtain
\begin{equation*}
\int_0^x \E u_v(n,x')dx' = vx.
\end{equation*}
By stationary of $u_v(n,\cdot)$, the left hand side is $x \cdot \E u_v(n,0)$.
Therefore,  $\E u_v(n,0)=v$ and  hence by ergodic theorem $u_v(n,\cdot) \in \mathbb{H}'(v,v)$.
\end{proof}

Now we turn to the proofs of Lemma \ref{lem:small-first-step-almost-surely-all-range} and Lemma
\ref{lem:uniform-integrability-of-uN}.
The following lemma  generalizes Lemma \ref{lem:large-deviation-for-Z-bar}:
\begin{lemma}
\label{lem:small-probability-for-first-large-step-outside-linear-box}
There are constants $d_1>0$ and $R_0$ such that for all $r \ge R_0$, $(m,p), (n,q) \in \Z
\times \Z$ ($n-m\ge 2$), with probability at least $1- e^{-d_1 r|m-n|}$,
\begin{equation*} \mu_{y,x}^{m,n} \Big\{ \gamma: \max_{m\le i \le n} | \gamma_{i}-[(m,y), (n,x)]_i| \ge
r|n-m| \Big\} \le e^{-r|m-n|}
\end{equation*} for all $x \in [q,q+1], y \in [p,p+1]$.
\end{lemma}

\begin{proof}
  Due to shear invariance,  without loss of generality we can assume $m=p=q=0$.
  Lemma \ref{lem:main-monotonicity} implies $\mu_{0,0}^{0,n} \preceq \mu_{y,x}^{0,n} \preceq
  \mu_{1,1}^{0,n}$ for any $x, y \in [0,1]$.
  Therefore, it suffices to show
  \begin{equation}\label{eq:small-prob-outside-linear-box-2}
  \Pp \biggl\{\mu_{0,0}^{0,n} \Big\{ \gamma: \max_{0\le i \le  n } |
  \gamma_{i}| \ge rn/2
  \Big\} > \frac{1}{2}e^{-rn}
  \biggr\} \le e^{-k_1 r n}
\end{equation}
for some constant $k_1$ and sufficiently large $r$.

  Let $A(r) = \{ \gamma: \max\limits_{0 \le i \le n} |\gamma_i| \ge \frac{1}{2}rn \} $.
  Repeating the computation in \eqref{eq:nominator-not-too-big}, we see that
for sufficiently large $r$ and some constant $k_2$,
  \begin{equation*}
  \Pp \left\{ Z^{0,n}_{0,0} \bigl( A(r)  \bigr) > \frac{1}{2}\rho_0^n e^{-2rn} \right\} 
    < \frac{2n\lambda^n e^{- (r/2)^2n}}{\pi r \sqrt{n-1} \rho_0^ne^{-2rn}} \le
    e^{-k_2 r n}.
\end{equation*}  
By Lemma \ref{lem:partition-function-not-too-small}, for sufficiently large $r$ and a constant $k_3$, we have
\begin{equation*}
      \Pp \left\{ Z^{0,n}_{0,0} \le \rho_0^ne^{-rn} \right\} \le \beta^n e^{-rnr_0} \le
      e^{-k_3 rn}.
\end{equation*}
Now (\ref{eq:small-prob-outside-linear-box-2}) follows from these two estimates.
\end{proof}

The next statement is a direct consequence of Lemma~\ref{lem:small-probability-for-first-large-step-outside-linear-box}.
\begin{lemma}
  \label{lem:small-probability-for-first-large-step-outside-linear-box-all-r}
There are constants $d_1,d_2,R_0>0$  such that for all  $(m,p), (n,q) \in \Z
\times \Z$ ($n-m\ge 2$), with probability at least $1- d_2e^{-d_1 R_0|m-n|}$,
\begin{equation*} \mu_{y,x}^{m,n} \Big\{ \gamma: \max_{m\le i \le n} | \gamma_{i}-[(m,y), (n,x)]_i| \ge
r|n-m| \Big\} \le e^{-r|m-n|}
\end{equation*}
 for all $x \in [q,q+1], y \in [p,p+1]$ and $r \ge R_0$, $r \in \N$.
\end{lemma}

\begin{proof}[Proof of Lemma~\ref{lem:small-first-step-almost-surely-all-range}]
  It suffices to prove the statement for fixed $c$ and~$(n,q)$.  
  Let  $K = 2c+R+1$.
  Lemma \ref{lem:one-step-control} implies that with probability one,  if $\frac{n-N}{2} \ge n_1=n_1(n,q,c)$, then for all $s$
satisfying $n_1  \le s \le \frac{n-N}{2}$ and all $x\in[q,q+1]$, 
\begin{equation*}
    \mu_{\nu,x}^{N,n} \pi_{n-1}^{-1} \bigl( [-Ks, Ks]^c \bigr)
    \le 4\nu \bigl( [-c |N|, c|N| ]^c \bigr) + 8e^{-\sqrt{s}}.
  \end{equation*}
This implies that for some $k_1>0$ and all $L \in [Kn_1, K(n-N)/2]$, we have
\begin{equation}
  \label{eq:less-than-linear}
\mu_{\nu, x }^{N,n } \pi_{n-1}^{-1} \bigl( [-L,L]^c \bigr) \le 4 \nu \bigl(  [-c|N|, c|N|]^c \bigr)
+ 8 e^{-k_1\sqrt{L}}.
\end{equation}
Noticing that for all $L \ge K(n-N)/2$, we have the trivial inequality 
\begin{equation*}
\mu_{\nu,x}^{N,n}\pi_{n-1}^{-1}\big( [-L,L]^c \big)  \le
\mu_{\nu,x}^{N,n}\pi_{n-1}^{-1}\big( [-K(n-N)/2,K(n-N)/2]^c \big),
\end{equation*}
we can extend \eqref{eq:less-than-linear} to all $L \in  [Kn_1, 2R_0(n-N)]$ by adjusting the constant $k_1$ appropriately. 
Here $R_0$ is taken from Lemmas~\ref{lem:small-probability-for-first-large-step-outside-linear-box}    and~\ref{lem:small-probability-for-first-large-step-outside-linear-box-all-r}.

The Borel--Cantelli lemma implies that with probability one, the statement of 
Lemma~\ref{lem:small-probability-for-first-large-step-outside-linear-box-all-r} holds true for all
sufficiently negative $m=N$ and $p \le c|N| + 1$.
In particular, for sufficiently negative $N$, we have
\begin{equation*}
  \mu_{y,x}^{N,n} \pi_{n-1}^{-1} \Bigl(  \Big[x+\frac{y-x}{n-N}-r|n-N|,\, x+\frac{y-x}{n-N}+r|n-N|\Big]^c \Bigr) \le e^{-r|n-N|}
\end{equation*}
for all $|y| \le c|N|$ and $r \ge R_0$. Applying this estimate to $y=\pm c|N|$ and using monotonicity, we
obtain for  $r \ge R_0$ and sufficiently negative~$N$:
\begin{align*}
  \notag
  & \mu_{\nu,x}^{N,n} \pi_{n-1}^{-1} \bigl(  [-2r|n-N|, 2r|n-N|]^c \bigr) 
  \\
  \notag 
\le&\mu_{\nu,x}^{N,n} \pi_{n-1}^{-1} \Bigl(  \big[-(|q|+c+2+r|n-N|), |q|+c+2 + r|n-N|\big]^c \Bigr) 
\\ 
\le&
 \mu_{\nu,x}^{N,n} \pi_{n-1}^{-1} \Bigl(  \Big[x+\frac{-c|N|-x}{n-N}-r|n-N|,\, x+\frac{c|N|-x}{n-N}+r|n-N|\Big]^c \Bigr)
\\ \notag
 \le &\nu \bigl( [-c|N|,
  c|N|]^c \bigr)+ 2e^{-r|n-N|}.
\end{align*}
Therefore, for some constant $k_2 > 0$ and all $L \in [2R_0(n-N), +\infty)$, we have 
\begin{equation}
  \label{eq:more-than-linear}
\mu_{\nu,x}^{N,n} \pi_{n-1}^{-1} \bigl( [-L, L]^c \bigr) \le \nu \bigl(  [-c|N|, c|N|]^c \bigr) +
2e^{- k_2 L}.
\end{equation}
Combining the estimates (\ref{eq:less-than-linear}) and (\ref{eq:more-than-linear}), we see that
\eqref{eq:small-first-step-almost-surely-all-range} holds for all $L\ge Kn_1$, which completes
the proof of the lemma.
\end{proof}

%
To prove the uniform integrability of $u^N_v(n,x)$ in Lemma~\ref{lem:uniform-integrability-of-uN},  we need an additional lemma:
\begin{lemma}
\label{lem:small-probability-of-first-large-step}
There is a constant $s_0$ such that for $|N|/2 \ge s \ge s_0$, 
\begin{equation*}
\Pp \left\{
\mu_{0,0}^{N,0} \pi^{-1}_{-1} ([-(R+2)s, (R+2)s]^c) \le 4e^{-\sqrt{s}}
\right\}
> 1 - e^{-\sqrt{s}}.
\end{equation*}
\end{lemma}

\begin{proof}
 In Lemma \ref{lem:sets-for-remote-control-to-holds}, we take $\delta = 1/9$, $\beta = 1/2$, $v=1$, $(m,q) = (0,0)$ and consider
 the set $\Omega_2^{(s)}(1/9, 1, 0,0)$.
 The lemma follows from the backward version of (\ref{eq:compactness-control-by-terminal-measure}) with $v_0=v'=0$, $v_1=1$
 and $\nu = \delta_0$.
\end{proof}

\begin{proof}[Proof of Lemma \ref{lem:uniform-integrability-of-uN}]
  By Lemma~\ref{lem:monotonicity_of_x-ux},
\begin{equation*}
u^N_v(n,x) - x = -\int_{\R} y  \mu_{Nv, x}^{N,n} (dy)
\end{equation*}
is non-increasing in $x$.
Therefore, it suffices to show the uniform integrability of $ \bigl(  u^N_v(n,x) \bigr)_{N < n}$
for fixed $(n,x) \in \ZR$.
We also notice that $u^N_v(0,0) \stackrel{d}{=} u^N_0(0,0) + v$.
So, without loss of generality, let us assume $(n,x) = (0,0)$ and $v=0$.
Let us write $f^N_{0,0,-1}(0,y) = f^N(y)$ and $u^N_0(0,0) = u^N$.

Lemma~\ref{lem:small-probability-of-first-large-step} implies that if $L=(R+2)s \in [(R+2)s_0,
(R+2)|N|/2]$, then
\begin{equation}
  \label{eq:small-probability-first-step-less-than-linear}
 \Pp \Big\{  \int_{|y| > L } f^N(y) dy \le 4e^{-k_1\sqrt{L}} \Big\} > 1 - e^{-k_2\sqrt{L}}
\end{equation}
for some constants $k_1$ and $k_2$.
Using the inequality
$$\int_{|y| > L} f^N(y) dy \le \int_{|y| > (R+2)|N|/2} f^N(y) dy$$
for $L \ge (R+2)|N|/2$ and adjusting the constants $k_1$, $k_2$ appropriately, we can extend~(\ref{eq:small-probability-first-step-less-than-linear}) to all $L \in [(R+2)s_0, R_0|N|]$.
Here, $R_0$ is taken from Lemma~\ref{lem:small-probability-for-first-large-step-outside-linear-box}.
Next, Lemma~\ref{lem:small-probability-for-first-large-step-outside-linear-box} implies that  if
$L =r|N| \ge R_0|N|$, then
\begin{equation}\label{eq:small-probability-of-first-large-step-larger-than-linear}
\Pp \Big\{  \int_{|y| > L} f^N(y) dy  \le e^{- L} \Big\}  > 1 - e^{-d_1L}.
\end{equation}
Combining the estimates~(\ref{eq:small-probability-first-step-less-than-linear}) and (\ref{eq:small-probability-of-first-large-step-larger-than-linear}), 
we can find constants $c_1,c_2,c_3,c_4$,  independent
of $N$,  such that for $L \ge (R+2)s_0$,
\begin{equation*}
\Pp \Big\{ \int_{|y| > L} f^N(y) dy \le c_1 e^{-c_2\sqrt{L}}  \Big\} > 1 - c_3 e^{-c_4 \sqrt{L}}.
\end{equation*}
This implies that $u^N = -\int_{\R}y  f^N(y)  dy$ are uniformly integrable.
\end{proof}

\subsection{Uniqueness of global solutions}\label{sec:uniqueness-of-global-solutions}
The main goal of this section is to finish the proof of Theorem~\ref{thm:global_solutions} by establishing the uniqueness
of global solutions.

Let $w(x) \in \mathbb{H}'$ and $V(x) = e^{-\int_0^x w(x')dx'}$ be its Hopf--Cole transform.
We can introduce the following point-to-line polymer measures:
\begin{equation*}
  \bar{\mu}_{V,x}^{N,n} (A_N \times ... \times A_{n-1})
= \frac{ \int_{A_N} dx_N \cdots \int_{A_{n-1}} dx_{n-1} \delta_{x}(dx_n) \ V(x_N)
  \prod\limits_{i=N}^{n-1} Z^{i,i+1}_{x_i, x_{i+1}}
}{\int_{\R} V(x_N) Z^{N,n}_{x_N,x}dx_N
}.
\end{equation*}
The fact that $w \in \mathbb{H}'$ guarantees that all integrals are finite.

\begin{lemma}
  \label{lem:condition-second-kind-polymer-measure-converge-to-LLN}
  Let $\big(w_N(\cdot) \big)$ be a stationary sequence of random functions in $\mathbb{H}'$ and
  $\bigl( V_N (\cdot) \bigr)$ be the corresponding Hopf--Cole transforms.
  Let $v \in \R$.
Suppose that one of the conditions \eqref{eq:no_flux_from_infinity}, \eqref{eq:flux_from_the_left_wins}, \eqref{eq:flux_from_the_right_wins} 
is satisfied by $W(\cdot)=W_N(\cdot) = \int_0^{\cdot} w_N(y')dy'$ for all $N$ with probability~$1$.	
Then for almost every $\omega$
and all $n \in \Z$, the probability measures $\nu_{N,n,x} (N < n)$ defined by
\begin{equation*}
\nu_{N,n,x}(dy) = \bar{\mu}_{V_N, x}^{N,n}\pi_N^{-1}(dy) = \frac{  Z^{N,n}(y,x) V_N(y)}{\int_{\R} 
  Z^{N,n}(y',x) V_N(y')dy'} dy
\end{equation*}
satisfy 
\begin{equation*}
\liminf_{N\to-\infty} e^{h|N|}\sup_{x \in [-L,L]} \nu_{N,n,x}( [(v+\varepsilon)N, (v-\varepsilon)N]^c ) = 0,
\end{equation*}
for all $L\in \N$ and  $\varepsilon > 0$, and some constant $h(\varepsilon) > 0$ depending on $\varepsilon$.
\end{lemma}
We give the proof of this lemma after we derive uniqueness from it.

\begin{proof}[Proof of the uniqueness part of Theorem \ref{thm:global_solutions}]
  Let $v \in \R$ and let $u_{v}(n,\cdot)$ be a stationary global solution in $\mathbb{H}'(v,v)$. 
We will prove that for almost every~$\omega$,  $u_{v,\omega}$ coincides with the global
solution constructed in Section~\ref{sec:existence-of-global-solution}.

Let  $V_v(n,\cdot)$ be the Hopf--Cole transforms of $u_v$ and $C_{v,m,n}$ be
the family of constants such that (\ref{eq:evolution-for-V}) holds true.
Let $\bar{\mu}_x^{-\infty,n}$ be defined as in \eqref{eq:def-of-mu-bar}.
Then we have~\eqref{eq:solution-through-pt-to-line-measure}.

Since $u_v(n,x) \in \mathbb{H}'(v,v)$, the potential of $u_v(n,x)$ satisfies one of the conditions
 \eqref{eq:no_flux_from_infinity}, \eqref{eq:flux_from_the_left_wins}, \eqref{eq:flux_from_the_right_wins}
depending on the value of $v$.
Therefore, by Lemma~\ref{lem:condition-second-kind-polymer-measure-converge-to-LLN}, we have
\begin{equation*}
  \liminf_{N \to -\infty} \bar{\mu}_x^{-\infty,n} \pi_N^{-1}( [(v+\varepsilon)N, (v-\varepsilon)N ]^c)= 0.
\end{equation*}
By Theorem~\ref{thm:all-space-time-point-compactness} we have that for $m$ large enough and $n-N
\ge 2m$,
\begin{multline*}
  \bar{\mu}_x^{-\infty,n} \pi_{n-m}^{-1}( [ (v+2\varepsilon)(n-m), (v-2\varepsilon)(n-m)]^c ) 
\\  \le 4 \bar{\mu}_x^{-\infty,n} \pi_N^{-1} ( [(v+\varepsilon)N, (v-\varepsilon)N ]^c) + 6 e^{-\sqrt{m}}.
\end{multline*}
Taking $\liminf$ as $N\to -\infty$,  we obtain
\begin{equation*}
  \bar{\mu}_x^{-\infty,n} \pi_{n-m}^{-1}( [ (v+2\varepsilon)(n-m), (v-2\varepsilon)(n-m)]^c ) \le 6e^{-\sqrt{m}}.
\end{equation*}
So $\bar{\mu}_x^{-\infty,n}$ satisfies SLLN with slope $v$ and is supported on $S_x^{-\infty,n}$.
Therefore,
by Lemma \ref{lem:uniqueness-of-polymer-measures}, we have $\bar{\mu}_x^{-\infty,n} = \mu_x^{-\infty,n}(v)$.
This shows that $u_v(n,\cdot)$ is exactly what we have constructed in Section~\ref{sec:existence-of-global-solution}, and the proof of uniqueness is complete.
\end{proof}

To prove  Lemma \ref{lem:condition-second-kind-polymer-measure-converge-to-LLN} we first need several auxiliary statements.
\begin{lemma}
\label{lem:property-on-stationary-rvs}
Let $(X_n)_{n\in\N}$ be a stationary sequence of  random variables such that $\Pp
(X_n < \infty) = 1$.
Then there is a random number $k=k(\omega)$ such that 
\begin{equation*}
  \Pp \left\{ \omega: X_n(\omega) \le k(\omega) \quad \text{\rm for infinitely many $n$}
  \right\} = 1.
\end{equation*}
\end{lemma}
\begin{proof}
Let $A_k = \{ \omega: X_n(\omega) \le k \quad \text{for finitely many $n$} \}$.
Clearly $A_{k+1} \subset A_{k}$ for all $k\in\N$.
Let $A_{\infty} = \bigcap\limits_{k=1}^{\infty} A_{k}$.
We want to prove that $\Pp (A_{\infty}) = 0$.

By the ergodic theorem, on $A_{\infty}$, for any $k$, we have
\begin{equation*}
0 = \lim_{n \to \infty} \frac{1}{n}\sum_{i=0}^{n-1} \ONE_{X_i \le k} = \E
(\ONE_{X_0 \le k} | \mathcal{I}),
\end{equation*}
where $\mathcal{I}$ is the invariant $\sigma$-algebra for the
stationary sequence $\big(X_n\big)$.
Therefore 
\begin{equation*}
0 = \E \big(\ONE_{A_{\infty}} \E (\ONE_{X_0 \le k}| \mathcal{I})  \big) = \E
\ONE_{A_{\infty}}\ONE_{X_0 \le k}.
\end{equation*}
Since $\Pp(X_0 < \infty)=1$, by the Bounded Convergence
Theorem we have 
\begin{equation*}
0 = \lim_{k\to \infty} \E \ONE_{A_{\infty}}\ONE_{X_0 \le k} = \Pp (A_{\infty})
\end{equation*}
as desired.
\end{proof}

\begin{lemma}
  \label{lem:truncation-and-concentration-almost-surely-true}
There is a full measure set $\Omega''$ such that the following is true for every $\omega \in \Omega''$.
For all  $c > 4\sqrt{\ln (2\lambda/\rho_0)})$ and $(n,q) \in \Z \times \Z$, 
there is a constant $m_0 = m_0(n,q,c)$ such that for all $m > m_0$, we have
\begin{equation}
  \label{eq:truncation-almost-surely}
 \frac{\int_{|y| \ge cm}  Z^{n-m,n}(y,q)e^{c|y|/17} dy   }{\int_q^{q+1} Z^{n-m,n}(y,q)dy} \le 2^{-m}
\end{equation}
and 
\begin{equation}
  \label{eq:concentration-almost-surely}
\Big| \ln \int_I Z^{n-m,n}(y,x)H(y) dy -  \ln \int_I e^{\alpha(m,x-y)} H(y) dy \Big| \le 2m^{3/4},
\end{equation}
for all intervals $I \subset
[-cm,cm]$, all $x \in [q,q+1]$ and all positive functions $H(\cdot)$.
Here, $\alpha(\cdot, \cdot)$ has been defined in~\eqref{eq:space-time-shape}.
\end{lemma}
We will prove this lemma after we use it to derive Lemma~\ref{lem:condition-second-kind-polymer-measure-converge-to-LLN}.

We also need a monotonicity statement about point-to-line polymer measures.
\begin{lemma}
  \label{lem:monotonicity-for-second-kind-polymer-measure}
  Let $x<x'$ and $V(x)$ be a positive function that grows at most exponentially.
  Then for any $m,n$ with $m<n$, the polymer measure $\bar{\mu}_{V,x}^{m,n}$ is stochastically dominated by $\bar{\mu}_{V,x'}^{m,n}$.
\end{lemma}

\begin{proof}
First, we have
\begin{equation*}
\mu_{V,y}^{m,k}(A_m \times \dotsb \times A_{k-1}) = \int_{A_{k-1}} \bar{\mu}_{V,y}^{m,k}\pi_{k-1}^{-1}(dx_{k-1})
\bar{\mu}_{V,x_{k-1}}^{m,k-1}(A_m \times \dotsb \times A_{k-2}).
\end{equation*}
Therefore, similarly to Lemma \ref{lem:dominance-1}, it suffices to show that
 $\bar{\mu}_{V, x}^{m,n}\pi_{n-1}^{-1} \preceq \bar{\mu}_{V,
   x'}^{m,n}\pi_{n-1}^{-1}$
and use an induction argument.

Now we compute the marginals at time $n-1$ : 
\begin{align*}
 \bar{\mu}_{V,x}^{m,n}\{X_{n-1}\le r\}
  &= \frac{ \int_{\R} dy\int_{(-\infty,r]} d\eta V(y)Z_{y,\eta}^{m,n-1} e^{-F_{n-1}(\eta)} g(\eta-x)  }{ \int_{\R}
    dy\int_{\R} d\eta V(y)Z_{y,\eta}^{m,n-1} e^{-F_{n-1}(\eta)} g(\eta-x)  } .
\end{align*}
Let
\begin{equation*}
\nu(d\eta) = \int dy V(y) Z^{m,n-1}_{y,\eta} e^{-F_{n-1}(\eta)} d\eta.
\end{equation*}
Then by Lemma \ref{lem:abstract-monotonicity},  $\bar{\mu}_{V,x}^{m,n}\{X_{n-1}\le r\}$ is decreasing in
$x$, so  $\bar{\mu}_{V, x}^{m,n}\pi_{n-1}^{-1}$  is dominated by  $\bar{\mu}_{V,
  x'}^{m,n}\pi_{n-1}^{-1}$.
\end{proof}

\begin{proof}[Proof of Lemma~\ref{lem:condition-second-kind-polymer-measure-converge-to-LLN}]
We take $\Omega''$ from the statement of Lemma
\ref{lem:truncation-and-concentration-almost-surely-true} and fix an arbitrary $\omega \in \Omega''$.

  Fix $n$ and $L \in \N$.
  Since Lemma \ref{lem:monotonicity-for-second-kind-polymer-measure} implies $\nu_{N,n,-L} \preceq
  \nu_{N,n,x} \preceq \nu_{N,n,L}$ for $x \in [-L, L]$ , it suffices to show that  for every $\varepsilon
  > 0$, 
\begin{equation*}
\liminf_{N\to -\infty} e^{h|N|}\max_{a = \pm L} \nu_{N,n,a} \big( [(v+\varepsilon)N,
(v-\varepsilon)N ]^c \big) = 0,
\end{equation*}
or, equivalently, that for every
$ \varepsilon \in (0,1)$ there is a random sequence $m_k = m_k(\omega,\varepsilon)   \uparrow + \infty$  such that 
\begin{equation}
  \label{eq:what-to-prove}
  \begin{split}
0 &=\lim_{k\to \infty}e^{hm_k} \nu_{n-m_k,n,a}\bigl( [-(v+\varepsilon)m_k, -(v-\varepsilon)m_k]^c
    \bigr)\\ 
&= \lim_{k\to\infty}  e^{hm_k}\frac{\int_{|y+vm_k| > \varepsilon m_k} Z^{n-m_k,n}(y,a) V_{n-m_k} (y) dy  }{ \int_{\R}
  Z^{n-m_k,n}(y,a) V_{n-m_k}(y)dy}
\end{split}
\end{equation}
for $a = \pm L$.

The proof consists of two steps.
The first step is to use  \eqref{eq:no_flux_from_infinity},  \eqref{eq:flux_from_the_left_wins},
\eqref{eq:flux_from_the_right_wins} and Lemma
\ref{lem:property-on-stationary-rvs} to find a random sequence $(m_k)$ with certain properties;
the second is to combine those properties and estimates provided by  Lemma
\ref{lem:truncation-and-concentration-almost-surely-true} to derive~\eqref{eq:what-to-prove}.

We can assume that $v\ge 0$, since the case $v<0$ is totally symmetric to the case $v>0$.
Let us fix some $\delta > 0$ such that
\begin{equation}
  \label{eq:choice-of-delta}
  \delta < 
\begin{cases}
\varepsilon/4,  & v =0, \\
(\varepsilon/4) \wedge (v/2) \wedge \frac{\varepsilon^2}{8v}, & v > 0.
\end{cases}
\end{equation}

\textbf{Step 1 --- find $(m_k)$:} we claim that there is a random constant $R=R(\omega)$ and a random sequence
  $(m_k)$ such that for every $m = m_k$, 
\begin{equation}
  \label{eq:boundedness-of-W}
  |W_{n-m}(y)|  \le R, \quad  y \in [-L,L+1], 
\end{equation}
\begin{equation}
  \label{eq:W-finite-slope}
  W_{n-m}(y) \ge -R(|y|+1) , \quad  y \in \R,
\end{equation}
and
  \begin{subequations}
    \label{eq:W-asymtopic-slope-at-infinity}
    \begin{align}
    \nonumber  v&= 0, \\ 
      \label{eq:zero-v-infinity-slope}
      W_{n-m}(y) &\ge -\delta|y|, && |y| \ge R, \\ 
     \nonumber   \text{or,}& \\ 
      \nonumber v &> 0, \\
      \label{eq:positive-v-negative-infinite-slope}
      |W_{n-m}(y)  - vy| &\le \delta|y|, && y < -R,\\
      \label{eq:positive-v-plus-infinity-slope}
        W_{n-m}(y)  &\ge (-v+2\delta)|y|, && y > R.
    \end{align}
  \end{subequations}

To see this, for each $m$, we let $X_m$, $Y_m$ and $Z_m$ be the infimum of $R$ such that 
\eqref{eq:boundedness-of-W}, \eqref{eq:W-finite-slope} and
\eqref{eq:W-asymtopic-slope-at-infinity} are satisfied.
Due to stationary of $W_N(\cdot)$, $(X_m)$, $(Y_m)$ and $(Z_m)$ are all stationary sequences of random variables.
Also, 
$X_m $ are a.s. finite because $W_N(\cdot)$ are locally finite;
$Y_m $ are a.s. finite because $W_N(\cdot) \in \mathbb{H}$; 
$Z_m$ are a.s. finite due to~\eqref{eq:no_flux_from_infinity}
or~\eqref{eq:flux_from_the_left_wins}, depending on $v$.
Therefore, by Lemma \ref{lem:property-on-stationary-rvs}, there 
is a random number $R=R(\omega)$ such that $X_m\vee Y_m\vee Z_m \le R$ for infinitely many $m$
almost surely.
This proves the claim.

\textbf{Step 2 --- show (\ref{eq:what-to-prove})}.
For simplicity we will write $m = m_k$ in what follows, so $m \to \infty$ actually means $m=m_k$, $k \to \infty$.
Let us fix $c \ge 17( v+1)  \vee 4 \sqrt{\ln(2\lambda/\rho_0)}$ 
and write
\begin{align*}
  &\quad\nu_{n-m,n,a}([-(v+\varepsilon) m, -(v-\varepsilon) m]^c)  \\
  &= \frac{\int_{|y|<cm, |y+vm| > \varepsilon m} Z^{n-m,n}_{y,a}e^{-W_{n-m}(y) } dy  }{ \int_{\R}
    Z^{n-m,n}_{y,a}e^{-W_{n-m}(y) }dy}
  + \frac{\int_{|y| \ge cm} Z^{n-m,n}_{y,a}e^{-W_{n-m}(y) } dy  }{ \int_{\R}
    Z^{n-m,n}_{y,a} e^{-W_{n-m}(y) } dy} \\
  &= A^m  + B^{m}.
\end{align*}
We will show that both $A^m$ and $B^m$  decay exponentially.

First we look at $B^m$.
By \eqref{eq:W-asymtopic-slope-at-infinity}, if $m$ is sufficiently large, then 
$-W_{n-m}(y)  \le (|v|+\delta) |y| \le c|y|/17$ for all $|y| \ge cm$.
Due to \eqref{eq:boundedness-of-W} we have
\begin{equation*}
\int_{\R} Z^{n-m,n}(y,a)e^{-W_{n-m}(y) }dy \ge e^{-R}\int_a^{a+1} Z^{n-m,n}(y,a)dy.
\end{equation*}
Therefore, by Lemma \ref{lem:truncation-and-concentration-almost-surely-true} we have
\begin{equation*}
B^m \le e^R\frac{\int_{|y| \ge cm} Z^{n-m,n}(y,a)e^{c|y|/17 } dy  }{ \int_a^{a+1}
    Z^{n-m,n}(y,a) dy} \le \frac{e^R}{2^m}
\end{equation*}
for sufficiently large $m$.

Next we look at $A^m$.
Using   Lemma~\ref{lem:truncation-and-concentration-almost-surely-true}, we obtain that for sufficiently large~$m$, 
\begin{align*}
A^m & \le \exp\big( 4m^{3/4}\big)\cdot \frac{\int_{|y|<cm, |y+vm| > \varepsilon m} e^{-\frac{(y-a)^2}{2m}-W_{n-m}(y) } dy  }{ \int_{-cm}^{cm}
      e^{-\frac{(y-a)^2}{2m}-W_{n-m}(y) }dy} \\
  &\le \exp\big( 4m^{3/4} + L^2/m + 2Lc \big)\cdot \frac{\int_{|y+vm| > \varepsilon m} e^{-\frac{y^2}{2m}-W_{n-m}(y) } dy  }{ \int_{-cm}^{cm}
    e^{-\frac{y^2}{2m}-W_{n-m}(y) }dy}.
  \end{align*}
Let us denote the ratio of integrals in the last line by $\tilde{A}^m$.
  It suffices to show that~$\tilde{A}^m$ decays exponentially.
We will consider the cases $v=0$ and $v>0$ separately.

Suppose $v=0$.
For sufficiently large $m$, we have $W_{n-m}(y) \ge -\delta |y|$ for all $|y| > \varepsilon m$ by
\eqref{eq:zero-v-infinity-slope} and $W_{n-m}(y) \le R$ for $y \in [0,1]$ by \eqref{eq:boundedness-of-W}.
Therefore, 
\begin{equation*}
  \tilde{A}^m \le 
  \frac{\int_{|y|> \varepsilon m } e^{- \frac{y^2}{2m} + \delta |y| } dy}{ \int_0^1
  e^{-\frac{y^2}{2m} - R} }
   \le e^{\frac{1}{2m} + R}    \int_{|y| > \varepsilon m }     e^{-\frac{y^2}{2m} + \delta |y| } dy
  \le e^{\frac{1}{2m} + R}  \cdot  \frac{4}{\varepsilon} e^{- \frac{\varepsilon^2}{4}m}
\end{equation*}
as desired.
Here,  in the last inequality, we used $\delta < \varepsilon /4$ to obtain
\begin{equation}
\label{eq:gaussian-integral-estimate}
\int_{|y| > \varepsilon m} e^{-\frac{y^2}{2m} + \delta |y|} dy \le \int_{|y| > \varepsilon m}
e^{-\frac{y^2}{4m}} dy \le \frac{4}{\varepsilon} e^{-\frac{\varepsilon^2 }{4}m}.
\end{equation}

Suppose $v > 0$.
Let $\tilde{A}^m = (A_1+A_2+A_3)/A_4$, where
\begin{align*}
A_1 &= \int_{ |y+vm| > \varepsilon m,  y \le -R} e^{-\frac{y^2}{2m}-W_{n-m}(y)} dy, & A_2&= \int_{-R}^Re^{-\frac{y^2}{2m}-W_{n-m}(y)}dy,
\\
A_3 &= \int_{R}^{\infty} e^{-\frac{y^2}{2m}-W_{n-m}(y)} dy, & A_4&= \int_{-cm}^{cm} e^{-\frac{y^2}{2m}-W_{n-m}(y)}dy.
\end{align*}
For sufficiently large $m$, by~(\ref{eq:positive-v-negative-infinite-slope}),
(\ref{eq:gaussian-integral-estimate}), (\ref{eq:W-finite-slope}) and~(\ref{eq:positive-v-plus-infinity-slope}),
 we have
\begin{align*}
  A_1 &\le \int_{ |y+vm| > \varepsilon m} e^{-\frac{y^2}{2m} - (v+\delta) y} dy
  \\&\le e^{( \frac{v^2}{2} +  v\delta ) m} \int_{|y'| > \varepsilon m} e^{-\frac{y'^2}{2m} + \delta |y'|} dy'
  \le \frac{4}{\varepsilon} \exp \bigl( (v^2/2 + v\delta- \varepsilon^2/4) m  
  \bigr),
\end{align*}
\begin{equation*}
A_2 \le \int_{-R}^R e^{-\frac{y^2}{2m} + R(|y| +1) } dy \le 2Re^{R^2 +R},
\end{equation*}

\begin{equation*}
  A_3  \le \int_R^{\infty} e^{-\frac{y^2}{2m} + (v-2\delta) y} dy
  \le\int_{-\infty}^{\infty} e^{-\frac{y^2}{2m} + (v-2\delta) y} dy = \sqrt{2m\pi} \exp \Bigl(
  \frac{(v-2\delta)^2}{2}m 
  \Bigr),
\end{equation*}
and 
\begin{equation*}
  A_4 \ge \int_{-vm}^{-vm+1} e^{-\frac{y^2}{2m} + (v-\delta)y} dy
  \ge \exp \bigl( (v^2/2 - v\delta) m \bigr).
\end{equation*}
Therefore, 
\begin{multline*}
  \tilde{A}^m  \le \frac{4}{\varepsilon}\exp \bigl( -(\varepsilon^2/4 - 2v\delta)m \bigr)
   \\+ 2Re^{R^2+R}  \exp \bigl(  -(v^2/2-v\delta) m \bigr) + \sqrt{2m\pi} \exp \bigl( -(v\delta - 2\delta^2)m \bigr),
\end{multline*}
and the right-hand side decays exponentially due to (\ref{eq:choice-of-delta}).
\end{proof}

Now we turn to the proof of~Lemma \ref{lem:truncation-and-concentration-almost-surely-true}. We begin with~\eqref{eq:truncation-almost-surely}.

\begin{proof}[Proof of~\eqref{eq:truncation-almost-surely}] Let us fix $(n,q)$ and~$c$. 
Due to the Borel--Cantelli lemma and the fact that for sufficiently large $m$, 
\begin{equation*}
\int_{|y| \ge cm } Z^{n-m,m}(y,q) e^{c|y|/17} dy \le \int_{|y-q| \ge cm/2} Z^{n-m,n}(y,q) e^{c|y-q|/16} dy,
\end{equation*}
the inequality~\eqref{eq:truncation-almost-surely} will follow if we prove that for some constant $k>0$ and sufficiently large $m$, 
  \begin{equation}
    \label{eq:small-probability}
\Pp \bigg\{ \frac{a_m}{b_m}   > 2^{-m } \bigg\} \le e^{-k m},
\end{equation}
where 
 \begin{equation*}
a_m = \int_{|y-q| \ge cm/2}  Z^{n-m,n}(y,q)e^{c|y-q|/16} dy, \quad b_m =  \int_q^{q+1} Z^{n-m,n}(y,q) dy.
\end{equation*}
By Lemma \ref{lem:partition-function-not-too-small}, we have
$\Pp \{b_m \le \rho_0^m\} \le e^{-k_1m}$ for some constant $k_{1}$.
By Markov inequality,
\begin{equation*}
\begin{split}
  \Pp \big\{ a_m \ge (\rho_0/2)^m \big\} \le  \Big(\frac{2}{\rho_0}\Big)^m \E a_m 
  =& \Big(\frac{2\lambda}{\rho_0}\Big)^m \int_{|y-q| \ge cm/2} \frac{1}{\sqrt{2m\pi}} e^{-\frac{(y-q)^2}{2m} + c|y-q|/16}dy \\
  \le& \Big(\frac{2\lambda}{\rho_0}\Big)^m \int_{|y-q| \ge cm/2} \frac{1}{\sqrt{2m\pi}} e^{-\frac{(y-q)^2}{4m}}dy \\
  \le& \frac{8}{c \sqrt{2m\pi}} e^{-\big( c^2/16- \ln (2 \lambda/\rho_0) \big) m} dy   \le e^{-k_2m}
\end{split}
\end{equation*}
for a constant $k_2>0$ if $c > 4 \sqrt{\ln(2\lambda/\rho_0)}$ and $m$ is sufficiently large.
Combining these two inequalities, we obtain (\ref{eq:small-probability}) and complete the proof of~\eqref{eq:truncation-almost-surely}.
\end{proof}

To prove~\eqref{eq:concentration-almost-surely}, the second part of Lemma~\ref{lem:truncation-and-concentration-almost-surely-true},
we need auxiliary statements.

\begin{lemma}
  \label{lem:continuity-of-partition-function-wrt-end-point} There is a number $K$ such that,
 with probability one, for any
  $c > 0$ and $(n,q) \in \Z \times \Z$, and
   for sufficiently large $m$, 
\begin{equation}
  \label{eq:continuity-of-partition-functions}
  \big| \ln Z^{n-m,n}(y, x) - \ln Z^{n-m,n}([y],[x]) \big| \le K \sqrt{m},\quad [x]=q,\quad |[y]| \le cm. 
\end{equation}
\end{lemma}

\begin{proof}
Without loss of generality we can assume $n = 0$. Let us define
\[
C_{x,y}^m=\bigg\{\mu_{y,x}^{-m,0} \{ \gamma: 
  |\gamma_{-m+1} - [y]| \vee |\gamma_{-1} - [x]| \ge (R'+1) \sqrt{m} \} \le 2^{-m+1}\bigg\}. 
\]
Due to Lemma \ref{lem:small-probability-for-single-square-root-jump} and $\mu_{p, q}^{-m, 0}
\preceq \mu_{y,x}^{-m,0} \preceq \mu_{p+1,q}^{-m,0}$ for $[x]=q$, $[y]=p$, we have 
\[
  \Pp \Big({\textstyle \bigcup_{[x]=p,\ [y]=q}}\, (C_{x,y}^m)^c \Big) \le 2e^{-r_2m}
\]
for sufficiently large $m$ and all $|p| \le cm$.
Then, the Borel--Cantelli  lemma implies that, with probability one, there is $m_1$ such that for $m>m_1$,
 all $[x]=q$ and $|[y]| \le cm$, $C_{x,y}^m$ holds.

Also, using~\ref{item:exponential-moment-for-maximum},  Markov's inequality, and the Borel--Cantelli lemma, we obtain that 
  with probability one, there is $m_2$ such that for $m>m_2$,
  \begin{equation}
  \label{eq:not-large-potential-almost-surely}
|F_{-m}(y) | \le \sqrt{m}, \quad |[y]| \le cm.
\end{equation}
Assume $m \ge 2\vee m_1\vee m_2$.
We consider only~$x$ and~$y$ such that $[x] = q$ and $|[y]| \le cm$.
Let us fix $p=[y]$ and define 
\begin{equation}\label{eq:definition-of-Z-tilda-minus-m-0-y-x}
  \tilde{Z}^{-m,0}_{y,x} = \int_{(z_1,z_2) \in E}\,dz_1dz_2
  e^{-\frac{1}{2}(y-z_1)^2 - F_{-m}(y) - \frac{1}{2}(x-z_2)^2 - F_{-1}(z_2)}
  Z^{-m+1,-1}_{z_1,z_2}
\end{equation}
where $E = \{(z_1,z_2)\, : \, |z_1 - p|\vee | z_2 -q| < (R'+1)\sqrt{m}\}$.
Since $m>m_1$, we have $  \bigl| 1 - \tilde{Z}^{-m,0}_{y,x}/Z^{-m,0}_{y,x}  \bigr| \le 2^{-m+1}$
and hence
\begin{equation}
\label{eq:diff-of-Z-tilde-and-Z}
\big|  \ln Z^{-m,0}_{y,x}  -  \ln \tilde{Z}^{-m,0}_{y,x} \big| \le 2^{-m+2}, 
\end{equation}
since $|\ln (1+x)| \le 2|x|$ for $|x| \le 1/2$.
Using~(\ref{eq:not-large-potential-almost-surely}) and~(\ref{eq:definition-of-Z-tilda-minus-m-0-y-x}), we see that
\begin{equation}
  \label{eq:continuity-of-Z-tilde}
\begin{split}
  &\big| \ln \tilde{Z}^{-m,0}_{y,x} - \ln \tilde{Z}^{-m,0}_{p,q} \big|\\
  \le &2 \sqrt{m} + \sup_{(z_1, z_2) \in E}\bigg( \biggl| \frac{(y-z_1)^2}{2} - \frac{(p-z_1)^2}{2} \biggr| + \biggl|
  \frac{(x-z_2)^2}{2} - \frac{(q-z_2)^2}{2} \biggr|  \bigg)\\
  \le & 2\sqrt{m} + 2  \bigl( (R'+1) \sqrt{m} + 1 \bigr).
\end{split}
\end{equation}
The lemma follows from~(\ref{eq:diff-of-Z-tilde-and-Z}), (\ref{eq:continuity-of-Z-tilde}) and the triangle inequality.
\end{proof}

\begin{lemma}
\label{lem:concentration-simultaneouly-for-one-point}
 With probability one, for all $c > 0$ and $(n,q) \in \Z \times \Z$, we have
\begin{equation*}
|\ln Z^{n-m,n}(y,x) - \alpha(m,x-y)| \le 2 m^{3/4},\quad [x]=q,\quad |[y]|\le cm,
\end{equation*}
for sufficiently large $m$.
\end{lemma}

\begin{proof} Let us fix $c$ and  $(n,q)$.
  By Theorem \ref{thm:concentration-of-free-energy} and shear invariance, for sufficiently large $m$
and all~$p \in \Z$, we have
\begin{equation*}
  \Pp \Big\{ \big| \ln Z^{n-m,n}(p,q) - \alpha(m,q-p) \big| \ge m^{3/4} \Big\} \le e^{-m^{1/3}},
\end{equation*}
The Borel--Cantelli lemma implies that almost surely,
\begin{equation}
  \label{eq:concentration-for-partition-function-almost-surely}
 \big| \ln Z^{n-m,n}(p,q) - \alpha(m,q-p) \big| \le m^{3/4},
\end{equation}
for sufficiently large $m$ and all $|p| \le cm$.
Also, if $[x]=q$ and $|[y]| \le cm$, we have
\begin{equation}
  \label{eq:Lip-continuous-for-alpha} 
\begin{split}
  | \alpha(m,[x]-[y]) - \alpha(m,x - y) |
  \le &\biggl| \frac{([x]-[y])^2}{2m} - \frac{(x-y)^2}{2m} \biggr| \\
\le &  \frac{2(|q|+1) + 2 (cm+1)}{m}.
\end{split}
\end{equation}
The lemma follows from Lemma \ref{lem:continuity-of-partition-function-wrt-end-point},
(\ref{eq:concentration-for-partition-function-almost-surely}) and (\ref{eq:Lip-continuous-for-alpha}).
\end{proof}

Now, the inequality~\eqref{eq:concentration-almost-surely}
follows from Lemma~\ref{lem:concentration-simultaneouly-for-one-point}, which completes the proof of 
Lemma~\ref{lem:truncation-and-concentration-almost-surely-true} and concludes the entire uniqueness part.

\subsection{Basins of pullback attraction}\label{sec:pullback}

The global solutions play the role of one-point pullback attractors.
The goal of this section is to prove Theorem \ref{thm:pullback_attraction}.

First we need a version of Lemma \ref{lem:condition-second-kind-polymer-measure-converge-to-LLN}
where $w_N\equiv w$ are independent of $N$, which is the case in Theorem \ref{thm:pullback_attraction}.

\begin{lemma}
  \label{lem:LLN-for-p2l-polymer-measure-non-random}
  Let $v \in \R$ and $w(\cdot) \in \mathbb{H}'$.
  If one of the conditions \eqref{eq:no_flux_from_infinity}, \eqref{eq:flux_from_the_left_wins}, \eqref{eq:flux_from_the_right_wins}
 is satisfied by $W(\cdot)= \int_0^{\cdot} w(y')dy'$,
then for almost every $\omega$ and every $n \in \Z$, the probability measures $\nu_{N,n,x} (N < n)$ defined by
\begin{equation*}
\nu_{N,n,x}(dy) = \frac{  Z^{N,n}(y,x) e^{-W(y)}}{\int_{\R} 
  Z^{N,n}(y',x) e^{-W(y')}dy'} dy
\end{equation*}
satisfy 
\begin{equation*}
\lim_{N\to-\infty} e^{h|N|}\sup_{x \in [-L,L]} \nu_{N,n,x}( [(v+\varepsilon)N, (v-\varepsilon)N]^c ) = 0,
\end{equation*}
for all $L\in \N$ and  $\varepsilon > 0$, and some constant $h(\varepsilon) > 0$ depending on $\varepsilon$.
\end{lemma}

\begin{proof}
  The proof is similar to that of Lemma~\ref{lem:condition-second-kind-polymer-measure-converge-to-LLN}.
  Because $w_N(\cdot) \equiv w(\cdot)$ are independent of $N$, there is no need to choose a subsequence
  $(m_k)$ to satisfy \eqref{eq:W-asymtopic-slope-at-infinity}, \eqref{eq:boundedness-of-W}, and
  \eqref{eq:W-finite-slope} as we did in the first step of proving Lemma~\ref{lem:condition-second-kind-polymer-measure-converge-to-LLN}.
Therefore,  we obtain $\lim$ instead of $\liminf$ in the conclusion.
\end{proof}

\begin{proof}[Proof of Theorem \ref{thm:pullback_attraction}] We define $\hat\Omega=\bar{\Omega} \cap \Omega'' \cap
  \Omega_v'$ and let $\omega\in\hat\Omega$.  We also define $V(x) = e^{-\int_0^{x } w(x') dx'}$
  and consider the measures
  \begin{equation*}
\nu_{N,n,x}(dy) = \bar{\mu}_{V, x}^{N,n}\pi_N^{-1}(dy) = \frac{V(y) Z^{N,n}(y,x)}{\int_{\R} V(y')
  Z^{N,n}(y',x)dy'} dy.
 \end{equation*}
 Then we have $\bar{\mu}_{V,x}^{N,n} = \mu_{\nu_{N,n,x},x}^{N,n}$
 and 
\begin{equation*}
\Psi^{N,n}_{\omega} w(x) = \int_{\R}  (y-x)   \bar{\mu}_{V, x}^{N,n}\pi^{-1}_{n-1} (dy).
\end{equation*} 
Due to Lemma~\ref{lem:monotonicity_of_x-ux}, it suffices to prove pointwise convergence, i.e., to show that
\begin{equation}
\label{eq:what-to-prove-pullback}
\lim_{N\to -\infty} \int_{\R} (y-x) \bar{\mu}_{V, x}^{N,n}\pi^{-1}_{n-1} (dy) = \int_{\R} (y-x) \mu_{x}^{-\infty,n}(v)\pi_{n-1}^{-1}(dy),\quad x\in\R.
\end{equation}
Using Lemmas~\ref{lem:LLN-for-p2l-polymer-measure-non-random} and~\ref{lem:small-first-step-almost-surely-all-range}, we
obtain that for some constants $b_1$ and $b_{2}$, 
  \begin{equation}
    \label{eq:growth-condition-for-g}
   \bar{\mu}_{V, x}^{N,n}\pi^{-1}_{n-1}([-L,L]^c)   \le b_1 e^{-b_2 \sqrt{L}}.
  \end{equation}
By Lemma \ref{lem:LLN-for-p2l-polymer-measure-non-random}, for fixed $(n,x) \in
\Z\times\R$, $(\nu_{N,n,x})_{N<n}$ is a family of probability measures satisfying LLN with slope
$v$.
Hence by Lemma \ref{lem:uniqueness-of-polymer-measures}, $\mu_{\nu_{N,n,x},x}^{N,n}$ converges weakly
to $\mu_x^n(v)$, so $\bar{\mu}_{V,x}^{N,n} \pi_{n-1}^{-1}$ converges weakly to  $\mu_{x}^n(v)\pi_{n-1}^{-1}$.
Now (\ref{eq:what-to-prove-pullback}) follows from this and~(\ref{eq:growth-condition-for-g}), and the proof is complete.
\end{proof}
\section{Overlap of polymer measures}
\label{sec:overlap}

In this section we prove Theorem~\ref{thm:convergence-total-variation}. In fact, we give a proof for the time-reversed version
of the theorem because this is more convenient. We recall that 
\begin{equation*}
\| \mu - \nu \|_{TV} = \sup_{A \in \mathcal{B}(\R)} |\mu(A) - \nu(A)|.
\end{equation*}
and that $\Omega_v'=\Omega'\cap \Omega_v$.

The convergence of polymer measures in total variation distance is a consequence 
of the existence of ratios of partition functions and the LLN for polymer measures.

For the rest of this section, we fix $v \in \R$ and always assume that $\omega\in\Omega'_v$.
We will also fix $(n_1,x_1)$ and $(n_2,x_2)$, and write~$\mu_i^N = \mu_{x_i}^{n_i} \pi^{-1}_N$, $i=1,2$.

\begin{lemma}
\label{lem:TV-distance-controlled-by-ratio-of-density}
Let $\mu$ and $\nu$ be two probability measures with densities $f$ and $g$ respectively, such that
both $f$ and $g$ are positive on some Borel set $C$, and zero outside~$C$.
Then 
\begin{equation*}
\| \mu - \nu \|_{TV} \le 1- \inf_{x\in C} \frac{g(x)}{f(x)} 
\end{equation*}
\end{lemma}
\begin{proof}
Let $A = \{ x \in C: f(x) \ge g(x)\}$ and $d = \inf_{x\in C} g(x)/f(x)$.
Then 
\begin{equation*}
  \| \mu - \nu \|_{TV} = \int_A (f(x)-g(x)) \,dx \le \int_A (1-d)f(x)\, dx
  \le (1-d)\int_C f(x) \, dx = 1-d.
\end{equation*}
\end{proof}

\begin{lemma}
\label{lem:sublinear-growth-of-polymer-meausres}
There are constants $\alpha_N, \beta_N$ depending on $\omega, x_i, n_i$ such that
\begin{align*}
\lim_{N\to -\infty}
\frac{\alpha_N}{N}&=\lim_{N\to -\infty} \frac{\beta_N}{N}=v, \\
\lim_{N\to -\infty}\mu_1^N([\alpha_N, \beta_{N}]^c) &= \lim_{N\to -\infty}\mu_2^N([\alpha_N, \beta_{N}]^c) =0.
\end{align*}
\end{lemma}
\begin{proof}
Since the measures $\mu_i^{N}$ satisfy the LLN with slope $v$, there is  a decreasing sequence of negative numbers  $(N_k)$ such that
\begin{equation*}
\mu_i^N ([   (-v-2^{-k})|N|, (-v+2^{-k})|N|]^c) \le 2^{-k}, \quad N \le N_k,\ i=1,2.
\end{equation*}
For every $N$, let $k$ be such that $N_{k+1} \le N < N_k$.
Then setting
\begin{equation*}
\alpha_N = (-v-2^{-k})|N|, \quad \beta_N = (-v+2^{-k})|N|
\end{equation*}
completes the proof.
\end{proof}

Let $f_i^N(x)$ be the density of $\mu_i^N$.
We will need the following representation of~$f_i^N$.
\begin{lemma}
\label{lem:representation-through-solution-to-burgers-equation}
Recall the function $V_v(N,x)$ which is the Hopf--Cole transform of the global solution $u_v(n,x)$.
Then 
\begin{equation*}
f_i^N(x) = \frac{ V_v(N,x) Z_{x,x_i}^{N,n_i}}{\int V_v(N,x') Z_{x',x_i}^{N,n_i} \, dx'}.
\end{equation*}
\end{lemma}
\begin{proof}
  By (\ref{eq:expression-of-polymer-measure-density}) in Theorem \ref{thm:convergence-of-density-functions} we have 
\begin{equation*}
f_i^N(x) = Z^{N,n_i}_{x,x_i} G_v \bigl( (N,x), (n_i,x_i) \bigr).
\end{equation*}
Thus for $x \neq y$, 
\begin{equation*}
\frac{f_i^N(x)}{f_i^N(y)} = \frac{Z^{N,n_i}_{x,x_i}}{Z^{N,n_i}_{y,x_i}} \frac{G_v \bigl( (N,x),
  (n_i,x_i) \bigr)}{G_v \bigl(  (N,y), (n_i,x_i) \bigr)}
= \frac{Z^{N,n_i}_{x,x_i}}{Z^{N,n_i}_{y,x_i}} \frac{G_v \bigl( (N,x), (N,0)  \bigr)}{ G_v \bigl(
  (N,y), (N,0) \bigr)}
= \frac{Z^{N,n_i}_{x,x_i} V_v(N,x)}{Z^{N,n_i}_{y,x_i}V_v(N,y)},
\end{equation*}
and our claim follows.
\end{proof}

\begin{proof}[Proof of Theorem~\ref{thm:convergence-total-variation}]
Let $D_i^N = \mu_i^N([\alpha_N,\beta_N])$, $i=1,2$, and let 
\begin{equation*}
\tilde{\mu}_i^N (A) = (D_i^{N} )^{-1} \mu_i^N( A \cap [\alpha_N, \beta_N]).
\end{equation*}
Then the measures $\tilde{\mu}_i^N$, $i=1,2$, are supported on $[\alpha_N,\beta_N]$ with densities given by $\tilde{f}_i^N(x)=(D_i^N)^{-1} f_i^N(x)$. Also, 
\begin{equation}
  \label{eq:modified-measures}
\|  \tilde{\mu}_i^N - \mu_i^N \|_{TV} \le 1-D_i^{N}.
\end{equation}
Combining this with Lemma \ref{lem:TV-distance-controlled-by-ratio-of-density} we obtain 
\begin{equation*}
\begin{split}
  \| \mu_1^N - \mu_2^N \|_{TV}
  &\le \| \mu_1^N - \tilde{\mu}_1^N \|_{TV} +
  \| \tilde{\mu}_1^N - \tilde{\mu}_2^N \|_{TV} +
  \| \tilde{\mu}_2^N - \mu_2^N \|_{TV} \\
  &\le 3-D^N_1-D^N_2 - \inf_{x \in [\alpha_N, \beta_N]} \frac{\tilde{f}_2^N(x)}{\tilde{f}_1^N(x)}.
\end{split}
\end{equation*}
Since $D^{N}_i \to 1$ as $N\to -\infty$, $i=1,2$, it suffices to show 
\begin{equation*}
\lim_{N \to -\infty} \inf_{x \in [\alpha_N, \beta_N]} \frac{\tilde{f}_2^N(x)}{\tilde{f}_1^N(x)} = 1.
\end{equation*}
Using the representation of $f^N_i$ in Lemma
\ref{lem:representation-through-solution-to-burgers-equation}, we see that
\begin{equation*}
\tilde{f}_i^N = \frac{ V_v(N,x) Z_{x,x_i}^{N,n_i}}{\int_{\alpha_N}^{\beta_{N}} V_v(N,x') Z_{x',x_i}^{N,n_i} \, dx'}.
\end{equation*}
and hence
\begin{equation*}
  \frac{\tilde{f}_2^N(x)}{\tilde{f}_1^N(x)} = \frac{Z^{N,n_2}_{x,x_2}}{Z^{N,n_1}_{x,x_1}}
  \frac{\int_{\alpha_N}^{\beta_{N}} V_v(N,x') Z_{x',x_1}^{N,n_1} \, dx'}{\int_{\alpha_N}^{\beta_{N}}
    V_v(N,x') Z_{x',x_2}^{N,n_2} \,dx'}
  \ge \frac{m_N}{M_{N}},
\end{equation*}
where 
\begin{equation*}
  m_N = \inf_{x \in [\alpha_N, \beta_N]} \frac{Z^{N,n_2}_{x,x_2}}{Z^{N,n_1}_{x,x_1}}, \quad
  M_n = \sup_{x \in [\alpha_N, \beta_N]} \frac{Z^{N,n_2}_{x,x_2}}{Z^{N,n_1}_{x,x_1}}.
\end{equation*}
Our goal is to show that $\lim_{N\to \infty} m_N/M_N = 1$.

Since the partition function is continuous with respect to endpoints, both the supremum and infimum
are achieved at some points $x = x^N_+$ and $x=x_-^{N}$.
Since $\lim_{N\to -\infty} x^N_{\pm}/N=v$, Theorem
\ref{thm:convergence-of-partition-function-ratios}  implies 
\begin{equation*}
  \lim_{N\to -\infty} m_N = \lim_{N\to -\infty} M_N
  = \lim_{N\to -\infty} \frac{Z^{N,n_2}_{x^N_{\pm}, x_2}}{Z^{N,n_1}_{x^N_{\pm}, x_1}}=
  G_v \bigl( (n_2,x_2), (n_1,x_1) \bigr)\in(0,\infty),
\end{equation*}
so $\lim_{N\to -\infty} m_N/M_N =1$. This completes the proof.
\end{proof}

\section{Proofs of lemmas from Section~\ref{sec:setting}}
\label{sec:aux}
\begin{proof}[Proof of Lemma~\ref{lem:invariant_spaces}]
 Let us check that if $W\in\HH$, then $\Phi^{n,n+1}_\omega W\in\HH$ for all $n$ and $\omega$.
Due to~\eqref{eq:forcing-averages-to-0}, there is a number $k=k(n,\omega)>0$ such that $F_n(x)+W(x)\ge -k(|x|+1)$
for all $x\in\R$. Since
\begin{align*}
 \int_\R g_{2\visc}(y-x)e^{-\frac{F_n(x)}{2\visc}-\frac{W(x)}{2\visc}}dx\le 
\int_\R g_{2\visc}(y-x)e^{\frac{k|x|+1}{2\visc}}dx<\infty,
\end{align*}
$\Phi^{n,n+1}_\omega W(y)$ is well-defined for all $y\in\R$, and
\begin{align*}
\liminf_{y\to+\infty}\frac{\Phi^{n,n+1}_\omega W(y)}{y}
&\ge -\liminf_{y\to+\infty}\frac{2\visc}{y}\ln \int_\R g_{2\visc}(y-x)e^{\frac{k|x|+1}{2\visc}}dx\\
&= -\liminf_{y\to+\infty}\frac{2\visc}{y}\ln \int_\R g_{2\visc}(y-x)e^{\frac{kx+1}{2\visc}}dx\\
&= -\liminf_{y\to+\infty}\frac{2\visc}{y}\ln (e^{\frac{ky}{2\visc}+\frac{k^2}{4\visc}+\frac{1}{2\visc}})=-k >-\infty.
\end{align*}
In the second line, we used that the contribution from the negative values of $x$ is asymptotically 
negligible due to the fast decay of the Gaussian kernel. 
For the last line, we used the Gaussian moment generating function. The behavior as $y\to-\infty$ is treated similarly. 
The local Lipschitz property follows from the $C^1$ property that can be obtained by differentiating the integrand in the definition of $\Phi$.

Iterating this, we obtain parts~\ref{it:inv-of-HH} and \ref{it:cocycle} of the lemma.
The proof of part~\ref{it:inv-of-HH-vv} is similar to that of part~\ref{it:inv-of-HH}.
\end{proof}

\begin{proof}[Proof of Lemma~\ref{lem:monotonicity_of_x-ux}]
  Let $V(\cdot)$ be the Hopf--Cole transform of $w(\cdot)$.
  Then
\begin{equation*}
  x-\Psi^{n_0,n_1}w(x) = x - \int_{\R} (x-y) \bar{\mu}_{V, x}^{m,n}(dy) = \int_{\R} y \bar{\mu}_{V,x}^{m,n}(dy).
  \end{equation*}
The conclusion then follows from Lemma~\ref{lem:monotonicity-for-second-kind-polymer-measure}.
\end{proof}

\bibliographystyle{alpha}
\bibliography{Burgers}
\end{document}